\documentclass[a4paper, 11pt]{amsart}

\usepackage[format=hang,font=small,labelfont=bf]{caption}

\usepackage[utf8]{inputenc}
\usepackage[english]{babel}

\usepackage[OT2,T1]{fontenc} 

\usepackage{etoolbox}

\usepackage{amsfonts}
\usepackage{amssymb}
\usepackage{amsthm}
\usepackage{hyperref}

\usepackage{mathrsfs}
\usepackage{mathtools}
\usepackage{enumerate}

\usepackage[shortlabels]{enumitem}

\usepackage[all]{xy}

\usepackage{fullpage}
\usepackage{xcolor}

\usepackage{wrapfig}

\usepackage{verbatim}
\usepackage{parskip}

\usepackage{pinlabel}
\usepackage[toc, page]{appendix}

\usepackage{tkz-graph}
\usetikzlibrary{arrows}
\usetikzlibrary{calc}
\usepackage{tikz-cd}

\usepackage{multicol}
\usepackage{cancel}

\usepackage{epigraph}
\usepackage{transparent}
\usepackage{extarrows}

\usepackage{stmaryrd}

\usepackage{calligra}

\usepackage{multirow}
\usepackage{booktabs}

\usepackage{thm-restate}

\usepackage[colorinlistoftodos]{todonotes}

\usepackage{float}
\usepackage{aliascnt}
\newaliascnt{eqfloat}{equation}
\newfloat{eqfloat}{h}{eqflts}
\floatname{eqfloat}{Equation}

\newcommand*{\ORGeqfloat}{}
\let\ORGeqfloat\eqfloat
\def\eqfloat{%
	\let\ORIGINALcaption\caption
	\def\caption{%
		\addtocounter{equation}{-1}%
		\ORIGINALcaption
	}%
	\ORGeqfloat
}

\newtheorem{thm}{Theorem}[section]
\newtheorem{theorem}{Theorem}
\newtheorem*{theorem*}{Theorem}
\newtheorem{mainTheorem}{Theorem}

\newtheorem{cor}[thm]{Corollary}

\newtheorem{lemma}[thm]{Lemma}
\newtheorem{prop}[thm]{Proposition}

\theoremstyle{rmk}
\newtheorem{remark_basic}[thm]{Remark}
\newtheorem{observation_basic}[thm]{Observation}
\newtheorem{notation_basic}[thm]{Notation}

\newtheorem{terminology_basic}[thm]{Terminology}

\theoremstyle{definition}
\newtheorem{definition_basic}[thm]{Definition}	
\newtheorem{example_basic}[thm]{Example}

\newenvironment{rmk}
{\pushQED{\qed}
	
	\begin{remark_basic}}
	{\popQED
\end{remark_basic}}

\newtheorem{asspt}{\textbf{Assumption}}

\newtheorem*{setting}{Setting}

\newtheorem{quests}{Questions}
\newtheorem{quest}{Question}

\newtheorem{axiom}{\textbf{Axiom}}

\newenvironment{definition}
{\pushQED{\qed}
	
	\begin{definition_basic}}
	{\popQED
\end{definition_basic}}

\newenvironment{example}
{\pushQED{\qed}
	
	\begin{example_basic}}
	{\popQED
\end{example_basic}}

\DeclareMathAlphabet{\mathcal}{OMS}{cmsy}{m}{n}

\newcommand{\R}{\mathbb{R}}
\newcommand{\PP}{\mathbb{P}}
\newcommand{\C}{\mathbb{C}}

\newcommand{\HH}{\mathbb{H}}
\newcommand{\A}{\mathbb{A}}

\newcommand{\Z}{\mathbb{Z}}
\newcommand{\N}{\mathbb{N}}

\newcommand{\Zmod}{\mathbb{Z}/2\mathbb{Z}}

\newcommand{\CP}{\mathbb{C}\mathbb{P}^n}
\newcommand{\RP}{\mathbb{R}\mathbb{P}^n}
\newcommand{\HP}{\mathbb{H}\mathbb{P}^n}

\newcommand{\APn}{\mathbb{A}\mathbb{P}^n}

\newcommand{\CPone}{\mathbb{C}\mathbb{P}^1}

\newcommand{\M}{\mathcal{M}}
\renewcommand{\L}{\mathcal{L}}

\newcommand{\CC}{\mathcal{C}}
\newcommand{\J}{\mathcal{J}}

\newcommand{\F}{\mathcal{F}}

\newcommand{\eps}{\varepsilon}

\newcommand{\Soo}{\mathrm{SO(2)}}

\newcommand{\rk}{\mathrm{rank}}

\newcommand{\Aut}{\mathrm{Aut}}
\newcommand{\Auteq}{\mathrm{Auteq}}
\newcommand{\Ham}{\mathrm{Ham}}
\newcommand{\Symp}{\mathrm{Symp}}
\newcommand{\Diff}{\mathrm{Diff}}

\newcommand{\Crit}{\mathrm{Crit}}
\newcommand{\HF}{\operatorname{HF}}
\newcommand{\CF}{\operatorname{CF}}

\newcommand{\Fuk}{\operatorname{\mathcal{F}uk}}

\newcommand{\rank}{\operatorname{rank}}
\newcommand{\Ext}{\mathrm{Ext}}

\renewcommand{\leq}{\leqslant}
\renewcommand{\geq}{\geqslant}

\newcommand{\lra}{\longrightarrow}

\newcommand{\re}{\mathrm{Re}}
\newcommand{\im}{\mathrm{Im}}
\renewcommand{\ker}{\mathrm{ker}}
\newcommand{\coker}{\mathrm{coker}}

\renewcommand{\Im}{\operatorname{\mathrm{Im}}}

\makeatletter

\newcommand{\Rmnum}[1]{\expandafter\@slowromancap\romannumeral #1@}
\makeatother

\newcommand{\del}{\partial}

\newcommand{\sheafHom}{\operatorname{\mathscr{H}\text{\kern -3pt {\calligra\large om}}\,}}

\newcommand{\sheafEnd}{\operatorname{\mathscr{E}\text{\kern -3pt {\calligra\large nd}}\,\,\,}}
\newcommand{\sheafEndmon}{\sheafEnd_{\text{\kern -3pt {\calligra mon}}\,}}

\renewcommand{\Im}{\operatorname{Im}} 

\newcommand*{\Jhat}[1]{#1\kern-0.37em\hat{\phantom{#1}}}

\title{Projective twists and the Hopf correspondence}
\author{Brunella Charlotte Torricelli}
\address{Brunella Torricelli, Centre for Mathematical Sciences, University of Cambridge, 
	CB3 0WB, UK}
\email{bct27@cam.ac.uk}
\begin{document}
	\maketitle

	\begin{abstract}
Given Lagrangian (real, complex) projective spaces $K_1, \dots , K_m$ in a Liouville manifold $(X, \omega)$ satisfying a certain cohomological condition, we show there is a Lagrangian correspondence (in the sense of \cite{wwcorr}) that assigns a Lagrangian sphere $L_i \subset K$ of another Liouville manifold $(Y, \Omega)$ to any given projective Lagrangian $K_i \subset X$, $i=1, \dots m$.

	We use the Hopf correspondence to study \emph{projective twists}, a class of symplectomorphisms akin to Dehn twists, but defined starting from Lagrangian projective spaces. When this correspondence can be established, we show that it intertwines the autoequivalences of the compact Fukaya category $\Fuk(X)$ induced by the projective twists $\tau_{K_i} \in \pi_0(\Symp_{ct}(X))$ with the autoequivalences of $\Fuk(Y)$ induced by the Dehn twists $\tau_{L_i} \in \pi_{0}(\Symp_{ct}(Y))$, for $i=1, \dots m$. 

	Using the Hopf correspondence, we obtain a free generation result for projective twists in a \emph{clean plumbing} of projective spaces and various results about products of positive powers of Dehn/projective twists in Liouville manifolds. 

	The same techniques are also used to show that the Hamiltonian isotopy class of the projective twist (along the zero section in $T^*\CP$) in $\Symp_{ct}(T^*\CP)$ does depend on a choice of framing, for $n\geq19$. Another application of the Hopf correspondence delivers smooth homotopy complex projective spaces $K\simeq \CP$, that do not admit Lagrangian embeddings into 
	$(T^*\CP, d\lambda_{\CP})$, $n=4,7$.
	\end{abstract}

\section{Introduction}

\subsection{Questions}

Given a symplectic manifold $(M, \omega)$ with contact boundary, a natural object of study is the group $\Symp_{ct}(M)$ of compactly supported symplectomorphisms that are the identity in a neighbourhood of the boundary. Its quotient $\pi_0(\Symp_{ct}(M))$ by the relation of symplectic isotopy is the \emph{symplectic mapping class group}, and is already a highly non-trivial object. When $H^1(M;\R)=0$, a symplectic isotopy is automatically Hamiltonian, and $\pi_0(\Symp_{ct}(M))$ coincides with the quotient $\Symp_{ct}(M)/\Ham_{ct}(M)$ by the subgroup $\Ham_{ct}(M)\subset \Symp_{ct}(M)$ of (compactly supported) Hamiltonian symplectomorphisms (namely time-1 maps of compactly supported Hamiltonian flows).

The symplectic mapping class group carries a (forgetful) comparison map \begin{align}\label{comparisonmap}
c:\pi_0(\Symp_{ct}(M)) \longrightarrow \pi_0(\Diff_{ct}^+(M))
\end{align} to the (compactly supported and orientation-preserving) smooth mapping class group of $M$. In general, the map is neither injective nor surjective. Its kernel is of particular interest as it captures phenomena which are exclusively symplectic and not visible in the smooth mapping class group.
The question of whether a symplectomorphism $\varphi \in \Symp_{ct}(M)$ is a non-trivial element of the kernel of $c$ (i.e is smoothly isotopic to the identity but not symplectically so) is called the \emph{symplectic isotopy problem}.

In dimension two, the kernel of $c$ is always trivial, and the symplectic mapping class group is isomorphic to the smooth mapping class group $\pi_0(\Diff_{ct}^+(M))$; this is a consequence of \emph{Moser's argument} (\cite{moser}).

\emph{Dehn twists} often provide examples of nontrivial symplectomorphisms that lie in the kernel of $\eqref{comparisonmap}$. Given a sphere $L$ (and a choice of parametrisation, called a \emph{framing}, see Definition \ref{definitionframing}), the periodicity of the (co)geodesic flow can be used to construct a compactly supported symplectomorphism of the cotangent bundle $\tau_L \in \Symp_{ct}(T^*L)$ (see Definition \ref{weinsteinextend}), called a standard Dehn twist.

The standard Dehn twist has infinite symplectic order, i.e infinite order in $\pi_0(\Symp_{ct}(T^*S^n))$ (\cite{seidelgraded}), and for $n=2$, it generates the entire mapping class group $\pi_0(\Symp_{ct}(T^*S^2))$ (\cite{seidelgeneration}).

Given a general symplectic manifold $(M, \omega)$ and a Lagrangian sphere $L\subset M$, the local construction of the standard Dehn twist can be implanted in a neighbourhood of $L$ via Weinstein's neighbourhood theorem, to yield a compactly supported symplectomorphism $\tau_L\in \Symp_{ct}(M)$. When $\dim(L)$ is even, the Dehn twist has finite order in $\Diff_{ct}^+(M)$ but often has infinite order in $\Symp_{ct}(M)$.
Seidel's early investigations provided the first global examples of (symplectically) non-trivial Dehn twists, in particular non-trivial elements of the kernel of the comparison map \eqref{comparisonmap}. For example, for a $K3$-surface $(M,\omega)$ containing two disjoint Lagrangian spheres $L_1,L_2 \subset M$, the class of $\tau_{L_1}$ has infinite order in $\pi_0(\Symp_{ct}(M))$, and hence in that case $c$ has infinite kernel (\cite{seidelgraded}).
Other important examples in which the kernel of $c$ is large comprise Dehn twists in Milnor fibres of any isolated hypersurface singularity (\cite{keatingfree}) and Dehn twists in projective hypersurfaces of degree $d>2$ (and more general divisors, \cite{tonkonog}).

One of the widely employed methodologies used in these investigations is symplectic Picard--Lefschetz theory. In this context, Dehn twists are regarded as the class of symplectomorphisms that encodes symplectic monodromy maps associated to nodal degenerations, i.e monodromies of \emph{Lefschetz fibrations} (see Section \ref{sectiontwist}).

For an exact symplectic manifold $(M,\omega)$, any Dehn twist $\tau_L$ along a Lagrangian sphere $L\subset (M, \omega)$ can be realised as the local monodromy of an exact Lefschetz fibration (with exact smooth fibre $(M, \omega)$ and exact base). One important result that has been proved recently in \cite{bgzfilling} (an alternative proof of which can be found in this paper) is that the global monodromy of such Lefschetz fibrations can never be isotopic to the identity in the symplectic mapping class group.

\begin{restatable}[{{\cite[Theorem 1.4]{bgzfilling}}}]{theorem}{thmbgz}
	\label{bgz}
	Let $(M, \omega)$ be a Liouville manifold, and let $L_1, \dots , L_m \subset M$ be Lagrangian spheres. Let $\phi= \prod_{i=1}^{k} \tau_{L_{j_i}} \in \Symp_{ct}(M)$, $j_i \in \{ 1, \dots , m \}$ be a positive word of Dehn twists. Then $\phi$ is not compactly supported isotopic to the identity in $\Symp_{ct}(M)$.
\end{restatable}

As a result, Dehn twists represent an extremely important source of (symplectically) non-trivial symplectic automorphisms of exact symplectic manifolds. 

In a more general setting, we can consider both positive as well as negative powers of Dehn twists. In this case, it becomes necessary to have a way of measuring the intersections of the Lagrangians generating the twists.
For example, if $L,L'$ are two Lagrangian spheres of a Liouville manifold $(M, \omega)$ which intersect in a single point, then the corresponding twists $\tau_L, \tau_{L'} \in \Symp_{ct}(M)$ satify the \emph{braid relation}, an isotopy $\tau_L\tau_{L'}\tau_L \simeq \tau_{L'}\tau_L\tau_{L'}$ in $\Symp_{ct}(M)$ (\cite{seidelknotted, stbraid}).
In a general situation, Keating showed that the suitable quantifier that obstructs the possibility of a non-trivial relation between the twists $\tau_L, \tau_{L'}$ is the rank of the Floer cohomology group $\HF(L,L')$, as follows.

\begin{thm}[{{\cite[Theorem 1.1 and 1.2]{keatingfree}}}]
	\label{keatingfree}
	Let $(M, \omega)$ be a Liouville manifold of dimension greater than $2$, and $L, L' \subset M$ be two Lagrangian spheres satisfying $\rank \HF(L,L')\geq 2$, and such that $L, L'$ are not quasi-isomorphic in the Fukaya category. Then the Dehn twists $\tau_L, \tau_{L'}$ generate a free subgroup of $\pi_0(\Symp_{ct}(M))$, and the associated functors $T_L, T_L'$ generate a free subgroup of $\Auteq(\Fuk(M))$.
\end{thm}

In \cite{seidelgraded}, Seidel introduced a class of symplectomorphisms defined from Lagrangian submanifolds with periodic geodesic flow. This type of Lagrangian includes spheres---in which case the symplectomorphisms are squared Dehn twists---and projective spaces. This paper focuses on the latter class of symplectomorphisms, that we call \emph{projective twists} (the appellation \emph{Dehn} will be associated exclusively to Dehn twists along spheres). The projective analogues of Dehn twists are \emph{always} contained in the kernel of the comparison map \eqref{comparisonmap}, which means that they are a class of symplectomorphisms which are never detectable by the smooth structure. Similarly to Dehn twists, projective twists arise as local monodromies of fibration-like structures (\cite{perutzmbl}); these fibrations are called \emph{Morse--Bott--Lefschetz} fibrations and their singularities are Morse--Bott degenerations. 

Unlike their spherical counterparts, projective twists have not been in the spotlight of research in symplectic toplogy, and this is for a number of reasons. The definition of projective twist requires the existence of a Lagrangian embedding of a projective space in the ambient symplectic manifold, which can result in strong topological restrictions. Moreover, the symplectic Picard--Lefschetz theory of \cite{seidelbook} does not have such immediate applications as for Dehn twists.

Nevertheless, a series of recent results indicates that projective twists do have interesting properties of the caliber of Dehn twists: Evans \cite{evans}, Harris \cite{harristwist}, Mak--Wu \cite{makwu1}.

The results of the present research are driven by the following questions, which in the existing literature have been considered for Dehn twists exclusively.

\begin{quests}\label{mainq} \text{   }\vspace{0.01mm}
	Let $(M, \omega)$ be a Liouville manifold.
	\begin{enumerate}[label=(\alph*)]
		\item Can a reduced word of projective twists be symplectically isotopic to the identity (i.e are there twists satisfying any non-trivial relations) in $\Symp_{ct}(M)$?\label{questiona}
		\item Can a reduced \emph{positive} word (i.e a product of positive powers) of projective twists be symplectically isotopic to the identity in $\Symp_{ct}(M)$?\label{questionb}
	\end{enumerate}	
\end{quests}

\subsection{Methods: the Hopf correspondence}

How can we study projective twists? Because much of the scholarship that emerged from the study of Dehn twists is the result of successful applications of symplectic Picard--Lefschetz theory, a first intuitive move is to approach the study of projective twists by means of their presentation as monodromies of Morse--Bott--Lefschetz fibrations. One strategy could be to adapt some of the arguments originally tailored for Dehn twists to a more general Picard--Morse--Bott--Lefschetz theory as developed in \cite{wwfibred}. However, this setting presents serious complications related to a potential loss of compactness of the moduli spaces of pseudoholomorphic curves of these fibrations (the total space of Morse--Bott--Lefschetz fibations is in general not exact, and the singular locus---a smooth manifold of the singular fibre---often admits rational curves). 

To examine the properties of these symplectomorphisms, in this paper we adopt a strategy that allows to reduce the study of projective twists to that of Dehn twists in an auxiliary Liouville manifold; this is made possible via the theoretical device of \emph{Lagrangian correspondences}.

A Lagrangian correspondence between two symplectic manifolds $(W, \omega)$ and $(Y, \Omega)$ is a Lagrangian submanifold of the product $W^-\times Y:=(W \times Y, -\omega \oplus \Omega)$. By the work of \cite{wwcorr, wwfunctor, wwquilt, gaowrapped2,gaowrapped1}, under suitable conditions, a Lagrangian correspondence induces a functor which associates a Lagrangian in $Y$ to a Lagrangian in $W$. 

In a first stage, we define an appropriate Lagrangian correspondence that relates a set of Lagrangian projective spaces in a given Liouville manifold $(W, \omega)$ to a set of Lagrangian spheres in an auxiliary manifold $(Y, \Omega)$ expressly built under some cohomological conditions.
Fix a tuple $(\mathbb{A},k, *, R)\in \{ (\R, 0, 1, \Zmod), (\C, 1, 2, \Z)\}$. Assume there are Lagrangian projective spaces $\mathbb{A}\PP^n \cong K_1, \dots K_m \subset W$ and a non-trivial class $\alpha \in H^*(W;R)$ such that $\alpha|_{K_i}$ generates $H^*(\mathbb{A}\PP^n;R)$. Then there is a Liouville manifold $(Y, \Omega)$, realised as a $T^*S^k$-bundle $q\colon Y \to W$, which contains an $S^k$-fibred coisotropic submanifold $V \to W$, defining a Lagrangian correspondence $\Gamma:= \{  (q(y), y) , \ y \in V \} \subset W^- \times Y$ in the sense of \cite{perutzgysin}. Then, over each projective Lagrangian $K_i \subset W$, the correspondence yields a Lagrangian sphere $L_i \subset Y$, for $i=1, \dots , m.$
We name $\Gamma$ the \emph{Hopf correspondence}.

Once the Hopf correspondence constructed, we use Ma'u--Wehrheim--Woodward theory, and Gao's extension for non-closed correspondences, to show that it induces a functor $\Gamma\colon \Fuk(W) \to \Fuk(Y)$ between the compact Fukaya categories (see Section \ref{inducedfunctor}). We then prove the existence of a commuting diagram (Section \ref{commutingfunctor})

\begin{equation}\label{commutingofunctor}
\begin{split}
\xymatrixcolsep{5pc} \xymatrix{
	\Fuk(Y)  \ar[r]^{T_{L_i}}&
	\Fuk(Y) &\\
	\Fuk(W) \ar[u]^{\Theta_{\Gamma}}  \ar[r]^{T_{K_i}}	& \Fuk(W)  \ar[u]^{\Theta_{\Gamma}} }
\end{split}	\end{equation}

where $T_{K_i} \in Auteq(\Fuk(W))$ and $T_{L_i} \in Auteq(\Fuk(Y))$ are the twists-functors induced by the graphs of the respective twists $\tau_{K_i}\in \Symp_{ct}(W), \tau_{L_i}\in \Symp_{ct}(Y)$.

\subsection{Results}
\subsubsection{Free groups generated by projective twists}
In Section \ref{freegenplumb} we consider Question \ref{mainq}\ref{questionb} and prove new criteria for projective twists to generate a free subgroup of the kernel of the comparison map \eqref{comparisonmap}. We consider a \emph{clean plumbing} (see Definition \ref{defidentify}) of Lagrangian projective spaces; a symplectic construction in which two copies of cotangent bundles $T^*\mathbb{A}\PP^n$ are glued along a common submanifold of the zero sections, and prove the following. \\

\begin{restatable}{theorem}{thmmainthm}
	\label{mainthm}
	Let $W= T^*\mathbb{A}\PP^n \#_{\mathbb{A}\PP^{\ell}} T^*\mathbb{A}\PP^n$ be a clean plumbing of (real, complex) projective spaces along a linearly embedded sub-projective space $\mathbb{A}\PP^{\ell}\subset W$, $\mathbb{A}\in \{  \R, \C \}$. Let $K_1,K_2\cong \APn \subset W$ denote the Lagrangian core components of the plumbing. Then the projective twists $\tau_{K_1}$ and $\tau_{K_2}$ generate a free group inside $\pi_0(\Symp_{ct}(W))$, and the associated functors $T_{K_1}, T_{K_2}$ generate a free subgroup of $Auteq(\Fuk(W))$.
\end{restatable}

\sloppy This theorem is proved using the Hopf correspondence to relate the functors $T_{K_1}, T_{K_2} \in Auteq(\Fuk(W))$ to functors $T_{L_1}, T_{L_2} \in Auteq(\Fuk(Y))$ induced by Dehn twists in a Liouville manifold $(Y, \Omega)$ constructed as a $T^*S^k$-bundle over $W$, $k\in \{ 0,1,3 \}$. This is made possible via the commuting diagram \eqref{commutingofunctor}.

We can then apply Keating's result (Theorem \ref{keatingfree}) to our setting to obtain a free generation result for $T_{L_1}, T_{L_2}$, that we translate into a free generation result for $T_{K_1}, T_{K_2}$ via the Hopf correspondence.

\begin{rmk}\label{rmkbraid}
	The case $W:=T^*\C\PP^1_1\#_{pt}T^*\C\PP^1_2$ can be obtained with the current literature (\cite{seidelknotted}, \cite{stbraid}), see Remark \ref{rmklowdimension}. \end{rmk}

\subsubsection{Positive products of twists in Liouville manifolds}

In Section \ref{generalisationprojective}, we restrict our attention to products of positive powers of twists, i.e Question \ref{mainq}\ref{questionb}.
 In a first instance, we analyse this question for Dehn twists, and we provide an alternative proof of Theorem \ref{bgz}, which was originally proved (by Barthes--Geiges--Zehmisch in \cite{bgzfilling}) via techniques involving open book decompositions. Our proof is implemented using Picard--Lefschetz theory. The idea is to build a Lefschetz fibration $\pi\colon E\to \C$ with smooth fibre the Liouville manifold $(M, \omega)$ and vanishing cycles the given Lagrangian spheres involved in the product $\phi \in \Symp_{ct}(M)$. In that way, the monodromy of $\pi$ is given by $\phi$. Assuming that there exists an isotopy $\phi \simeq Id$ as in the statement, we extend $\pi$ to a fibration over $\CPone$, and, by analysing the moduli space of pseudoholomorphic sections (following \cite{seideles}), obtain a contradictory statement. 

It can happen that a product of Dehn twist, despite being necessarily not (compactly supported symplectically) isotopic to the identity, preserves some Lagrangian submanifolds of $M$. The question arises, whether one can find a Lagrangian $T\subset M$ such that there can be no compactly supported symplectic isotopy $\phi(T)\simeq T$. The existence of such a Lagrangian would result in a stronger version of Theorem \ref{bgz}. In Section \ref{relativegeo}, we address this question. We find one possible candidate Lagrangian $T\subset M$ with the above properties, but unfortunately cannot prove that such a Lagrangian always exist.

A Lagrangian $T\subset M$ is called conical if it is an exact, properly embedded Lagrangian that is preserved by the Liouville flow over the cylindrical ends of $M$.

\begin{restatable}{theorem}{thmrelativeversion}
	\label{relativeversion1}
	\sloppy	Let $(M^{2n}, \omega)$ be a Liouville manifold containing embedded Lagrangian spheres $L_1, \dots , L_m$ and a conical Lagrangian disc $T$ intersecting one of the spheres $L_j$ transversely in a point.
	Let $\phi:= \prod_{i=1}^{k} \tau_{L_{j_i}}\in \Symp_{ct}(M)$, $j_i \in \{ 1, \dots , m \}$ be a positive word of Dehn twists involving $\tau_{L_j}$. Then the Lagrangians $T$ and $\phi(T)$ are not isotopic via a compactly supported Lagrangian isotopy.
	
\end{restatable}

\begin{example}
	For $m>0$, consider an iterated transverse plumbing $M:=T^*S^m\#_{pt} T^*S^m \#_{pt} T^*S^m \cdots \#_{pt} T^*S^m$ (see Section \ref{plumbing} for the definition of plumbing). Let $\phi \in \Symp_{ct}(M)$ be a product of Dehn twists along the Lagrangian spheres of $M$, such that $\phi$ contains the Dehn twist along the $j$-th sphere. In this case, the conical Lagrangian disc $T\subset M$ of Theorem \ref{relativeversion1} is given by any cotangent fibre of the $j$-th summand.
\end{example}

The arguments we use in the proof of Theorem \ref{relativeversion1} are centred around the same principles of the method used for Theorem \ref{bgz}, with some necessary adjustments due to the non-compactness of the Lagrangian $T\subset M$.

At last, in Section \ref{projectiveproduct}, we turn to applications related to projective twists. Using the Hopf correspondence, we prove a result that can be considered the (real) projective counterpart to Theorem \ref{bgz}.

\begin{restatable}{theorem}{thmrealgeneral}
	\label{realgeneral}
	Let $(W^{2n}, \omega)$ be a Liouville manifold containing Lagrangian real projective spaces $K_1, \dots K_m$, $K_i \cong \RP$. Suppose that there is a class $\alpha \in H^1(W; \Zmod)$ such that for every $i=1,\dots , m$, $\alpha|_{K_i}$ generates $H^*(\RP; \Zmod)$.
	Let $\varphi \in \Symp_{ct}(W)$ be a positive word in the subset of projective twists $\{ \tau_{K_i}\}_{i \in \{ 1, \dots , m \}}$. Then $\varphi$ is not isotopic to the identity in $\pi_0(\Symp_{ct}(M)).$
\end{restatable}

Using the cohomological assumption of the theorem, we establish the Hopf correspondence and prove the theorem by contradiction. The idea is that in these circumstances, there exists a product of Dehn twists $\phi \in \Symp_{ct}(\widetilde{W})$ in the symplectic double cover $q\colon (\widetilde{W}, \widetilde{\omega})\lra (W,\omega)$, such that $q\circ \phi=\varphi\circ q$. Then, an isotopy $\varphi\simeq Id$ in $\Symp_{ct}(W)$ can be lifted to an isotopy $\phi \simeq Id $ in $\Symp_{ct}(\widetilde{W})$, contradicting Theorem \ref{bgz}.

Unfortunately, the same techniques do not yield a result for complex projective twists; in that case, the auxiliary manifold $(Y,\Omega)$ defines a $\C^*$-bundle $q\colon Y\to W$ and a compactly supported isotopy on $W$ does not lift to a compactly supported isotopy on $Y$, so that the arguments used for Theorem \ref{realgeneral} do not apply here.

\subsubsection{Framing of projective twists and Lagrangian embeddings of homotopy projective spaces}

The last section uses the Hopf correspondence to examine the ways in which the symplectic structure interpheres with the underlying topological structures, such as diffeomorphism and homeomorphism class, of Lagrangian homotopy projective spaces. 
In this paper, a manifold that is homeomorphic but not diffeomorphic to a standard (real or complex) projective space is called an AD projective space. A manifold that is homotopy equivalent but not homeomorphic to a standard (real or complex) projective space is called an AT projective space. Similarly, an AD sphere is a sphere that is homeomorphic but not diffeomorphic to the standard sphere (we decide to drop the usual epithet \emph{exotic}, see Definition \ref{abolishexotic}).

A notorious conjecture, known under the name of \emph{nearby Lagrangian conjecture}, states that given a closed smooth manifold $Q$, any closed exact Lagrangian submanifold of $(T^*Q,d\lambda_Q)$ is Hamiltonian isotopic to the zero section. 
This conjecture has generated a lot of interest in the symplectic community, but its statement is currently confirmed only up to homotopy equivalence (\cite{abouzaidnearby}, \cite{kraghnearby}, \cite{aknearby}; in Section \ref{stateart} we summarise the state of the art of this conjecture).
For a homotopy sphere $L$, it is known that the choice of smooth structure can be an obstruction to the existence of a Lagrangian embedding $L\hookrightarrow T^*S^n$. Namely, for $n>4$ odd, AD spheres which do not bound parallelisable manifolds admit no Lagrangian embedding into $T^*S^n$ (\cite{abouzaidframed, ekspheres}).

Using the existing literature about $S^1$-actions on AD spheres (\cite{bredonpairing}, \cite{jamesfree}, \cite{kasilingamprojective}), we find, in Section \ref{nonemb}, examples of non-standard homotopy complex projective spaces which do not admit Lagrangian embeddings into $T^*\CP$. These results are compatible with the predictions derived from the nearby Lagrangian conjecture.

\begin{theorem}[{{Theorems \ref{cor1}, \ref{cor2}}}]
	\begin{enumerate}[label=(\roman*)]
		\item 	There is a manifold $P$ homotopy equivalent to $\C \PP^4$ and with the same first Pontryagin class such that neither $P$ nor $P \# \Sigma^8$ admit an exact Lagrangian embedding into $T^*\C\PP^4$.
		\item 	There is an element $\Sigma^{14}$ in the group of homotopy $14$-spheres $\Theta_{14}$ such that $\C\PP^7\#\Sigma^{14}$ does not admit an exact Lagrangian embedding into $T^*\C\PP^7$.	
	\end{enumerate}	
\end{theorem}

On the other hand, in Section \ref{aptwist}, we present new results which prove that in general, the Hamiltonian isotopy class of projective twists does depend on a choice of framing, i.e a choice of smooth parametrisation $f \colon \CP \to L$ (see Definition \ref{definitionframing}). It was proved by Dimitroglou Rizell and Evans (\cite{dretwist}) that a non-standard parametrisation $S^n\to L$ of a Lagrangian sphere can give rise to a Dehn twist that is not isotopic to the standard Dehn twist $\tau_{S^n}$. 

We use classical homotopy theory and the Hopf corrsepondence to transpose the existence of non-standard parametrisations of Dehn twists of \cite{dretwist} into instances of projective twists depending on their framing.

\begin{restatable}[Corollary \ref{embcor1}]{theorem}{thmframing}
	\label{framingthm}
	The $\CP$-twist depends on the framing when $n=19,23,25,29$.
\end{restatable}

This shows that in general, $\Symp_{ct}(T^*\CP)$ is not generated by the standard projective twist along the zero section $\tau_{\CP}$ (see Corollary \ref{notgenerated}). Moreover, we also note that the use of advanced topological technology (\emph{topological modular forms}) can prove the existence of infinitely many non-standardly framed (complex) projective twists (Remark \ref{infinitemany}).

\subsection{Organisation of the paper}

The rest of the paper is organised as follows. 

Sections \ref{hopfbundle} and \ref{hopfcorr} are the two theoretical cores that support the arguments throughout the whole paper. After recalling the principal properties of twists in Section \ref{sectiontwist}, in Section \ref{hopfbundle} we prove commutative diagrams involving Dehn twists, Hopf map and projective twists in the geometric setting. In Section \ref{hopfcorr} we define the Hopf correspondence and its applications for diagrams of functors of the Fukaya category induced by Dehn/projective twists. 

The central body of the paper is is divided in three parts in which we apply the methods developed.
We prove a free group generation criterion for projective twists in plumbings in Section \ref{freegenplumb}, we study positive products of twists in general Liouville manifolds in Section \ref{dehngeneral} and framings of projective twists as well as Lagrangian embedding of homotopy projective spaces in Section \ref{epilogue}.

\subsection{Acknowledgements}

I would like to express my deepest gratitude to my PhD supervisor Ivan Smith, who provided support at every stage of this project. It was he who first suggested that I study products of projective twists, and helped me reconfigure and redefine the project as it unfolded. This work would not have seen the light of day without his meticulous explanations and knowledgeable advice; many results in this paper matured from original ideas he generously shared with me. I am also very grateful to him for providing helpful comments on numerous versions of this paper.

The ideas behind the proof of Theorem \ref{bgz} were first suggested by Paul Seidel, and I would like to thank Jack Smith for helping me implement it, and for his help and feedback throughout the writing process.

I am extremely grateful to Cheuk Yu Mak for sharing his expertise on projective twists, which helped me improve Section \ref{hopfbundle} considerably. I also want to thank him for his precious feedback on earlier versions of this work.

Many thanks also to Manuel Krannich and Oscar Randall-Williams, for sharing valuable insights on the action of the stable stem on homotopy spheres, which contributed to the non-triviality result of Lemma \ref{nontrivialcriterion}. Moreover, it was M.K who suggested I upgrade the result to infinitely many dimensions using topological modular forms (Remark \ref{infinitemany}).

I would like to extend my thanks to Ailsa Keating for guidance in the early phase of this work, to Gabriel Paternain for taking the time to hear about my work and give me valuable feedback, and to Jonny Evans, Momchil Konstantinov, Emily Maw and Agustin Moreno for helpful discussions.

I am most thankful to Vivien Gruar, whose generous help and support (both administrative and personal) played an essential role in fostering the optimal work environment to pursue my research.

Last but not least, I wish to express my sincere gratitude to my family and friends, who accompanied and guided me in these years of meandering explorations towards intellectual autonomy. 

This work was supported by the UK Engineering and Physical Sciences Research Council (EPSRC).

\section{Twists}\label{sectiontwist}

This section provides the contextualisation necessary for studying Dehn twists and projective twists in symplectic topology; it can be skipped by the expert reader.
We summarise the constructions of twists from a geodesic flow perspective (Section \ref{deftwist}), and as local monodromies of fibrations (Section \ref{dtandlf}).

\subsection{Twists from geodesic flow}\label{deftwist}
In this section we recall the definitions of Dehn and projective twists that employs the periodicity of the geodesic flow of spheres and projective spaces (the main references are \cite[1.2]{seideles}, \cite[4.b]{seidelgraded}, \cite[2.1]{makwu1}).
Let $(K,g)$ be a closed connected Riemannian manifold admitting a periodic cogeodesic flow $\Phi_K^t \colon T^*K \rightarrow T^*K$ on its cotangent bundle $(T^*K,d\lambda_{T^*K})$, such that each geodesic of length $2\pi$ is closed (so that the shortest period of a unit-speed geodesic is $2\pi$).

Let $\| \cdot \|_K$ be the norm associated to the given Riemannian metric $g$. The normalised cogeodesic flow satisfies $\Phi_K^{2\pi}=Id$ and can be extended to a Hamiltonian $S^1-$action $\sigma^H_t$ on $T^*K \setminus K$, with moment map $H: T^*K \setminus K \rightarrow \mathbb{R}, \ H(v)=\| v \|_K$.
\begin{definition}\label{twistlocal}
Let $K$ be diffeomorphic to $S^n$. For $\eps>0$, define an auxiliary smooth cut-off function $r_{\eps}\colon \R^+\to \R$ such that $0<r_{\eps}(t)< \pi$ for all $t<\eps$ and 

 \begin{align}\label{sphereconv}
	r_{\eps}(t)= \left\{
	\begin{array}{ll}
	\pi- t& t\ll \eps \\
	0 & t \geq \eps
	\end{array}
	\right.
	\end{align}

The model Dehn twist $\tau_K^{loc} \colon T^*K \rightarrow T^*K$  is defined as

	\begin{align}
	\tau_K^{loc}(\xi)= \left\{
	\begin{array}{ll}
	\sigma^H_{r_{\eps}(\|\xi \|_K)}(\xi) & \xi \notin K  \\
	-\xi & \xi \in K.
	\end{array}
	\right.
	\end{align}	
	
\end{definition}
	
	\begin{definition}
Let $K$ be diffeomorphic to $\APn$ for $\A\in \{\R, \C, \HH \}$. For $\eps>0$, let $r_{\eps}\colon \R^+\to \R$ be a smooth cut-off function such that $0<r_{\eps}(t)<2 \pi$ for all $t <\eps$ and

\begin{align}\label{projconv}
r_{\eps}(t)= \left\{
\begin{array}{ll}
2\pi-t& t \ll \eps \\
0 & t \geq \eps
\end{array}
\right.
\end{align}

The model projective twist $\tau_K^{loc} \colon T^*K \rightarrow T^*K$  is defined as
	
	\begin{align}\label{projtwistdef}
	\tau_K^{loc}(\xi)= \left\{
	\begin{array}{ll}
	\sigma^H_{r_{\eps}(\|\xi \|_K)}(\xi) & \xi \notin K  \\
	\xi & \xi \in K.
	\end{array}
	\right.
	\end{align}
	\end{definition}

\begin{rmk}
	Our choice of cut-off functions $r_{\eps}$ follows \cite[2.1]{makwu1}, but the construction is independent of such choices (\cite{seidelgraded}).
\end{rmk}

\begin{thm}[{{\cite[Corollary 4.5]{seidelgraded}}}]
	Let $(K,g)$ be a Riemannian manifold admitting a periodic (co-)geodesic flow and satisfying $H^1(K; \R)=0$. Then the symplectomorphisms $\tau_K^{loc}$ have infinite order in $\pi_0(\Symp_{ct}(T^*K))$.
\end{thm}
\begin{thm}{{\cite[Proposition 4.6]{seidelgraded}}}]\label{cpntrivial}
	The symplectomorphism $\tau_{\CP}^{loc}$ of Definition \ref{twistlocal} is isotopic to the identity in $\Diff_{ct}(T^*\CP)$.
\end{thm}

\begin{rmk}\label{sympsmooth}
	We will often denote the standard twists by $\tau_{S^n}:=\tau_{S^n}^{loc}$ or, for $\A\in \{ \R, \C, \HH\}$, $\tau_{\APn}:=\tau_{\APn}^{loc}$. With the conventions \eqref{sphereconv} and \eqref{projconv}, the isomorphisms $S^1\cong \R\PP^1$, $S^2 \cong \C \PP^1$ and $S^4 \cong \HH \PP^1$ induce isotopies $\tau_{S^1}^2\simeq \tau_{\R\PP^1}$, $\tau_{S^2}^2 \simeq \tau_{\C \PP^1}$ and $\tau_{S^4}^2 \simeq \tau_{\HH\PP^1}$ respectively (see \cite{seidelgraded}, \cite{harristwist}).
\end{rmk}

Now suppose $(L,g)$ is a Riemannian manifold admitting a Lagrangian embedding $L \subset M$ into a general symplectic manifold $(M,\omega)$.

\begin{definition}\label{definitionframing}
	Let $K \in \{  S^n, \RP, \CP , \HP  \}$. A \emph{framed Lagrangian} sphere/projective space is a Lagrangian submanifold $L \subset M$ together with an equivalence class $[f]$ of diffeomorphisms $f \colon K \rightarrow L$; $f_1 \sim f_2$ iff $f_2^{-1}f_1$ is isotopic, in $\Diff(K)$, to an element of the isometry group Iso$(K,g)$. An equivalence class $[f]$ as above is called a \emph{framing}.
\end{definition}

\begin{definition}\label{weinsteinextend} Let $(L, [f])$ be a framed Lagrangian sphere/projective space in $(M,\omega)$. Using Weinstein's neighbourhood theorem, extend a framing representative $f\colon K \to L$ to a symplectic embedding $\iota\colon D_{s}T^*K \to M$, where $D_sT^*K:=\{v \in T^*K, \|v \|_K <s \}$, $s>0$. There is a model twist $\tau^{loc}_{K}$, supported in the interior of $D_{s}T^*K$, and define
	
	\begin{align*}
	\tau_{L} \cong 
	\left\{
	\begin{array}{ll}
	\iota \circ \tau^{loc}_{K}\circ \iota^{-1} & \text{ on } \im(\iota) \\
	\text{Id } & \text{ elsewhere  }  
	\end{array}
	\right.
	\end{align*}
	In the case where $L$ is a sphere, the map $\tau_L$ is the well-known \emph{Dehn} twist. When $L$ is a projective space, the resulting map is called a \emph{projective twist}. In this paper, the term \emph{Dehn} twist is exclusively reserved for twists that are constructed from a Lagrangian sphere.
\end{definition}

\begin{rmk}
	\begin{enumerate}
		\item A Dehn twist along an exact Lagrangian sphere, or a projective twist along an exact projective Lagrangian in an exact symplectic manifold are exact symplectomorphisms in the sense of Definition \ref{exactsymplecto}. The same holds for products of such twists. This follows by construction (for direct computations see for example \cite[Lemma 4.4]{bgzfilling}, \cite[Lemma 2.1]{chiangtwist}). 
		\item Theorem \ref{cpntrivial} implies that given a symplectic manifold $(M, \omega)$, any Lagrangian $L\cong \CP \subset M$ will define an element $\tau_{L}\in \Symp_{ct}(M)$ that is isotopic to the identity in $\Diff_{ct}(M)$.
	\end{enumerate}
	
\end{rmk}
As shown by Dimitroglou Rizell and Evans in \cite{dretwist}, the choice of framing does play a role in determining the symplectic isotopy class of a spherical Dehn twist. In Section \ref{epilogue}, we prove that this is also the case for projective twists. Before then, any given Lagrangian submanifold involved in the construction of a twist is assumed to be endowed with a choice of framing and we omit mentioning this datum as the results of this paper, up to the last section, are independent of such choices. This is because the autoequivalence of the Fukaya category induced by a Dehn twist (see Section \ref{functortwists}) is independent of a choice of framing (as a consequence of the shape of the functor, see \cite[Corollary 17.17]{seidelbook}). The same is true for the functor induced by the projective twist (\cite[Theorem 6.10]{makwu1}).

\subsection{Twists as monodromies}\label{twistsasmonodromies}\label{dtandlf}
This section approaches twists from a different perspective, one that presents these symplectomorphisms as monodromy maps of fibration-like structures. Dehn twists occur as (local) monodromies of Lefschetz fibrations, and this is one of the features that has made the study of Dehn twists particularly productive. On the other hand (it is a lesser known fact that) projective twists can be modelled as local monodromies of \emph{Morse--Bott--Lefschetz} fibrations, another class of fibrations admitting more degenerate singularities, which we will not discuss here. 

Below, we give a brief review of Lefschetz fibrations (mainly following \cite{seidelbook,maysei}) on Liouville manifolds, aimed at setting the notation for future sections, and recalling the well-known \emph{Picard--Lefschetz theorem}.

\begin{definition}\label{defliouville}
	A Liouville manifold of finite type is an exact symplectic manifold $(W, \omega=d\lambda_W)$, where $\lambda_W \in \Omega^1(W)$ is called the Liouville form, such that there exists a proper function $h_W\colon W \rightarrow \lbrack 0, \infty )$ and $c_0>0$ with the following property. For all $ x \in ( c_0, \infty )$, the vector field $Z_W$ dual to $\lambda_W$, called the Liouville vector field, satisfies $dh_W(Z_W)(x)>0$. 
	
	For a regular value $c$ of $h_W$, a closed sublevel set $M:=h_W^{-1}([0,c])$ of a Liouville manifold $(W, d\lambda_W)$ is a compact symplectic manifold with contact type boundary $(\Sigma:=h_W^{-1}(c), \lambda_W|_{\Sigma})$, and it is called a Liouville domain. 
\end{definition}

\begin{definition}\label{exactsymplecto}
	An exact symplectomorphism between two Liouville manifolds $(W_1, d\lambda_1), (W_2, d\lambda_2)$ is a diffeomorphism $\psi\colon W_1 \to W_2$ satisfying $\psi^*\lambda_2-\lambda_1=df$, for a compactly supported function $f\colon W_1\to \R$. 
\end{definition}

\begin{definition}\label{defacs}
	Let now $(M,d\lambda)$ be a Liouville domain with contact boundary $(\Sigma=\partial M,\alpha=\lambda|_{\Sigma})$. The negative Liouville flow identifies a collar neighbourhood $C(\Sigma)$ of the boundary with $(-\eps, 0]\times \partial M$, such that $\lambda|_{C(\Sigma)}=e^t\alpha$. An almost complex structure $J$ of \emph{contact type near the boundary} is one that satisfies $de^t\circ J= -\lambda$.
\end{definition}

\begin{definition}
	Given a Liouville domain $(M,d\lambda)$ as above, we can use the identification of the collar neighbourhood $C(\Sigma)$ to glue an infinite cone and define the symplectic completion of $M$: \begin{align}
	(W, \omega_W):=(M \cup [ 0, \infty) \times \partial M, d(e^t \alpha)), 
	\end{align}
	where $t$ is the coordinate on $(0, \infty)$, such that the Liouville flow extends to $Z_W$ with $Z_W|_{[ 0, \infty) \times \partial M}=\partial_t$.
	
	An almost complex structure $J$ of contact type extends to an almost complex structure $J_W$ on the completion satisfying \begin{itemize}
		\item $J_W(\frac{\partial }{\partial t})=R_{\alpha}$, where $R_{\alpha}$ is the Reeb vector field associated to $\alpha$,
		\item $J_W$ is invariant under translations in the $t$-direction,
		\item $J_W|_{M}=J$.
	\end{itemize}
	This kind of almost complex structure will be called cylindrical. 
\end{definition}

We will only consider Liouville manifolds that are complete (i.e with complete Liouville vector field) and of finite type, which we can identify as the union of a Liouville domain with a cylindrical non-compact end, equipped with an almost complex structure cylindrical at infinity.

Let $(E^{2n+2}, \Omega_E,\lambda_E)$ be a Liouville manifold, with a compatible almost complex structure $J_E$, and consider the complex plane with its standard symplectic form and complex structure $j_{\C}$. Let $\pi\colon E \to \C$ be a map with finitely many critical points, which are all non-degenerate, and contained in a compact set of $E$.
Denote by $\Crit(\pi):=\{ x \in E, \ D_x\pi=0 \}$ the set of critical points, and by $\Crit v(\pi):=\pi(\Crit(\pi))$ the set of critical values.


\begin{definition}\label{deflf}
	A Lefschetz fibration on (the Liouville manifold) $E$ is a $(J_E,j_{\C})$-holomorphic map $\pi$, i.e $D\pi \circ J_E=j_{\C} \circ D\pi$, with the above properties and the following additional features.
	\begin{enumerate}
		\item For all $x\in E\setminus \Crit(\pi)$, $\ker(D_x\pi)\subset T_xE$ is symplectic and its symplectic complement with respect to $\Omega_E$, denoted by $T_xE^h$ is transverse to it, so that there is a symplectic splitting
		\begin{align}\label{splitting} 
		T_xE=\ker(D_x \pi)\oplus T_xE^h.
		\end{align}
		
		\item Every smooth fibre is symplectomorphic to the completion of a Liouville domain $(M, d\lambda_M)$.
		
		\item\label{trivialitycondition} There is an open neighbourhood $U^h \subset E$ such that $\pi\colon E \setminus U^h \lra \C$ is proper and $\pi|_{U^h}$ can be trivialised via an isomorphism $f\colon U^h \cong \C \times ([0, \infty) \times \partial M)$ such that
		\begin{align}\label{productlike}
		f^*(\lambda_E)=\lambda_{\C}+e^t\lambda_M.
		\end{align}\end{enumerate}
\end{definition}
For more details about how this fibration is modeled outside of a neighbourhood of the critical points, see \cite[(2.1)]{maysei}.

The decomposition in \eqref{splitting} defines a canonical connection over $\C\setminus \Crit v(\pi)$. By the triviality condition \ref{trivialitycondition}., for every path $\gamma\colon [0,1]\to D\setminus \Crit(\pi)$, there are well-defined parallel transport maps $h_{\gamma}\colon E_{\gamma(0)}\to E_{\gamma(1)}$ which yield symplectomorphisms between smooth fibres.

\begin{definition}\label{compatiblej}
	A pair $(J_E,j_{\C})$ is said to be \emph{compatible} with $\pi$ if the following holds.
	\begin{itemize}
		\item $D\pi\circ J_E=j_{\C}\circ D\pi$.
		\item There is a Kähler structure $J_0$ such that $J_E=J_0$ in a neighbourhood of $\Crit(\pi)$.
		\item On the neighbourhood $U^h$, $J_E$ is a product, $f^*(J_E)=(j_{\C},J^{vv})$, where $J^{vv}$ is a cylindrical almost complex structure compatible with $de^t\lambda_M$.
		\item $\Omega_E(\cdot, J_E \cdot)$ is symmetric and positive definite.
	\end{itemize}
\end{definition}

\begin{rmk}\label{nongenericrmk}
	This choice of almost complex structure is not generic. However, the space of compatible almost complex structures on the total space of an exact Lefschetz fibration is contractible (\cite[Section 2.1]{seideles}), and the moduli spaces we will consider still meet the usual regularity requirements (\cite[Section 2.2]{seideles}).
\end{rmk}

For a Lefschetz fibration on a Liouville manifold $(E,\Omega_E)$, the proper fibration obatined as $E\setminus U^h \lra \C$, for an open neighbourhood $U^h\subset E$ as above, carries the same symplectic information as $\pi$ with the difference that its fibres are Liouville domains, and as result the total space admits a non-trivial \emph{horizontal boundary}, given by the union of the boundaries of all fibres. 

In most of the paper we will employ this latter type of Lefschetz fibration (for notational simplicity), and unless specified, an \emph{exact} Lefschetz fibration will denote a fibration obtained in this way. 

Let now $\pi\colon E\to \C$ be an exact Lefschetz fibration, with smooth fibre given by the Liouville domain $(M,d\lambda)$. By the triviality assumption of Definition \ref{deflf}, there is a neighbourhood of $U^{\partial} \subset E$ of the horizontal boundary $\partial^h E$ that is isomorphic to an open neighbourhood of the trivial bundle $\C \times \partial M$: \begin{align}\label{horizontalboundary}
U^{\partial}\cong\C \times M^{out} \subset \C \times M
\end{align}
where $M^{out}\subset M$ is an open neighbourhood of $\partial M$. The isomorphism is compatible with the Liouville forms and the almost complex structures.

Let $\pi\colon E \to \C$ be a Lefschetz fibration with exact compact fibre $(M, \omega)$ and distinct critical values $\Crit v(\pi)=\{ w_0, \dots , w_m \} \subset D_R$, where $D_R\subset \C$ is a disc of radius $R$. Fix a base point $z_{*} \in \R$, such that $z_*\gg R$, and an identification $\pi^{-1}(z_*)\cong M$. In what follows we will frequently use the fact that via parallel transport, any fibre $\pi^{-1}(z)$ for $z\in \C$ with $\re(z)>R$ can be symplectically identified with the smooth fixed fibre $M$ via parallel transport.

\begin{definition}
	\begin{enumerate}
		
		\item A vanishing path associated to a critical value $w_i\in \Crit v(\pi)$ is a properly embedded path $\gamma_i \colon \R^+\to \C$ with $\gamma_i^{-1}(\Crit v(\pi))=\{ 0 \}$ and $\gamma_i(0)=w_i$, such that
		outside of a compact set containing the critical values, the image of $\gamma_i$ is a horizontal half ray at height $a_i \in \R$: \begin{align}\label{height0}
		\exists T>0 \text{ such that } \forall t>T, \  \re(\gamma_i(t))>R, \ \Im(\gamma_i(t))=a_i.
		\end{align}
		\item A distinguished basis of vanishing paths for $\pi$ is a collection of $m+1$ disjoint paths $(\gamma_0, \dots \gamma_m) \subset \C$ defined as above, with pairwise distinct heights satisfying $a_0<a_1 < \cdots < a_m$.

		\item The corresponding basis of Lefschetz thimbles is the unique set of Lagrangian discs $(\Delta_{\gamma_0}, \dots , \Delta_{\gamma_m}) \subset E$ where $\Delta_{\gamma_i}$ is defined as the set of points which under the limit $t \to 0$ of the parallel transport maps over $\gamma_i$ are mapped to the critical point in $\pi^{1}(w_i)$ (the proof of uniqueness can be found in \cite[(16b)]{seidelbook}).
		Given a general Lefschetz thimble $\L$, define its height $a(\L)$ as the value defined in \eqref{height0}. For a pair of thimbles $(\L_0, \L_1)$ set $\L_0 > \L_1$ if $a(\L_0)>a(\L_1)$.
		
		\item There is an associated basis of vanishing cycles $(V_0, \dots , V_m)$ where for all $i=0,\dots , m$, $$V_i=\partial \Delta_{\gamma_i}=\Delta_{\gamma_i}\cap M\subset M$$ (using the above identification for smooth fibres).
		Every vanishing cycle $V_i\subset M$ is an exact Lagrangian sphere which comes with an equivalence class in of diffeomorphisms $S^{n} \to V_i$ defined up to the action of O$(n+1)$ (called a \emph{framing}). This is induced by the restriction of a diffeomorphism $D^{n+1} \lra \Delta_i$ (which is canonical, see \cite[Lemma 1.14]{seideles}).
	\end{enumerate}
\end{definition}


\begin{figure}[h]
	\centering
\includegraphics[width=12cm]{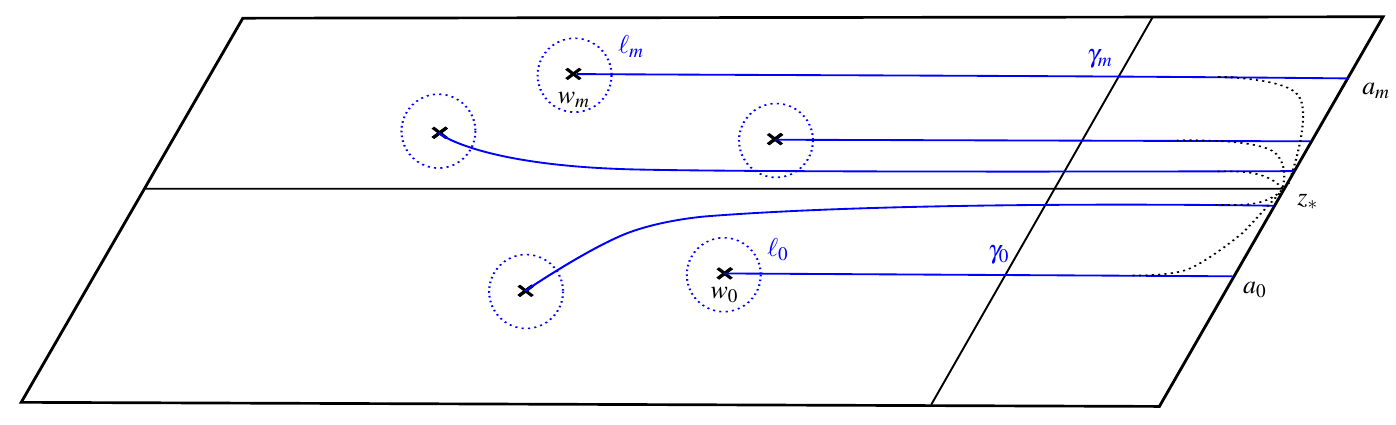}
	\caption{A distinguished basis of vanishing paths $(\gamma_0, \dots, \gamma_m)$.}\label{figurebasis}
\end{figure}

\begin{definition}
	The global monodromy is the symplectomorphism $\phi\in \Symp_{ct}(M)$ whose Hamiltonian isotopy class is defined by the anticlockwise parallel transport map around a loop through the base point $z_*$ encircling all the critical values of the fibration. (Typically, this loop is defined as the smoothing of the concatenation of the loops centered at $z_*$ going around a single critical value as in Figure \ref{figurebasis}.).
\end{definition}
The symplectic Picard-Lefschetz theorem (\cite{arnold}) states that the global monodromy $\phi$ is isotopic to the product of the Dehn twists along the vanishing cycles $(V_0, \dots , V_m)$, \begin{align}
\phi \simeq \tau_{V_0}\cdots \tau_{V_m} \in \Symp_{ct}(M),
\end{align}
and the Hamiltonian isotopy class is independent of the choice of basis of vanishing paths.

On the other hand, given the data $( M, (V_0, \dots, V_m ) )$, there is an exact Lefschetz fibration $\pi\colon E\to \C$ with fibre $(M,\omega)$, and vanishing cycles $(V_0, \dots , V_m)\subset M$, unique up to exact symplectomorphism (\cite[(16e)]{seidelbook}).

\begin{rmk}
Lefschetz fibrations can be viewed as a special case of \emph{Morse--Bott--Lefschetz} (abbreviated MBL) fibrations, a class of fibrations which allows non-isolated singularities. The monodromies of such fibrations are symplectomorphisms called \emph{fibred twists} (\cite{perutzmbl}), which naturally generalise Dehn twists.
Projective twists are a special type of fibred twists, sand therefore also admit a presentation as local monodromies. However, in this paper we won't study projective twists from this perspective.

\end{rmk}

\subsection{Functor twists}\label{functortwists}\label{geomtwistfunctor}
This section only contains the notation (and the general notions involved) that we will use to denote the functors of the Fukaya category that are induced by twists. 

Let $(M,\omega)$ be a Liouville manifold and let $k$ be a field of characteristic $2$. Given two closed exact Lagrangian submanifolds $L_0,L_1 \subset M$ the Floer complex is freely generated as a vector space by the intersection points of the (perturbed) Lagrangians $\CF(L_0,L_1;k):=\bigoplus_{x \in L_0 \cap L_1} k \langle x \rangle$. The boundary operator $\partial \colon \CF(L_0,L_1;k) \to \CF(L_0,L_1;k)$ counts $J_M$-holomorphic strips with boundary conditions on $(L_0, L_1)$ and asymptotic conditions on intersection points. For a compatible cylindrical almost complex structure $J_M$, the moduli spaces of such curves are compact oriented manifolds (\cite[Sections 8,9]{seidelbook}) and the operator $\partial$ squares to zero (\cite[(9e)]{seidelbook}), so that $(\CF(L_0,L_1;k), \partial)$ is a well-defined co-chain complex whose cohomology is the Floer cohomology ring $\HF(L_0,L_1;k)$. Floer cohomology is designed to be invariant under Hamiltonian isotopies; if $\phi$ is the flow of a Hamiltonian vector field, then $\HF(L_0,\phi(L_1)) \cong \HF(L_0,L_1)$. 

Very simply put, the compact Fukaya category, $\Fuk(M)$, is an $A_{\infty}$-category whose objects are closed exact Lagrangian \emph{branes}, which are Lagrangian submanifolds with additional algebraic data, and morphisms the Floer cochain groups between transversely intersecting Lagrangians (\cite[(9j),(12g)]{seidelbook}). This category encodes intersection data associated to all its objects, including the Floer differential $\partial=\mu^1$, the Floer cup product $\mu^2$ and higher order products $\mu^k$ (see e.g \cite[(9j), (12g))]{seidelbook}). It is well-defined for any Liouville manifold (see \cite{seidelbook}).

Two Lagrangians that are Hamiltonian isotopic are quasi-isomorphic objects in the Fukaya category, which means they are isomorphic objects of the associated cohomological category, that we denote by $H(\Fuk(M))$. We denote the automorphisms of $H(\Fuk(M))$ (i.e the automorphisms of the Fukaya category up to quasi-isomorphism) by $\Auteq(\Fuk(M))$.

Let $Tw(\Fuk(M))$ be the category of twisted complexes in $\Fuk(M)$ (see \cite[(3k)]{seidelbook}), and $D^{\flat}\Fuk(M):= H(Tw \Fuk(M))$ the cohomology category of $Tw(\Fuk(M))$.

There is a map \begin{align}\label{mapinducedfunctor}
\Phi: \Symp_{ct}(M) \rightarrow \Auteq(D^{\flat}\Fuk(M))
\end{align}
to the group of auto-equivalences of the Fukaya category (modulo quasi-isomorphism), such that given $\phi \in \Symp_{ct}(M)$, $\Phi(\phi)$ sends a Lagrangian $L\subset M$ to another Lagrangian $\phi(L)\subset M$. The map factors through the quotient by the subgroup $\Ham_{ct}(M)\subset \Symp_{ct}(M)$ of compactly supported Hamiltonian diffeomorphisms, so given an
exact Lagrangian sphere/projective space $L$ and its associated twist $\tau_L$, $\Phi(\tau_L)$ defines a well-defined element of $\Aut(D^{\flat}\Fuk(M))$ that we denote by $T_L$. 

In \cite{seideles}, Seidel showed that for a Dehn twist $\tau_L$, the induced functor $T_L \in \Aut (D^{\flat}\Fuk(M))$ fits into an exact triangle (see \cite[(17j)]{seidelbook}).


Recently, there have been generalisations of Seidel's triangle for a wider class of symplectomorphisms, achieved through a range of different techniques. Wehrheim--Woodward (\cite{wwfibred}) proved the existence of an exact triangle for fibred twists using quilt theory adapted to Morse--Bott Lefschetz fibrations.

Mak and Wu (\cite{makwu1}) treated the case of projective twists, using Lagrangian cobordism theory as developed in \cite{bcobordism1} and \cite{bcobordism2}.
They proved that the autoequivalence induced by a (real, complex, quaternionic) projective twist is isomorphic to a double cone of functors in $\Aut(Tw\Fuk(M))$ (\cite[Theorem 6.10]{makwu1}).


Under the appropriate circumstances, the mirror symmetry conjecture gives conjectural descriptions of such functors. If a symplectic manifold $(M, \omega)$ has a mirror complex manifold $(X, J)$, it is technically possible to introduce autoequivalences of the Fukaya category originally defined as autoequivalences of the derived category of coherent sheaves of $X$ (we call such autoequivalences \emph{algebraic twist functors}, and will only refer to them in Remark \ref{mirrordiagramrmk}).

\section{Commuting diagrams of twists}\label{hopfbundle}

In this section we introduce the geometric ideas underpinning the philosophy of the Hopf correspondence. We prove a criterion for relating projective twists in a Liouville manifold $(W,\omega)$ to Dehn twists in another Liouville manifold $(Y, \Omega)$.

\subsection{Complex projective Lagrangians}\label{complexhopf}

We begin by considering Lagrangian complex projective spaces.

Fix the round metric on $S^{2n+1}$, with norm $\| \cdot \|_S$ and consider the free $S^1$-action on $S^{2n+1}$ by complex multiplication. The orbits of the action are great circles (``Hopf circles''), hence geodesics, and the action is isometric.

Consider the quotient map $h\colon S^{2n+1}\to \CP$, which is the (generalised) Hopf fibration. It is a Riemannian submersion that uniquely defines the Fubini-Study metric $g_P$ on $\CP$. Identify the tangent bundles with their corresponding cotangent bundles $TS^{2n+1} \cong T^*S^{2n+1}$, $T\CP\cong T^*\CP$ via the the canonical isomorphism induced by the metrics. 

The Hopf action on $S^{2n+1}$ lifts to a Hamiltonian $S^1$-action on the cotangent bundle $(T^*S^{2n+1},\omega_{T^*S^n})$ (\cite{guillemin}). Let $\mu\colon T^*S^{2n+1}\to \R$ be the moment map of this action. Assume $0$ is a regular value of $\mu$ and consider the level set $\mu^{-1}(0) \subset T^*S^{2n+1}$, which has the structure of a principal $S^1$-bundle $p\colon\mu^{-1}(0)\to T^*\CP$ over the symplectic quotient $T^*S^{2n+1}\sslash S^1:=\mu^{-1}(0)/S^1 \cong T^*\CP$.

\begin{lemma}\label{localmodelcpn}
	Let $\tau_{S^{2n+1}} \in \Symp_{ct}(T^*S^{2n+1})$, $\tau_{\CP}\in \Symp_{ct}(T^*\CP)$ be the model Dehn and projective twists respectively. Let $p\colon V:=\mu^{-1}(0)\to T^*\CP$ be the symplectic quotient map as above.
	There is a commuting diagram
	\begin{equation}\label{localcpn}
	\begin{split}
	\xymatrixcolsep{5pc} \xymatrix{
		V  \ar[r]^{\tau_{S^{2n+1}}|_{V}} \ar[d]^{p} &
		V\ar[d]^{p} &\\
		T^*\CP \ar[r]^{\tau_{\CP}}	& T^*\CP.}.
	\end{split}	\end{equation} 
\end{lemma}
\proof The Hopf action is isometric, i.e for any $g\in S^1$, the induced map $\psi_g \in \Diff(S^{2n+1})$ is an isometry. This implies that the differential maps on the tangent bundles $D_p\psi_g\colon T_pS^{2n+1}\to T_{\psi_g(p)}S^{2n+1}$ commute with the geodesic flow. 

The co-geodesic flow $\Phi_{H}^t$ on $T^*S^{2n+1}$ is induced by the Hamiltonian function \begin{align}\label{geodham}
\widetilde{H} \colon  T^*S^{2n+1}& \longrightarrow \R \\
(p,\xi)& \longmapsto\|\xi \|_{S}. \nonumber
\end{align}

This is $S^1$-invariant, so that there is a Hamiltonian function $H\colon T^*\CP \to \R$ defined on the quotient, with respect to the submersion metric $g_P$, which induces the (co-)geodesic flow on $T^*\CP$.
Since $p$ is induced by a Riemannian submersion, we have the relation $p \circ \Phi_{\widetilde{H}}^t|_V = \Phi_{H}^t \circ p|_V$, 

and for any choice of cut-off function $r_{\eps}$ as in Section \ref{deftwist},
\begin{align}
p \circ \sigma^{\widetilde{H}}_{r_{\eps}(\| \xi \|_S)} (\xi)= \sigma^{H}_{r_{\eps}(\| p(\xi)\|_{P})} \circ  p(\xi), \ \xi \in V\subset T^*S^{2n+1}
\end{align}
where $\sigma_t^H, \sigma_t^{\widetilde{H}}$ are the Hamiltonian $S^1$-actions induced by $H$ and $\widetilde{H}$ respectively, as in Section \ref{deftwist}.

Any geodesic connecting a point on $S^{2n+1}$ to its antipode projects to a closed geodesic of minimal period on $\CP$, so the definitions of the twists in Section \ref{deftwist} imply that $p \circ \tau_{S^{2n+1}}|_V = \tau_{\CP}\circ p|_V$.

\endproof

We now extend the above discussion to a more global situation; in order to do that it is necessary to set the following assumption.

\begin{asspt}\label{assptcx}
	Let $(W, \omega)$ be a $4n$-dimensional Liouville manifold with a homology class $\alpha \in H^2(W ;\Z)$ and Lagrangian complex projective spaces $K_1, \dots ,K_m \subset W$ such that \begin{align*}
	\forall i: \ \alpha|_{K_i}=x \in H^2(\CP; \Z),
	\end{align*}
	where $x=c_1(\mathcal{O}_{\CP}(-1))$ is the generator of the cohomology ring $H^*(\CP;\Z)\cong \Z \lbrack x \rbrack /x^{n+1}$. 
\end{asspt}

\begin{prop}\label{cxlocalhopf}
	Let $(W, \omega)$ be a $4n$-dimensional Liouville manifold containing embedded Lagrangian complex projective spaces $K_1, \dots ,K_m \subset W$. Assume there exists a class $\alpha\in H^2(W;\Z)$ satisfying Assumption \ref{assptcx}. Then there is a $(4n+2)$-dimensional Liouville manifold $(Y, \Omega)$ with Lagrangian spheres $L_1, \dots ,L_m \subset Y$, a coisotropic submanifold $V\subset Y$ with the structure of an $S^1$-fibre bundle $p\colon V\to W$ such that for each $i \in \{  1, \dots , m\}$, $L_i\subset V$ and there is a commuting diagram 	\begin{equation}\label{localhopf1}
	\begin{split}
	\xymatrixcolsep{5pc} \xymatrix{
		V  \ar[r]^{\tau_{L_i}|_V} \ar[d]^{p} &
		V\ar[d]^{p} &\\
		W \ar[r]^{\tau_{K_i}}	& W.}.
	\end{split}	\end{equation}\end{prop}
The class $\alpha \in H^2(W;\Z)$ restricts to a generator $x \in H^2(\CP; \Z)$ on each Lagrangian $K_i$, so there is a complex line bundle $\mathcal{L} \rightarrow W$ satisfying $c_1(\mathcal{L})=\alpha$ which is modelled on the tautological line bundle $\mathcal{O}_{\CP}(-1)$ over $K_i$, for $i=1, \dots m$. 
Fix a metric $\| \cdot \|_{\mathcal{L}}$ on $\mathcal{L}$, and for $u\in \mathcal{L}$ define a function $r(u):=\|u\|_{\L}$. Set $V:=\{ u\in \mathcal{L}, \  r(u)=1  \}$. Over $K_i$, $V$ defines a sphere $L_i:=V|_{K_i}$.

\begin{lemma}\label{convexlineb}
	The $\C^*$-bundle associated to $\mathcal{L}$ is a Liouville domain where the spheres $L_i$ are embedded as Lagrangian submanifolds.
\end{lemma}
\proof Denote this bundle by $q\colon Y\rightarrow W$. Following \cite[7.2]{ritter}, we build a symplectic form $\Omega$ on $Y$, making the spheres $L_i$ Lagrangian, and find the appropriate vector field which will be Liouville with respect to $\Omega$.

The metric induces a connection one form $\gamma^{\nabla}$ on $\mathcal{L}\setminus 0$ satisfying  \begin{align}
\begin{split}
\gamma^{\nabla}|_{H_u^{\nabla}}=0, \ \ \gamma^{\nabla}|_{T_u^v\mathcal{L}}=\gamma \ \  \forall u \in \mathcal{L} \setminus 0 \\
[d\gamma^{\nabla}]=-q^*(c_1(\mathcal{L}))=-q^*(\alpha). 
\end{split}
\end{align} where $H_u^{\nabla} \mathcal{L}$ is the horizontal distribution associated to the connection $\nabla$ at $u$, $T_u^v\mathcal{L}$ the vertical distribution, and $\gamma$ the fibrewise angular form defined by the metric. 
Let $\Omega:= q ^*\omega+ d(f(r)\gamma^{\nabla})$, for a function $f \in C^{\infty}(\R)$ with

\begin{equation*}
\begin{array}{ll}
f(1)=0	& \\
f'(r)>0 & \text{ for all } r \in \R.
\end{array}
\end{equation*}

Then $\Omega$ defines a symplectic form in a neighbourhood of $\{ r=1 \}$, and $L_i$ is Lagrangian with respect to $\Omega$. Let $\lambda$ be the Liouville $1$-form on $W$ with $d\lambda=\omega$. Define $\lambda_Y:= q^*\lambda+f(r)\gamma^{\nabla}$ so that $d(\lambda_Y)=\Omega$. Then $(\lambda_Y, \Omega)$ defines a Liouville structure near $\{ r=1 \}$ (the symplectic dual to $\lambda_Y$ points outwards along a small neighbourhood of $\{  r=1 \}$). Therefore, a symplectic completion along this neighbourhood yields a Liouville manifold that is diffeomorphic to $Y$, containing the Lagrangian spheres $L_1, \dots , L_m$. \endproof

\proof[Proof of Proposition \ref{cxlocalhopf}]\sloppy 
Let $\mathcal{L}\to W$ be the complex line bundle we have constructed above with $c_1(\mathcal{L})=\alpha$.
For each Lagrangian projective space $K_i \subset W$, the restriction of the bundle $\mathcal{L}|_{K_i}$ is modeled on the tautological line bundle, which implies that $L_i\to K_i$ is modelled on the Hopf quotient map $h\colon S^{2n+1}\to \CP$. The commutativity of \eqref{localhopf1}
follows by the local commuting diagram of cotangent bundles \eqref{localcpn}. \qed

\begin{example}\label{counterex}
	Without assumption \eqref{assptcx}, Proposition \ref{cxlocalhopf} is in general not true, as the following example illustrates. 
	Consider the manifold $W$ obtained by attaching a $3$-handle to the contact boundary of $DT^*\C\PP^2$ such that $H^2(W;\Z)=0$.
	On one hand, $W$ contains a non-trivial Lagrangian $K=\C\PP^2 \subset W$ coming from the zero section (which is preserved by the handle attachment, since it is disjoint from the boundary. Note that the handle attachment is subcritical, so in fact the whole wrapped Fukaya category is preserved, see \cite{gpsii}).
	However, as there is no non-trivial $2$-cohomology class on $W$, there is no non-trivial $S^1$-bundle over $W$ that can be used to build a sphere over $K$.  
\end{example} 

\subsection{Real projective Lagrangians}\label{realhopf}
A similar procedure can be applied to a Liouville manifold containing real projective Lagrangians with an appropriate cohomology criterion. First recall the following.

Let $S^0 \cong \Zmod$ act on the sphere $S^n$ by the antipodal map. The quotient map $h: S^n \lra \RP$ is in this case a covering map, and induces a symplectic double cover $q\colon T^*S^n \rightarrow T^*\RP$ with $q^*\omega_{T^*\RP}=\omega_{T^*S^n}$.

\begin{lemma}[{{\cite[Lemma 2.4]{makwu1}}}]\label{cotangentcommuting2}
	Let $\tau_{\RP} \in \Symp_{ct}(T^*\RP)$ be the $\RP$-twist defined as in Section \ref{deftwist}. Then the diagram
	\begin{equation}\label{commutingrpn}
	\begin{split}
	\xymatrixcolsep{5pc} \xymatrix{
		T^*S^{n}   \ar[r]^{\tau_{S^{n}}} \ar[d]^{q} &
		T^*S^{n} \ar[d]^{q} &\\
		T^*\RP \ar[r]^{\tau_{\RP}}	& T^* \RP}.
	\end{split}	\end{equation}commutes.
\end{lemma}

\begin{asspt}\label{assptreal}
	Let $(W, \omega)$ be a $2n$-dimensional Liouville manifold with a homology class $\alpha \in H^1(W ;\Zmod)$ and Lagrangian real projective spaces $K_1, \dots K_m \subset W$ such that 
	
	\begin{align*}
	\forall i: \ \alpha|_{K_i}=x \in H^1(\RP; \Zmod)
	\end{align*} where $x=e(\gamma_{\R}^{1,n+1})$ is the the Euler class of the real tautological bundle $\gamma_{\R}^{1,n+1}\to \RP$, and generator of the cohomology ring $H^*(\RP;\Zmod)\cong \Zmod \lbrack x \rbrack /x^{n+1}$. 
\end{asspt}

\begin{prop}\label{realocalhopf}
	Let $(W, \omega)$ be a $2n$-dimensional Liouville manifold containing embedded Lagrangian real projective spaces $K_1, \dots K_m \subset W$. Assume there is a class $\alpha \in H^1(\RP; \Zmod)$ satisfying Assumption \ref{assptreal}. Then, there is a $2n$-dimensional Liouville manifold $(\widetilde{W}, \tilde{\omega})$ containing Lagrangian spheres $L_1, \dots , L_m \subset \widetilde{W}$ and a commuting diagram 
	
	\begin{equation}\label{diagramrpn}
	\begin{split}
	\xymatrixcolsep{5pc} \xymatrix{
		\widetilde{W}  \ar[r]^{\tau_{L_i}} \ar[d]^{q} &
		\widetilde{W}\ar[d]^{q} &\\
		W \ar[r]^{\tau_{K_i}}	& W.}.
	\end{split}	\end{equation}
	
\end{prop}

\proof In this case, the class $\alpha \in H^1(W; \Zmod)$ defines a symplectic double cover $q\colon (\widetilde{W}, \widetilde{\omega}) \rightarrow (W, \omega)$. Each Lagrangian $K_i\cong \RP$ then lifts to its double cover $L_i$, which is a sphere $S^n \subset \widetilde{W}$. Let $\lambda$ be the Liouville form on $W$.
As $q$ is symplectic, $\widetilde{\omega}=q^*(\omega)=q^*(d\lambda)=d(q^*(\lambda))$, and $\widetilde{\lambda}:=q^*(\lambda)$ defines a Liouville form on $\widetilde{W}$, which gives $\widetilde{W}$ the structure of a Liouville manifold.
Then the result follows by the local case illustrated by Lemma \ref{cotangentcommuting2}.

\begin{rmk}\label{quaternionhopf}

	It is possible to obtain an analogue diagram for the quaternionic twist as follows. Consider the free $S^3\simeq \text{Sp}(1)$-action on $S^{4n+3}$ inducing the quotient map $h\colon S^{4n+3}\to \HP$. This is a submersion as in the complex case, and the same arguments (with the natural metrics) yield the local commuting diagram
	\begin{equation}\label{quaternioniclocalhopf}
	\begin{split}
	\xymatrixcolsep{5pc} \xymatrix{
		\mu^{-1}(0)   \ar[r]^{\tau_{S^{4n+3}}|_{\mu^{-1}(0)}} \ar[d]^{p} &
		\mu^{-1}(0)\ar[d]^{p} &\\
		T^*\mathbb{H}\PP^n \ar[r]^{\tau_{\mathbb{H}\PP^n}}	& T^* \mathbb{H}\PP^n}.
	\end{split}
	\end{equation}
	where $p\colon \mu^{-1}(0)\lra T^*\HP$ is the $S^3$-fibre bundle induced given by the symplectic quotient map of the Hamiltonian action induced on $T^*S^{4n+3}$.
	
	Given an $8n$-dimensional symplectic manifold $(W,\omega)$ containing quaternionic projective Lagrangians, one would hope to find a cohomological condition to ensure the existence of a symplectic $(8n+6)$-dimensional manifold $(Y, \Omega)$ with corresponding Lagrangian spheres, as we did for the real and complex cases. However, homotopy classes of maps $W\to \HH\PP^{\infty}\cong \text{Sp}(1)$ do not classify quaternionic line bundles over $W$, so there is no analogue of Assumptions \ref{assptcx}, \ref{assptreal} to ensure the existence of such a manifold, and a commuting diagram of the form of \eqref{localhopf1}.

\end{rmk}

\section{The Hopf correspondence}\label{hopfcorr}

In this section we discuss the main theoretical device in action in this paper; Lagrangian correspondences. We begin by reviewing the main concepts from Wehrheim-Woodward Lagrangian correspondence theory (Section \ref{lagcorresp}). The rest of the chapter is then focused on the correspondence that will be used in our applications, the \emph{Hopf correspondence}. Given a real/complex projective Lagrangian $K\subset W$ in a Liouville manifold $(W,\omega)$ satisfying \eqref{assptreal}/ \eqref{assptcx}, the Hopf correspondence associates to it a Lagrangian sphere $L\subset Y$ in an auxiliary Liouville manifold $(Y,\Omega)$.
The key use of the Hopf correspondence in this section is aimed at achieving a categorical version of the commuting diagrams of the previous section. To do this, we first show that the Hopf correspondence $\Gamma \subset W^- \times Y$ induces a well-defined functor $\Theta_{\Gamma}\colon \Fuk(W)\to \Fuk(Y)$ (Sections \ref{inducedfunctor}, \ref{sectionhopf}).
We then show that the functors of $\Fuk(W)$ induced by projective twists are entiwined, via the correspondence, with the functors of $\Fuk(Y)$ induced by the Dehn twists (Section \ref{commutingfunctor}).
In Section \ref{sectiongysin}, we show that the Hopf correspondence can be used to build a symplectic Gysin sequence as established in \cite{perutzgysin}.

\subsection{Lagrangian correspondences}\label{generalcorrespondence}\label{lagcorresp}
We summarise the basic definitions and results associated to Lagrangian correspondences in the setting of \cite{wwcorr, wwfunctor,wwquilt, mwwfunctor}. For the entire section we let $k$ be a coefficient field of characteristic two.

\begin{definition}[{{\cite{wwquilt}}}]
	A \emph{Lagrangian correspondence} between two symplectic manifolds $(M_k, \omega_k)$ and $(M_{k+1}, \omega_{k+1})$ (``from $M_{k}$ to $M_{k+1}$'') is a Lagrangian submanifold $L_{k, k+1} \subset (M_k^- \times M_{k+1}):= (M_k \times M_{k+1},  -\omega_k \oplus  \omega_{k+1})$. 
	A \emph{cycle of Lagrangian correspondences} of length $r \geq 1$ is a sequence of symplectic manifolds $(M_0, \dots , M_{r+1}=M_0)$ together with a sequence of Lagrangian correspondences $\underline{L}:= (L_{01}, L_{12}, \dots , L_{r-1,r}, L_{r,1})$ such that $L_{k,k+1}\subset M_k^- \times M_{k+1}$ for $k=0, \dots ,r$.
\end{definition}

For example, a Lagrangian submanifold $L$ of a symplectic manifold $(M, \omega)$ is a trivial example of Lagrangian correspondence, seen as $L \subset \{  pt \}^- \times M=M$ (see other examples below).

\begin{definition}[{{\cite[Definition 2.0.4]{wwfunctor}}}]\label{lagcomp}
	Let $(M_i, \omega_i), i=0,1,2$ be symplectic manifolds and $L_{01} \subset M_0^- \times M_1$, $L_{12}\subset M_1^- \times M_2$ be Lagrangian correspondences. 
	\begin{enumerate}
		\item The correspondence transpose to $L_{01}$ is defined as $L_{01}^t:=\{  (m_1, m_0) | (m_0, m_1)\in L_{01}\} \subset M_1^- \times M_0$. Note that for a simple Lagrangian $L\subset M$ of a single symplectic manifold $M$, we won't distinguish $L$ from its conjugate.
		\item The composition of $L_{01}$ and $L_{12}$ is defined as \begin{align}
		L_{01}\circ L_{12}:= \left \{  (m_0, m_2) \in M_0^- \times M_2 \big | \ \exists m_1 \in M_1: \ 	\begin{array}{l}
		(m_0, m_1)\in L_{01}\\
		(m_1, m_2) \in L_{12}     
		\end{array}
		\right \} \subset M_0^- \times M_2
		\end{align}
		and it is called embedded if $L_{01}\circ L_{12}$ defines an embedded Lagrangian submanifold of $M_0^-\times M_2$.
	\end{enumerate}
\end{definition}

\begin{example}[{{\cite[1.1]{perutzgysin}}}]\label{coiso}
	Let $(M^{2n},\omega_M)$ a symplectic manifold with a coisotropic embedding $\iota \colon V \hookrightarrow M$. If the foliation defined by the integrable distribution $TV^{\omega}$ is a fibration $p\colon V\to B$ then the leaf space is a symplectic manifold $(B,\omega_B)$ satisfying $p^*\omega_B=\iota^*(\omega_M)$. The (transpose) graph of $p$, 
	$$\Gamma:=\{ (p(v), v), \ v \in V   \} \subset (B \times M, -\omega_B\oplus\omega_M)$$ is a Lagrangian correspondence.

\end{example}
A special case of Example \ref{coiso} is when the coiosotropic submanifold is obtained as a level set of a moment map induced by a Hamiltonian action:

\begin{example}[{{\cite[Example 2.0.2 (e)]{wwquilt}}}]\label{gmomentlagrangian} Let $(M,\omega_M)$ be a symplectic manifold. Let $G$ be a compact Lie group acting on $M$ Hamiltonianly with moment map $\mu \colon M \rightarrow \mathfrak{g}^*$. If $G$ acts freely on $\mu^{-1}(0)$, the latter is a smooth $G-$fibred coisotropic over the symplectic quotient $W:=M\sslash G=\mu^{-1}(0)/G$. $W$ is a symplectic manifold with symplectic structure $\omega_{M\sslash G}$ given by the Marsden-Weinstein theorem (see for example \cite[Section 5.4]{mcduffsal}). 
	The graph of the quotient map $p\colon \mu^{-1}(0)\rightarrow W$ is a Lagrangian submanifold of $(M \times W, - \omega_M \oplus \omega_W)$ and defines a Lagrangian correspondence, relating Lagrangians of $M$ with Lagrangians of its symplectic quotient. 
\end{example}

\subsection{Induced functors}\label{inducedfunctor}

In \cite{wwfunctor} and \cite{wwquilt}, Wehrheim and Woodward introduced a Floer cohomology theory adapted to cycles of closed Lagrangian correspondences $\underline{L}:=(L_{01}, \dots , L_{r0})$, called \emph{quilted Floer cohomology} and denoted by $\HF(\underline{L};k)$.
Pseudo-holomorphic quilts are a generalisation of the usual pseudoholomorphic strips used in standard Lagrangian Floer theory, and the quilted invariant is defined by counting pseudoholomorphic quilts with boundary constraints defined by the Lagrangian correspondences (\cite[Section 5]{wwquilt}). It can be viewed as a Floer theory in product symplectic manifolds (we refer to \cite[Section 4.3]{wwquilt} for definitions).

One of the main features is that given a cycle $\underline{L}$ of Lagrangian correspondences, quilted Floer cohomology is invariant under embedded composition (as in Definition \ref{lagcomp}) of subsequent Lagrangians in $\underline{L}$.

\begin{thm}[{{\cite[Theorem 5.4.1]{wwquilt}}}] \label{mainthmquilt}
	Let $\underline{L}=(L_{01}, \dots , L_{r(r+1)})$ be a cyclic sequence of closed, exact embedded and oriented Lagrangian correspondences between Liouville manifolds $(M_0, \dots , M_{r+1}=M_0)$ such that $\forall i, L_{(i-1)i}\circ L_{i(i+1)}$ is embedded.
	Then, for $\underline{L'}:=(L_{01}, \dots , L_{(j-1)j}\circ L_{j(j+1)}, \dots , L_{r(r+1)})$, there is an isomorphism $HF(\underline{L};k) \cong HF(\underline{L'};k)$.
\end{thm}

In \cite{mwwfunctor} the same authors and Ma'u proved that under certain assumptions, a Lagrangian correspondence $L_{01}$ between given symplectic manifolds $(M_0, \omega_0)$ and $(M_1, \omega_1)$, defines an $A_{\infty}$-functor $\Gamma_{01}$ between $\Fuk(M_0)$ and the dg-category of $A_{\infty}$-modules over $\Fuk(M_1)$. The functor is realised as the geometric composition $(\cdot)\circ \Gamma$ of Lagrangians submanifolds of $M_0$ with the correspondence, and this important result relies on the invariance Theorem \ref{mainthmquilt}. If for every Lagrangian in $M_0$ the composition outputs an embedded Lagrangian of $M_1$, the induced functor is between Fukaya categories.

\begin{thm}[{{\cite[Theorem 1.1]{mwwfunctor}}}]\label{functorcorrespondence}
	Assume $M_0, M_1$ are Liouville manifolds, and let $\Gamma_{01} \subset M_0^- \times M_1$ be a closed, exact and embedded correspondence such that for any closed embedded Lagrangian $K_0 \subset M_0$, the geometric composition
	\begin{align}
	L_1:=K_0 \circ \Gamma_{01}=\left \{ m_1 \in M_1, \ \exists m_0 \in K_0 \text{ such that } (m_0,m_1)\in \Gamma_{01} \right \} \subset M_1 
	\end{align}is a closed embedded Lagrangian in $M_1$.
	This assignement defines an $A_{\infty}$-functor \begin{equation}\label{mwwfunctor}
	\begin{split}
	\Theta_{01}:  &\mathcal{F}uk(M_0) \longrightarrow  \mathcal{F}uk(M_1), \\
	& \Theta_{01}(K_0)=L_1.
	\end{split}
	\end{equation} \end{thm}

In the above theorem, the correspondences are required to be closed, exact (or satify suitable monotonicity conditions) and embedded. Gao (\cite{gaowrapped2, gaowrapped1}) developed non-compact generalisations of Theorem \ref{functorcorrespondence}, including non-compact Lagrangian correspondences, in the setting of wrapped Fukaya categories. 

In both cases, the main theoretical device at work behind a result such as Theorem \ref{functorcorrespondence} (or Gao's equivalent) is quilted Floer theory, which, in \cite{gaowrapped1}, was adapted to a version suitable for non-compact correspondences. In this work we focus on a Lagrangian correspondence in a setting that features some properties of both theories. Before introducing our setting (see below), we review the types of Lagrangians that are admitted in a Gao's setting.

Let $(M_0, \omega_0)$, $(M_1, \omega_1)$ be Liouville manifolds with cylindrical almost complex structures $J_0$, $J_1$ and Liouville flows $Z_0, Z_1$ respectively. The product manifold $(M_0\times M_1, -\omega_0 \times \omega_1)$ is a Liouville manifold with respect to the product almost complex structure $J_{01}:=-J_0 \times J_1$ and Liouville flow $Z_{01}:=\pi_0^*(Z_0)+\pi_1^*(Z_1)$, for the projections $\pi_i\colon M_0\times M_1 \lra M_i$, $i=1,2$.

Let $\Sigma \subset M_0 \times M_1$ be the contact hypersurface given in \cite[2.2]{gaowrapped1}, such that we can fix a choice of cylindrical end that is compatible with the choices above. In other words, there is a compact set $U\subset M_0\times M_1$ bounded by $\Sigma$, such that there is a diffeomorphism $M_0\times M_1 \setminus U\cong [0, \infty) \times \Sigma$ (\cite[(2.5)]{gaowrapped1}).

\begin{definition}\label{cylag}
	A Lagrangian submanifold is said to be \emph{conical} if it is an exact, properly embedded Lagrangian which is preserved by the Liouville vector field over the cylindrical end.
\end{definition}

\begin{definition}[{{\cite[Definition 3.9]{gaowrapped1}}}]\label{gaolags}
	A Lagrangian submanifold $\Gamma_{01} \subset M_0^- \times M_1$ is called admissible if it is 
	\begin{enumerate}
		\item Either a product of conical Lagrangian submanifolds of $M_0^-$ and $M_1$;
		\item Or a Lagrangian that is conical with respect to the cylindrical end $\Sigma \times [0,+\infty)$.
	\end{enumerate}
\end{definition} 

Gao defines geometric composition for this type of Lagrangian correspondences and proves the analogue of Theorem \ref{mainthmquilt} (\cite[Theorem 1.5]{gaowrapped1}). Moreover, he shows the open analogue of Theorem \ref{functorcorrespondence}, namely that such a Lagrangian correspondence induces a functor of wrapped Fukaya categories (\cite[Theorem 1.2]{gaowrapped2}).

Below we focus on the type of correspondences we consider in this paper, which arises as a special case of Example \ref{coiso} for a non-compact coisotropic. It is a class of exact, embedded, but not closed correspondences between Liouville manifolds. 

\begin{setting}
	Let $(M_0, \omega_0)$ and $(M_1, \omega_1)$ be Liouville manifolds such that there is a fibration $q\colon M_1 \lra M_0$ with Liouville fibres.
	
	Let $\Gamma \subset M_0^- \times M_1$ be a Lagrangian correspondence obtained as the (transpose) graph of a proper fibration $p\colon V \to M_0$, where $V\subset M_1$ is a fibred coisotropic as in Example \ref{coiso}, and $q|_V=p$.
	
	On $M_0^-\times M_1$ set the product almost complex structure $J_{01}=-J_0\times J_1$ for cylindrical almost complex structures on $M_0$ and $M_1$, such that the fibration $(id, q) \colon M_0^- \times M_1 \lra M_0^- \times M_0$ is $(J_{01}, J_{00})$-holomorphic, for $J_{00}=-J_0\times J_0$.
	
	Then $\Gamma=\{ (p(v), v), \ v\in V \}$ is properly fibred over the diagonal $ \Delta_{M_0}:=\{ (p(v), p(v)), \ v\in V) \}= \subset M_0^- \times M_0$ which is a conical Lagrangian correspondence in $M_0^-\times M_0$.
	However, the original correspondence $\Gamma$ is not conical, or more generally admissible in the sense of Definition \ref{gaolags}. 
\end{setting}

Consequently, the above setting doesn't exactly fit either the compact nor the open quilted theories, but is located between the two: it represents a class of non-compact correspondences which nevertheless induces a functor of compact Fukaya categories.

\begin{axiom} \label{axiomcorrespondence}
	The type of Lagrangian correspondence $\Gamma \subset M_0^- \times M_1$ defined in the above setting induces a functor \begin{equation}\label{functorcoisotropic}
	\begin{split}
	\Theta_{01}:  &\mathcal{F}uk(M_0) \longrightarrow  \mathcal{F}uk(M_1), \\
	& \Theta_{01}(K_0)=L_1.	 
	\end{split}
	\end{equation}
\end{axiom} 

Experts will recognise the validity of the above statement that we have set as an axiom. Proving it as a theorem would require a lengthy digression necessary to fill in all details covered in \cite{wwfunctor, gaowrapped1, gaowrapped2}. 
In Lemma \ref{lemmaaxiom}, we restrict to proving the invariance of quilted Floer cohomology under Lagrangian correspondences.
Given invariance, the results of \cite{wwfunctor} yield a functor on the cohomological category. The extension to an $A_{\infty}$-functor, which would turn Axiom \ref{axiomcorrespondence} into a theorem, can then be obtained by considering higher $A_{\infty}$-products, which we omit here.

\begin{lemma}\label{lemmaaxiom}
	Let $K\subset M_0$, $L' \subset M_1$ be closed exact Lagrangians and consider the cycle of correspondences $(K, \Gamma, L') \subset (pt, M_0, M_1)$. Then the quilted Floer cohomology group $\HF(K, \Gamma, L')$ is well-defined and satisfies the invariance property \begin{align}
	\HF(K, \Gamma, L')\cong\HF(K \circ \Gamma, L')=\HF(L,L').
	\end{align}
	
\end{lemma}

\proof 

By definition (see \cite[Section 4.3]{wwquilt}), the generators of the cochain complex $\CF(K, \Gamma, L')$ are given by the generators of $\CF(K\times L, \Gamma)$.
These intersection points must be contained in a compact region, since $K\subset M_0$ and $L'\subset M_1$ are closed Lagrangians. By \cite[Proposition 2.2.1]{wwcorr} the cochain groups $\CF(K \circ \Gamma, L')=\CF(L,L')$ and $\CF(K\times L, \Gamma)$ are isomorphic. 

We now analyse the Floer trajectories involved in the computation of $\HF(K, \Gamma, L')$. 

By the maximum principle, the only non-compactness phenomenon that could occur would be a $J_{01}$-holomorphic curve escaping a compact set on the non-compact boundary condition $\Gamma$.
However, all such curves, and any Floer trajectory of interest, are contained in a compact set, as we now explain.

By assumption, $J_{01}$-holomorphic curves with boundary conditions on $(K, \Gamma, L')$ project under $(id,q)$ to $J_{00}$-holomorphic curves involved in the complex for the tuple $(K, \Delta_{M_0}, K')$, where $K' \circ \Gamma=L'$ and $(id,q)(\Gamma)=\Delta_{M_0}$.

The (quilted) Floer cohomology group for the cycle of Lagrangian correspondences $(K, \Delta_{M_0}, K') \subset (pt, M_0, M_0)$ can be defined as the Floer cohomology group $\HF(K,\Delta_{M_0}, K'):=\HF^*(K \times K', \Delta_{M_0})$ (\cite[Lemma 4.8]{gaowrapped1}). Moreover by \cite[Theorem 1.2]{gaowrapped2}, $\Delta_{M_0}$ induces the identity functor, so clearly all the $J_{00}$-holomorphic strips involved in the complex $\CF(K,\Delta_{M_0}, K')$ are well-behaved, and moreover we have $\HF(K,\Delta_{M_0}, K')\cong \HF(K,K')$.

Because of properness of $(q, id)|_{\Gamma}\colon \Gamma_{01}\to M_0 \times M_0$, if there was any $J_{01}$-holomorphic curve escaping to infinity at the boundary condition $\Gamma$, then it would project to a $J_{00}$-holomorphic curve escaping to infinity at the boundary condition on $\Delta_{M_0}$, which cannot happen.
\qed

\begin{rmk}
	Let $K,K' \subset (M_0,\omega_0)$ closed exact Lagrangians. Note that for any conical correspondence (not just the diagonal) $\Gamma_{00}\subset M_0^-\times M_0$, compactness of moduli spaces of curves involved in the quilted complex $\CF(K, \Gamma_{00}, K')$ (for compact Lagrangians $K, K' \subset M_0$) is preserved. Namely, all intersection points lie in a compact region, so by exactness both energy and symplectic area are bounded. We can apply a reverse isoperimetric inequality according to which the length of the boundary of such a curve is bounded by a quantity proportional to its area (\cite[Theorem 1.4]{ysreverse}).
	
	This ensures that the boundary of all pseudoholomorphic curves is contained in a compact set, which can then be determined by using a monotonicity lemma in the likes of \cite[Lemma 13]{seidelsmithstretch}.
	Again, by exactness there is no bubbling so the moduli spaces of such curves are compact.

\end{rmk}

\subsection{The Hopf correspondence}\label{sectionhopf}
We can finally introduce the correspondence of interest; the \emph{Hopf correspondence}. 
This a Lagrangian correspondence obtained as the graph of a spherically fibred coisotropic submanifold as in Example \ref{coiso}.

We use the discussions of Sections \ref{complexhopf}, \ref{realhopf} to explain how, for each type of Lagrangian projective space $K\cong \mathbb{A}\PP^n \subset W$, $\mathbb{A}\in \{  \R, \C \}$ in a Liouville manifold $(W,\omega)$ satisfying the appropriate cohomology Assumption \eqref{assptreal}, \eqref{assptcx}, there is a Lagrangian correspondence relating $K$ to a Lagrangian sphere $L$ in an auxiliary Liouville manifold $(Y,\Omega)$.

\subsubsection{Lagrangian $\CP$}\label{hopfcpn}
Let $(W^{4n},\omega)$ be a Liouville manifold admitting Lagrangian submanifolds $K_i \cong \CP \hookrightarrow W$, $i=1, \dots, m$. Assume there is a class $\alpha \in H^2(W;\Z)$ satisfying Assumption \eqref{assptcx}.
The discussion of Section \ref{complexhopf} delivers a $\C^*$-bundle $q \colon Y\lra W$ (associated to the complex line bundle $\mathcal{L}\to W$ with $c_1(\mathcal{L})=\alpha$), whose total space is a Liouville manifold $(Y, \Omega)$ (proof of Proposition \ref{cxlocalhopf}). Set $V:=Y|_{\{  r=1 \}}$, the unit length bundle (determined by the metric on $Y$ induced by a choice of hermitian metric on $\mathcal{L}$).
If $V\hookrightarrow Y$
is the inclusion, then by construction $\iota^*\Omega=q^*\omega|_V$, so the symplectic reduction of $V$ by $S^1$ is given by $(W, \omega)$, and $V$ is a fibred coisotropic submanifold of $(Y, \Omega)$ with $S^1$-fibre bundle structure $p=q|_V\colon V\to W$.

For any Lagrangian projective space $K_i\subset W$, the restriction $V|_{K_i}\to K_i$ is a Lagrangian sphere $L_i\cong S^{2n+1} \subset Y$. 

\begin{definition}\label{definitionhopf}
	The (transpose) graph \begin{align}\label{graph}
	\Gamma:=\{ (p(y),y), y \in V  \}  \subset W^- \times Y
	\end{align}
	defines a Lagrangian correspondence (\cite[Proposition 1.1]{perutzgysin}), which we call the \emph{Hopf correspondence}.
	By construction, for $K_i \cong \CP \subset \{ pt\} \times W$, the correspondence maps $K_i$ to the embedded Lagrangian sphere $L_i:=K_i\circ \Gamma\cong S^{2n+1} \subset \{ pt \} \times Y\cong Y$, $i= 1, \dots, m$ via geometric composition.
\end{definition}

\begin{rmk}
	This Lagrangian correspondence can equivalently be thought of as a correspondence of the type of Example \ref{gmomentlagrangian}, where the coisotropic $V$ is a regular level set of a Hamiltonian $S^1$ ``Hopf''-action, and $(W,\omega)$ its symplectic quotient (note that the local models \eqref{localcpn}, \eqref{commutingrpn}, \eqref{quaternioniclocalhopf} are obtained from this perspective). This explains the choice of name for the correspondence.

\end{rmk}

\subsubsection{Lagrangian $\RP$}\label{hopfrpn}

Let $(W^{2n}, \omega)$ be a Liouville manifold admitting Lagrangian embeddings $K_i \cong \RP \hookrightarrow W$, $i=1, \dots m$. Assume there is a cohomology class $\alpha \in H^1(W;\Zmod)$ satisfying Assumption \ref{assptreal}.
 
Then, there is a Liouville manifold $(Y,\Omega)=(\widetilde{W}^{2n}, \widetilde{\omega})$ obtained as symplectic double cover of $W$, and containing Lagrangian spheres $L_1, \dots , L_m \subset \widetilde{W}$. The double cover $q\colon \widetilde{W} \lra W$ defines an $S^0$-fibration over $W$, and in this case the ``coisotropic submanifold'' is the total space itself. As above, we define the Hopf correspondence as $\Gamma:=\{ (q(y),y), y \in \widetilde{W}   \}  \subset W^- \times \widetilde{W}$.

\subsection{Commuting diagrams of functors}\label{commutingfunctor}
Let $(W, \omega)$ and $(Y, \Omega)$ be Liouville manifolds and $K_1, \dots K_m\subset W$ be real/complex projective Lagrangians satisfying Assumptions \eqref{assptreal}/\eqref{assptcx}. Let $q\colon Y \lra W$ be the fibration we constructed in the previous subsections, and $\Gamma \subset W^- \times Y$ be the Hopf correspondence, obtained as the graph $\Gamma=\{ (p(v), v), \ v\in V \}$ of the spherically fibred coisotropic $p=q|_V\colon V\to W$. This correspondence is properly fibred over the diagonal $\Delta_{W} =\{ (p(v), p(v)), \ v\in V) \} \subset W^- \times W$, via $(id,q)\colon W\times Y \lra W\times W$ and satisfies the conditions of Axiom \ref{axiomcorrespondence}. Therefore, there is a well-defined functor \begin{align}\label{functorinduced}
\Theta_{\Gamma}: \Fuk(W) \rightarrow \Fuk(Y), \ \Theta_{\Gamma}(K)=K \circ \Gamma=:L.
\end{align}

Let $L_1, \dots , L_m \subset Y$ the Lagrangian spheres associated to $K_1, \dots , K_m$ through the correspondence. For each $i=1, \dots m,$ let $T_{K_i} \in Auteq(\Fuk(W))$ and $T_{L_i} \in Auteq(\Fuk(Y))$ be the (geometric) twists functors induced by the graphs of the respective twists $\tau_{K_i}\in \Symp_{ct}(W), \tau_{L_i}\in \Symp_{ct}(Y)$. 

\begin{cor}\label{diagramfunctors}
	There is a commuting diagram at the level of compact Fukaya categories 
	\begin{equation}\label{categoricaltwistplumbing}
	\begin{split}
	\xymatrixcolsep{5pc} \xymatrix{
		\Fuk(Y)  \ar[r]^{T_{L_i}}&
		\Fuk(Y) &\\
		\Fuk(W) \ar[u]^{\Theta_{\Gamma}}  \ar[r]^{T_{K_i}}	& \Fuk(W)  \ar[u]^{\Theta_{\Gamma}}. }
	\end{split}	\end{equation}
	
	In particular, iterative applications of this diagram yield
	\begin{align}\label{commutativeprod}
	\Theta_{\Gamma} \circ \prod T_{K_i}^{k_i}= \prod T_{L_i}^{k_i}\circ \Theta_{\Gamma}.
	\end{align}
\end{cor}
\proof
Consider the functors $T_{K_i}$ and $T_{L_i}$ as correspondences induced by the graphs of the respective twists $\tau_{K_i}\in \Symp_{ct}(W), \tau_{L_i}\in \Symp_{ct}(Y)$. Then we have to check that the compositions of correspondences $\Theta_{\Gamma} \circ T_{K_i}=T_{L_i} \circ \Theta_{\Gamma}$, as Lagrangians in $W^- \times Y$, coincide. By construction, this equality amounts to the commutativity of diagram \eqref{localhopf1}.
\endproof

\begin{rmk}
	Note that for a coefficient field of characteristic zero, the functor associated to the real projective twist has a different shape which produces a different diagram (\cite[Corollary 1.3]{mwtwist2}).
\end{rmk}

\begin{rmk}\label{mirrordiagramrmk}
	Given a hypothetical mirror pair $(X, M)$ for a symplectic manifold $(M, \omega)$ with $c_1(M)=0$ and complex manifold $(X,J)$ we can make the following observation.
	
	In \cite{htwist}, Huybrechts and Thomas conjectured that the functors induced by projective twists on the derived Fukaya category $D^{\flat}(\Fuk(M))$ should be mirror to a class of autoequivalences of $D^{\flat}(X)$, induced by so called $\PP$-objects (see \cite[Definition 1.1]{htwist}). This is the analogue of the statement proved by Seidel that autoequivalences of $D^{\flat}(\Fuk(M))$ induced by Dehn twists should be mirror to autoequivalences of $D^{\flat}(X)$ induced by ``spherical objects'' (see \cite[Definition 1.1]{stbraid}).
	
	From this perspective, we can view the diagram \eqref{diagramfunctors} as a conjectural mirror to the following situation.

	By \cite[Proposition 1.4]{htwist}, a $\PP$-object $\mathcal{P} \in D^{\flat}(X)$ in the central fibre of an algebraic deformation $j\colon X \hookrightarrow \mathcal{X}$ and satisfying $0\neq A(\mathcal{P}) \cdot \kappa(\mathcal{X})\in \Ext^2(\mathcal{P}, \mathcal{P})$ has an associated spherical object given by $j_*(\mathcal{P}) \in D^{\flat}(\mathcal{X})$. Here, $A(\mathcal{P})\in \Ext^1(\mathcal{P}, \mathcal{P}\otimes\Omega_X^1)$ is the Atiyah class of $\mathcal{P}$ and $\kappa(\mathcal{X}) \in H^1(X, \mathcal{T}_X)$ the Kodaira-Spencer class of the family $\mathcal{X}$. 
	Furthermore, the autoequivalences associated to each object (also called ``twists''), $T_{\mathcal{P}}$ and $T_{j_*\mathcal{P}}$, are related by a commutative diagram (\cite[Proposition 2.7]{htwist})
	\begin{equation}\label{mirrordiagram}
	\begin{split}
	\xymatrixcolsep{5pc} \xymatrix{
		D^{\flat}(X) \ar[r]^{j_*}  \ar[d]^{T_{\mathcal{P}}} &
		D^{\flat}(\mathcal{X})  \ar[d]^{T_{j_*\mathcal{P}}} &\\
		D^{\flat}(X) \ar[r]^{j_*}	&D^{\flat}(\mathcal{X}) }
	\end{split}	\end{equation}
	
\end{rmk}

\subsection{Lagrangian Gysin sequence}\label{sectiongysin}

Let $\Gamma \subset W^- \times Y$ be the Hopf correspondence. Given real/complex projective Lagrangian submanifolds $K, K' \subset W$ and their corresponding spherical lifts $L, L' \subset Y$ through the functor $\Theta_{\Gamma}$, a version of Perutz's Gysin sequence of \cite{perutzgysin} can be used to establish a relationship between the ranks of the Floer cohomology groups $\HF(K,K')$ and $\HF(L,L')$. We will need this relation in the next section, for the proof of Theorem \ref{mainthm}.

Let $V\to W$ be the $S^k$-fibred coisotropic defining the correspondence, $k\in \{0,1\}$, with Euler class $\alpha \in H^{k+1}(W;R)$, $R\in \{ \Zmod, \Z \}$ and Lagrangian projective spaces $K, K' \subset W$ satisfying Assumptions \eqref{assptreal}/\eqref{assptcx} respectively.

Let $L=\Theta_{\Gamma}(K)=K\circ \Gamma\subset Y$, $L'=\Theta_{\Gamma}(K')=K'\circ \Gamma\subset Y$ be the associated Lagrangian sphere given by the correspondence.

\begin{lemma}\label{gysinper}
	There is an exact triangle of the shape
	\begin{equation}\label{exactriangle}
	\begin{split}
	\xymatrix{ 
		\HF^*(K,K') 	\ar[rr]^{ \alpha\  \cup \cdot} && \HF^{*+k+1}(K, K')  \ar[dl]^{\Gamma_*}  \\
		& \HF^{*+k+1}(L,L')\ar[ul] 
	}	\end{split}  \end{equation}
	
\end{lemma}
\proof
This exact sequence follows from the Gysin triangle proved by Perutz in \cite[Theorem 1]{perutzgysin}, which has the more general form 
\begin{align}
... \longrightarrow	\HF^*(K,K') 	\xlongrightarrow{e(V) \cup \ \cdot} \HF^{*+k+1}(K, K') \xlongrightarrow{\Gamma_*}\HF^{*+k+1}(K, \Gamma^t, \Gamma, K') \longrightarrow ... 
\end{align}
where the last group is the quilted Floer cohomology group of the cycle of Lagrangian correspondences $\underline{L}:=(K, \Gamma, \Gamma^t, K')\subset (pt, W,Y,W)$, satisfying $\HF(K, \Gamma, \Gamma^t, K')  \cong \HF^*(K \circ \Gamma, K' \circ \Gamma) \cong \HF(L,L')$.

The isomorphism follows from Axiom \ref{axiomcorrespondence} (in particular Lemma (\ref{lemmaaxiom} applied to a sequence of four Lagrangian correspondences). The compositions $L=K \circ \Gamma$, $L'=\Gamma^t \circ K'=K' \circ \Gamma$ are embedded, and coincide with the spheres (which are Lagrangian in $W$) in the sphere bundle $V$ over $K$, respectively $K'$.

The first map in the original exact sequence \eqref{exactriangle} is the quantum cup product with the Euler class $e(V) \in QH^*(W)$. In this case the exactness assumptions on the ambient symplectic manifold $W$ ensure the well-definedness of the operation and as $QH^*(W) \cong H^*(W)$ is a ring isomorphism, there is no quantum deformation involved and we obviously have $e(V)=\alpha \in H^*(W)$. The second map, $\Gamma_*$, is induced by the Lagrangian correspondence, and needs to be understood in the context of quilted Floer theory. We refer the reader to \cite[Section 4.1]{perutzgysin} for a more refined description of the maps (in the setting of Hamiltonian Floer theory).
\endproof

\begin{cor}\label{inequalityrank}
	The Gysin sequence produces the rank inequality \begin{align}
	hf(L,L'):=\rank\HF(L,L') \leq 2 \rank\HF(K, K').
	\end{align}
\end{cor}

In Section \ref{freegenplumb}, we will need to compare functors induced by (projective and Dehn) twists to the identity functor. In particular, it will be necessary to distinguish objects of $\Fuk(W)$ with their image under the twist functors. The following lemma gives a helpful criterion.

\begin{lemma}\label{lagisotopy}
	Let $K', \overline{K}'$ be quasi-isomorphic objects in $\Fuk(W)$. Then the maps $f_1\colon\CF^*(K, K') \xlongrightarrow{\alpha\cup \cdot} \CF^{*+k+1} (K, K')$ and $f_2\colon\CF^*(K, \overline{K}') \xlongrightarrow{\alpha \cup \cdot} \CF^{*+k+1} (K, \overline{K}')$ have quasi-isomorphic mapping cones.
\end{lemma}
\proof

Consider the long exact sequences associated to the mapping cones of the cup product maps
$f_1\colon \CF^{*}(K,K')\lra \CF^{*+k+1}(K,K')$, $f_2\colon \CF^{*}(K, \overline{K}')\lra \CF^{*+k+1}(K, \overline{K}')$.

These sequences fit in a diagram of the shape

\begin{tikzcd}
{\CF^*(K,K')} \arrow[d] \arrow[rr, "f_1=\alpha \cup \cdot "] &  & {\CF^{*+k+1}(K,K')} \arrow[d] \arrow[rr] &  & Cone(f_1) \arrow[rr] \arrow[d] & & {\CF^{*+k+1}(K,K')} \arrow[d] \\
{\CF^*(K,\overline{K}')} \arrow[rr, "f_2=\alpha \cup \cdot "]                    &  & {\CF^{*+k+1}(K,\overline{K}')} \arrow[rr]           &  & Cone(f_2) \arrow[rr]           & & {\CF^{*+k+1}(K,\overline{K}')} 
\end{tikzcd}

Since $K', \overline{K}'$ are quasi-isomorphic objects in $\Fuk(W)$, there is a characteristic element $\eta \in \HF(K', \overline{K}')$ which induces isomorphism $\HF(K,K')\to \HF(K,\overline{K}')$ (the Floer product with $\eta$, see \cite[(8k)]{seidelbook}). Therefore, the vertical maps $\CF^*(K, K')\lra \CF^*(K, \overline{K}')$ are well-defined, and they are quasi-isomorphism. By the five lemma, it follows that the mapping cones $Cone(f_1)$ and $Cone(f_2)$ are also quasi-isomorphic.

\qed

\section{Free groups generated by projective twists}\label{freegenplumb}
In this section we apply the Hopf correspondence to prove our first result about products of projective twists.

Consider a transverse plumbing $W:=T^*\APn\#_{pt}T^*\APn$ of cotangent bundles of projective spaces, for $\A \in \{ \R, \C \}$. Then, the main result of this section (Theorem \ref{mainthm}) shows that the Lagrangian cores of the plumbing define two projective twists which generate a free subgroup of $\pi_0(\Symp_{ct}(W))$. In fact, Theorem \ref{mainthm} is a stronger statement that holds not only for transverse plumbings, but also more generally for \emph{clean} plumbings along sub-projective spaces (see Definition \ref{defidentify}).

For the proof, we use the Hopf correspondence to reduce the statement of Theorem \ref{mainthm} into a statement about Dehn twists, and apply \emph{Keating's free generation result} (Theorem \ref{keating}) for Dehn twists (\cite{keatingfree}).

As a corollary, we show that there are infinitely many Lagrangian isotopy classes of embeddings $\CP \hookrightarrow W$ which are smoothly isotopic, but pairwise not Lagrangian isotopic.

\subsection{Clean Lagrangian plumbing}\label{plumbing}
We first recall a construction from \cite[Appendix A]{abouzaidplumbing} of clean \emph{Lagrangian plumbing} of two Riemannian manifolds $Q_1,Q_2$ along a submanifold $B \subset Q_i$, $i=1,2$.
Fix three closed smooth manifolds $B, Q_1, Q_2$, for each $i=1,2$ an embedding $B \hookrightarrow Q_i$ and an isomorphism $\varrho: \nu_{B/Q_1} \to \nu_{B/Q_2}^*$ from the normal bundle $\nu_{B/Q_1}$ to the conormal bundle $\nu_{B/Q_2}^*$.

Pick a Riemannian metric on $B$, an inner product and a connection on $\nu_{B/Q_1}\cong \nu_{B/Q_2}^*$ (which induces an inner product and connection on $\nu_{B/Q_2}\cong \nu_{B/Q_1}^*$). This data induces a metric on the total spaces $\nu_{B/Q_i}$, and a neighbourhood $U_i$ of $B \subset Q_i$ can be identified with a disc subbundle $D_{\eps}\nu_{B/Q_i}$ of radius $\eps>0$. With this identification we write $x \in U_i$ as $x=(a,b)$ for $b \in B$, $a \in D_{\eps}(\nu_{B/Q_i})_b$ (the fibre over $b$). 

For each $x =(a,b) \in U_i$, the connection gives a decomposition of the fibres $T_x^*Q_i \cong T_b^*B \oplus
(\nu_{B /Q_i}^*)_b$. We get an identification of a neighbourhood of $B \subset T^*Q_i$ as \begin{align}\label{nhdb}
D_{\eps}\nu_{B/Q_i} \oplus D_{\eps}T^*B\oplus D_{\eps}\nu_{B/Q_i}^*.
\end{align}

Let $V_i$ be a neighbourhood of $Q_i \subset T^*Q_i$ which in \eqref{nhdb} coincides with $D_{\eps}T^*B \oplus D_{\eps}\nu_{B/Q_i}^*$ over $U_i \cong D_{\eps}\nu_{B/Q_i}$.

\begin{definition}\label{defidentify}
	\begin{enumerate}
		\item 	As a smooth manifold, the clean plumbing of $Q_1, Q_2$ along $B$, denoted by $M:=D_{\eps}(T^*Q_1 \#_{B}T^*Q_2)$ is defined by gluing $V_1$ to $V_2$ along $D_{\eps}\nu_{B/Q_1} \oplus D_{\eps}T^*B\oplus D_{\eps}\nu_{B/Q_1}^* \subset V_1$ identified with $D_{\eps}\nu_{B/Q_2}^* \oplus D_{\eps}T^*B\oplus D_{\eps}\nu_{B/Q_2}$ via $(\varrho, id_{T^*B},-\varrho^*)$. Its Liouville completion will be denoted by $T^*Q_1\#_{B}T^*Q_2$.
		\item The plumbing construction inherits an exact symplectic structure, since the identification maps of $1.$ preserve the canonical structures on $D^*Q_i$. Let $Z_i$ be the standard radial Liouville vector field on $V_i$. We define a Liouville vector field $Z$ on the plumbing by letting $Z=\rho_1Z_1+\rho_2Z_2$, for smooth functions $\rho_i: M \to [0,1]$ supported on $V_i$ such that $\rho_1+\rho_2=1$. This endows $M$ with the structure of an exact symplectic manifold.
	\end{enumerate}
\end{definition}

In the next sections, we will apply this plumbing construction to cotangent bundles of projective spaces and spheres. We will work with (ungraded) Floer cohomology groups
$\HF(Q_1,Q_2,k)$ where $k$ is a coefficient field of characteristic two. Note that by exactness of the Lagrangians and the manifold $W$, the Floer differentials in $\CF(Q_i, Q_i)$ vanish, and as $Q_1$ and $Q_2$ intersect cleanly along $B$, there is an isomorphism $\HF(Q_1,Q_2)\cong H^*(B)$ (\cite{pozniak}).

\subsection{Theorem \ref{mainthm}}\label{sectionproof}

We now prove the main theorem of this section.

\thmmainthm*

\begin{rmk}\label{rmklowdimension}
	The case $W:=T^*\C\PP^1_1\#_{pt}T^*\C\PP^1_2$ can be deduced from the existing literature, by considering $X$ as an $A_2$-configuration and the isotopies $\tau_{\C\PP^1_i}\simeq \tau_{S^2_i}^2$ of Remark \ref{sympsmooth}. There is a homomorphism (\cite[Proposition 8.4]{seidelknotted}) $\rho: Br_3 \to \pi_0(\Symp_{ct}(W))$ sending the generators of the braid group $\sigma_i$ to $\rho(\sigma_i)=\tau_{S^2_i}$, $i=1,2$. 
	The associated homomorphism $\hat{\rho}: Br_3 \to \Auteq(\Fuk(W))$ fits in the diagram
	\begin{align}
	\xymatrixcolsep{5pc} \xymatrix{
		Br_3  \ar[r]^{\rho}\ar[rd]^{\hat{\rho}}  &
		\pi_0(\Symp_{ct}(W)) \ar[d] &\\
		& Auteq(\Fuk(W))}.
	\end{align}
	and its injectivity (\cite{stbraid}) implies the injectivity of $\rho$. Then, as $\langle \sigma_1^2, \sigma_2^2 \rangle \cong Free_2$, it follows $\langle \tau_{K_1}, \tau_{K_2} \rangle \cong Free_2$. \\ 
	Also note that $\rho$ is in fact an isomorphism (\cite{weiwei}), so $\pi_0(\Symp_{ct}(T^*S^2\#_{pt} T^*S^2)) = Br_3$.
\end{rmk}

Theorem \ref{mainthm} takes inspiration from Keating's free generation result for Dehn twists in Liouville manifolds.

\begin{thm}\cite[Theorem 1.1 and 1.2]{keatingfree}\label{keating}
	Let $(Y, \Omega)$ be a Liouville manifold of dimension greater than $2$, and $L,L' \subset Y$ be two Lagrangian spheres satisfying $\rank \HF(L,L') \geq 2$, and such that $L,L'$ are not quasi-isomorphic in $\Fuk(Y)$. The Dehn twists $\tau_L, \tau_{L'}$ generate a free subgroup of $\pi_0(\Symp_{ct}(Y))$, and the associated functors $T_L, T_{L'} \in \Auteq(\Fuk(Y))$ generate a free subgroup of $\Auteq(\Fuk(Y))$.
\end{thm}

Keating proves the geometric part of Theorem \ref{mainthm} by making a categorical detour, first proving that the associated functors $T_L, T_{L'} \in \Auteq(\Fuk(Y))$ induced by the Dehn twists generate a free subgroup of $\Auteq(\Fuk(Y))$ so that the composition
\begin{align}
Free_2 \to \pi_0 \Symp_{ct}(Y) \to \Auteq(\Fuk(Y))
\end{align}
is injective.

By identifying a Dehn twist with its associated functor, Keating exploits the algebraic properties of the latter to arrive at the following rank inequalities (which are central in her final proof).

\begin{lemma}\cite[Lemma 8.1]{keatingfree}\label{keatinglemma}
	Let $\tilde{L}, L, L' \subset Y$ be Lagrangians such that $\tilde{L}$ is a sphere, $\tilde{L} \ncong L$ in the Fukaya category, and $hf(\tilde{L},L):=\rank(\HF(\tilde{L}, L')) \geq 2$. Then, for all $n \neq 0$:\begin{align}\label{rank1}
	hf(\tilde{L}, L') > hf(L, L') \Rightarrow hf(\tilde{L}, \tau_{\tilde{L}}^n(L')) < hf(L, \tau_{\tilde{L}}^n(L')).
	\end{align}
\end{lemma}

\begin{lemma}\cite[Claim 8.2]{keatingfree}\label{keatingclaim}
	Let $L, L' \subset Y$ be two Lagrangian spheres in an exact symplectic manifold as in Theorem \ref{keating} satisfying $hf(L,L')=2$. Then for all $m \neq 0$ we have 
	\begin{align}\label{rank11}
	hf(L', L)=hf(L', \tau_{L'}^mL)< hf(L, \tau_{L'}^mL).
	\end{align}
	
\end{lemma}

We will apply these inequalities to Lagrangian spheres obtained from the Hopf correspondence, to produce similar results for projective twists and prove Theorem \ref{mainthm}.

\subsubsection{Strategy}
The plumbing $(W, \omega)$ and its real/complex projective Lagrangian cores $K_1,K_2 \subset W$ satisfy the cohomological conditions $\eqref{assptreal}/ \eqref{assptcx}$.

In the case in which $W$ is a transverse plumbing (which retracts to the wedge sum of the two spheres) there is a ring isomorphism ($\A \in \{\R, \C \}$, $R \in \{ \Zmod, \Z\}$)
$$\tilde{H}^*(W;R)\cong \tilde{H}^*(K_1;R)\oplus \tilde{H}^*(K_2;R)\cong \tilde{H}^*(\APn;R)\oplus \tilde{H}^*(\APn;R),$$
so that it is immediate to see the existence of a class $\alpha=(\alpha_1,\alpha_2) \in H^{k}(W;R)$ restricting to $K_1$ and $K_2$ to the generator of $H^*(\APn;R)$, for $k\in \{1,2\}$. For a clean plumbing along a linearly embedded sub-projective space $\A\PP^{\ell}$, this restriction property still holds because the `difference' map of the Mayer-Vietoris sequence is always zero.

By Propositions \ref{cxlocalhopf}, \ref{realocalhopf}, the cohomological condition ensures the existence of a Liouville manifold $(Y, \Omega) \lra (W, \omega)$ and a Hopf correspondence $\Gamma \subset W^-\times Y$ that gives rise to associated Lagrangian spheres $S^m\cong L_i=K_i\circ \Gamma \subset Y$, $i=1,2$ and commuting diagrams of twist functors \eqref{categoricaltwistplumbing}. 
Then, given a product (a word in $\tau_{K_1}, \tau_{K_2}$) $\varphi \in \Symp_{ct}(W)$ of projective twists, the Hopf correspondence yields a corresponding product of Dehn twists (a word in $\tau_{L_1}, \tau_{L_2}$) $\phi \in \Symp_{ct}(Y)$.

In the real projective case, the geometric statement of Theorem \ref{mainthm} can be obtained by an isotopy-lifting argument using the geometric diagrams of Section \ref{hopfbundle}
(the strategy adopted in Section \ref{projectiveproduct}). Assuming the projective twists do satisfy a relation, this procedure lifts the isotopy to $\Symp_{ct}(Y)$, producing a relation between Dehn twists, which cannot hold by Keating's theorem. However, this geometric argument does not give a statement at the level of Fukaya categories, for which the use of the Hopf correspondence at the level of functors in $\Auteq(\Fuk(W))$ is necessary (Sections \ref{commutingfunctor}, \ref{sectiongysin}).

The spheres $L_1, L_2 \cong S^m$ intersect cleanly along a sub-sphere $S^{r}$, for the tuple $(\mathbb{A}, m, r) \in \{ (\R, n, \ell), (\C, 2n+1, 2\ell+1) \}$ ($n,\ell \in \N^*$), and as noted before, $\HF(L_1,L_2;k)\cong H^*(S^r;k)$. 
Since $L_i\subset Y$ are exact spheres $\HF(L_i,L_i;k)\cong H^*(S^m;k)$ (\cite{floer}) and therefore
\begin{equation}\label{rank2}
\rk\HF(L_1,L_2)=\rk\HF(L_i,L_i)=2, \ i=1,2.
\end{equation}

In the following sections we will study the ranks of the Floer cohomology groups $\HF(\cdot,\varphi(\cdot))$ and show that there is always a Lagrangian $\hat{K}\subset W$ such that \begin{equation}\label{notisomorphic}
\HF(\hat{K},\hat{K}) \not \cong   \HF(\hat{K},\varphi(\hat{K})).
\end{equation}
As a result, $\hat{K}$ and $\varphi(\hat{K})$ are not quasi-isomorphic objects in $\Fuk(W)$, and therefore the functor induced by $\varphi$ cannot be isomorphic to the identity in $\Auteq\Fuk(W)$.
This will also rule out the possibility of $\varphi$ being isotopic to the identity in $\Symp_{ct}(W)$.

We prove \eqref{notisomorphic} by applying the rank inequalities \eqref{rank1}, \eqref{rank11} to Lagrangian spheres in $(Y, \Omega)$ obtained via the correspondence $\Gamma$, in combination with the symplectic Gysin sequence associated to $\Gamma$ (Corollary \ref{inequalityrank}).

In the following section we reiterate a part of Keating's proof of Theorem \ref{keating} using the rank inequalities \eqref{rank1}, \eqref{rank11} that hold for the word of Dehn twists $\phi \in \langle \tau_{L_1}, \tau_{L_2} \rangle$ associated to $\varphi$ via the Hopf correspondence. This will clarify the methods used in proving the analogous statement for projective twists (namely, Theorem \ref{mainthm}) in Section \ref{unifiedproof}.

\subsubsection{Associated word of Dehn twists.}\label{singletw}

Let $\varphi \in \langle \tau_{L_1}, \tau_{L_2} \rangle \subset \Symp_{ct}(W)$ be a word of projective twists as in the statement of Theorem \ref{mainthm}. Consider the Hopf correspondence $\Gamma \subset W^-\times Y$ and the associated word of Dehn twists $\phi \in \Symp_{ct}(Y)$ as above. In this section we replicate the last steps in Keating's proof of the injectivity of the homomorphism $$Free_2\lra \Auteq(\Fuk(W)).$$

We first make the following observation about a word of twists and its conjugates.

\begin{lemma}\label{corollaryconjugate}
	Let $\phi \in \Symp_{ct}(Y)$ a symplectomorphism which has the shape of a global conjugate, i.e $\phi=\psi^{-1}\phi'\psi$, for $\psi, \phi' \in \Symp_{ct}(Y)$ not isotopic to the identity. Then there is a closed Lagrangian $\tilde{L}\subset Y$ such that $\HF(\tilde{L}, \phi(\tilde{L})) \not \cong \HF(\tilde{L}, \tilde{L})$ if and only if $\phi'$ satisfies $\HF(\hat{L}, \phi'(\hat{L})) \not \cong \HF(\hat{L}, \hat{L})$ for some closed Lagrangian $\hat{L}\subset Y$.
\end{lemma}
\proof
Assume there is a Lagrangian $\hat{L} \subset Y$ such that $\HF(\hat{L}, \phi'(\hat{L}))\not \cong \HF(\hat{L}, \hat{L})$. Then by invariance of Floer cohomology under symplectomorphisms $\HF(\psi^{-1} \hat{L}, \psi^{-1}\phi'\psi(\psi^{-1} \hat{L})) \cong \HF(\psi^{-1} \hat{L}, \psi^{-1}(\phi' \hat{L}) \cong \HF( \hat{L}, \phi' \hat{L}) \not \cong \HF(\hat{L}, \hat{L})$, so for $\tilde{L}:=\psi^{-1}(\hat{L})$ we have $\HF(\tilde{L}, \phi(\tilde{L})) \not \cong \HF(\tilde{L}, \tilde{L})$. The other direction is similar.
\qed 

We will apply the above lemma to a word of Dehn twists $\phi \in \Symp_{ct}(Y)$ which in its reduced has the shape of global conjugate, i.e a word $\phi=\psi^{-1}\phi'\psi$, for two reduced words $\psi, \phi' \in \langle \tau_{L_1}, \tau_{L_2} \rangle \subset \Symp_{ct}(Y)$ not isotopic to the identity. Then the lemma shows that it is always possible to switch between $\phi$ and its conjugate, as the correct choice of Lagrangian keeps track of the Floer cohomological action of the original word.

Without loss of generality, we can therefore use this conjugation argument to restrict the focus on reduced words $\phi \in \langle \tau_{L_1}, \tau_{L_2} \rangle$ which are either (for $i,j \in \{ 1,2 \}$):

\begin{enumerate}
	\item A power of a single Dehn twist, i.e $\phi=\tau_{L_i}^s$, $s\in \Z^*$ (Lemma \ref{lemmadistinct}).
	\item A word starting with a power of $\tau_{L_i}$ and ending in a power of $\tau_{L_j}$, $i\neq j$ (Lemma \ref{mixedtwist}).
\end{enumerate}

\begin{lemma}\label{lemmadistinct}
	Let $\phi=\tau_{L_i}^{s}\in \Symp_{ct}(Y)$ be a reduced word of Dehn twists which is a power of a single Dehn twist, $i,j \in \{ 1,2\}$, $s \in \Z^*$. The associated functor is not isomorphic to the identity in $\Auteq(\Fuk(Y))$, so in particular, $\phi$ cannot be isotopic to the identity in $\Symp_{ct}(Y)$.
\end{lemma}

\proof 

We show that there exists a closed Lagrangian $\hat{L} \subset Y$ such that $$\HF(\hat{L}, \phi (\hat{L})) \ncong \HF(\hat{L}, \hat{L}).$$
For $\phi=\tau_{L_i}^s$, a possible candidate is given by $\hat{L}=L_j$, $i,j \in \{1,2\} $, $i\neq j$.

Namely, the rank inequality stated by Lemma \ref{keatingclaim}, gives
$$2=hf(L_j,L_j)=hf(L_i,L_j)=hf(L_i, \tau_{L_i}^s L_j) < hf(L_j, \tau_{L_i}^s L_j).$$
\qedhere
\begin{rmk}
	The geometric result of the above lemma can also be proven independently from Keating's results, as a corollary to Theorem \ref{bgz} (see Section \ref{alternativebgz}).
\end{rmk}

\begin{lemma}\label{mixedtwist}
	
	Let $\phi \in \langle \tau_{L_1}, \tau_{L_2} \rangle \subset \Symp_{ct}(Y)$ be a reduced word of Dehn twists around the Lagragian (spherical) cores which is a product where the first and last factors are powers of distinct Dehn twists. Then the functor associated to $\phi$ is not isomorphic to the identity in $\Auteq(\Fuk(Y))$ so in particular $\phi$ cannot be isotopic to the identity in $\Symp_{ct}(Y)$.
\end{lemma}
\proof We show that there is a closed Lagrangian $\hat{L} \subset Y$ such that $\HF(\hat{L}, \phi\hat{L}) \not \cong \HF(\hat{L}, \hat{L})$. 

We can assume w.l.o.g that the first factor of $\phi$ to be a power of $\tau_{L_2}$ and the last is a power of $\tau_{L_1}$ (otherwise consider $\phi^{-1}$), so that we have a word of the following shape

\begin{align}\label{productshape}
\phi= \tau_{L_2}^{b_k} \tau_{L_1}^{a_k} \cdots \tau_{L_2}^{b_1}\tau_{L_1}^{a_1}, \  a_i, b_i \in \Z^* \text{ for } 1 \leq i \leq k.
\end{align}

In the case we are considering, we have $hf(L_1,L_2)=2$. Apply Lemma \ref{keatingclaim} to get \begin{align*}
2=hf(L_1,L_1)=hf(L_2,L_1)=hf(L_2, \tau_{L_2}^{b_1} \tau_{L_1}^{a_1}L_1) < hf(L_1, \tau_{L_2}^{b_1}\tau_{L_1}^{a_1}L_1).
\end{align*}

Now apply Lemma \ref{keatinglemma} (with $n=a_2$, $\tilde{L}=L_1, L=L_2$ and $L'=\tau_{L_2}^{b_1} \tau_{L_1}^{a_1}L_1$)
and get 
\begin{align*}
hf(L_1, \tau_{L_2}^{b_1}\tau_{L_1}^{a_1}L_1)=hf(L_1,\tau_{L_1}^{a_2} \tau_{L_2}^{b_1}\tau_{L_1}^{a_1}L_1) < hf(L_2,\tau_{L_1}^{a_2} \tau_{L_2}^{b_1}\tau_{L_1}^{a_1}L_1).
\end{align*}
Apply Lemma \ref{keatinglemma} again (with $n=b_2$, $\tilde{L}=L_2, L=L_1$ and $L'=\tau_{L_1}^{a_2}\tau_{L_2}^{b_1} \tau_{L_1}^{a_1}L_1$)
\begin{align*}
hf(L_2,\tau_{L_1}^{a_2} \tau_{L_2}^{b_1}\tau_{L_1}^{a_1}L_1)= hf(L_2, \tau_{L_2}^{b_2}\tau_{L_1}^{a_2} \tau_{L_2}^{b_1}\tau_{L_1}^{a_1}L_1) < hf(L_1, \tau_{L_2}^{b_2}\tau_{L_1}^{a_2} \tau_{L_2}^{b_1}\tau_{L_1}^{a_1}L_1).
\end{align*}

Continue to apply Lemma \ref{keatinglemma} iteratively until the final step
\begin{align*}
hf(L_2, \tau_{L_2}^{b_k} \tau_{L_1}^{a_k} \cdots \tau_{L_2}^{b_1} \tau_{L_1}^{a_1}L_1)< hf(L_1, \tau_{L_2}^{b_k} \tau_{L_1}^{a_k} \cdots \tau_{L_2}^{b_1} \tau_{L_1}^{a_1}L_1).
\end{align*}

Then \begin{align*}
hf(L_1, \tau_{L_2}^{b_k} \tau_{L_1}^{a_k} \cdots \tau_{L_2}^{b_1} \tau_{L_1}^{a_1}L_1)> 2+ 2k-1=2k+1.
\end{align*}
So setting $\hat{L}=L_1$ we have $\HF(\hat{L}, \phi(\hat{L})) \not \cong \HF(\hat{L}, \hat{L})$. \qed

\begin{cor}\label{rankgroups}
	Let $\phi \in  \langle \tau_{L_1}, \tau_{L_2} \rangle \subset \Symp_{ct}(Y)$ be a word of Dehn twists that is a product of the shape \eqref{productshape}.
	Then there is a Lagrangian $\hat{L} \subset Y$ such that $$\lim_{s \to \infty } \rank HF^*(\hat{L}, \phi^s (\hat{L})) = \infty.$$
\end{cor}
\proof
Let $\phi$ be of the shape \eqref{productshape}. Then $$\phi^s= (\tau_{L_2}^{b_k} \tau_{L_1}^{a_k} \cdots \tau_{L_2}^{b_1}\tau_{L_1}^{a_1})(\cdots )(\cdots)(\tau_{L_2}^{b_k} \tau_{L_1}^{a_k} \cdots \tau_{L_2}^{b_1}\tau_{L_1}^{a_1})$$
has `factor length' $k\cdot s$ (in the sense of \eqref{productshape}). By the proof of Lemma \ref{mixedtwist}, the rank of $ \HF(L_1, \phi(L_1))$ depends on the number $k \in \N$ appearing in the factor decomposition of $\phi$. 
Therefore $$hf(L_1, \phi^s(L_1))>2ks+1$$
so we can set $\hat{L}:=L_1$.
\qed

\subsubsection{Proof}\label{unifiedproof}

We now go back to the original word $\varphi \in \langle \tau_{K_1}, \tau_{K_2} \rangle \subset \Symp_{ct}(W)$ of projective twists in the statement of Theorem \ref{mainthm}, and we show that it cannot induce the identity functor in $\Auteq(\Fuk(W))$. Lemma \ref{corollaryconjugate} holds for any symplectomorphism, so by the same conjugation argument explained before, we can focus the attention on words that are either (for $i,j \in \{ 1,2 \}$)
\begin{enumerate}
	\item A power of a single twist $\varphi=\tau_{K_i}^s$, $s\in \Z^*$.
	\item A mixed product of the shape $\varphi:= \tau_{K_i}^{b_k} \tau_{K_j}^{a_k} \cdots \tau_{K_i}^{b_1}\tau_{K_j}^{a_1} \in \Symp_{ct}(W)$, $i\neq j$, $a_m, b_m \in \Z^*$, $1\leq m \leq k$.
\end{enumerate}

\begin{prop}\label{propsingle}
	Let $\varphi=\tau_{K_i}^s \in \langle \tau_{K_1}, \tau_{K_2} \rangle \subset \Symp_{ct}(W)$ be a reduced word of projective twists which is a power of a single twist, $i \in \{1,2 \}$, $s\in \Z^*$. Then the functor induced by $\varphi$ is not isomorphic to the identity in $\Auteq(\Fuk(W))$, and in particular $\varphi$ cannot be isotopic to the identity in $\Symp_{ct}(W)$.
\end{prop}
\proof  

Let $\varphi =\tau_{K_i}^{s} \in \Symp_{ct}(W)$, $s \in \Z^*$. Assume by contradiction that the functor induced by $\varphi$ (still denoted by $\varphi$) is isomorphic to the identity, so that any Lagrangian $\hat{K} \subset W$ is quasi-isomorphic, as an obejct of $\Fuk(W)$, to $\varphi(\hat{K})$. 

By Lemma \ref{lagisotopy}, there is a quasi-isomorphism of the mapping cones of the cup product maps
$f_1\colon \CF^{*}(\hat{K},\hat{K})\lra \CF^{*+k+1}(\hat{K},\hat{K})$ and $f_2\colon \CF^{*}(\hat{K}, \varphi(\hat{K}))\lra \CF^{*+k+1}(\hat{K}, \varphi(\hat{K}))$ (we are considering ungraded Floer cohomology groups so technically the degrees are irrelevant here).
Therefore, by the exact triangle of Lemma \ref{gysinper}, if $\hat{L} \subset Y$ is the Lagrangian lift of $\hat{K}$ through the correspondence $\Gamma$, and $\phi \in \Symp_{ct}(Y)$ the symplectomorphism associated to $\varphi$, then 
$\HF(\hat{L}, \hat{L})\cong \HF(\hat{L}, \phi (\hat{L}))$. 

So if we set $\hat{K}:=K_j$, $j\neq i$, by assumption we have $\HF(K_j, \varphi(K_j)) \cong \HF(K_j,K_j)$ and the above argument yields $\HF(L_j, \phi(L_j))=\HF(L_j, \tau_i^s(L_j)) \cong \HF(L_j, L_j)$ which is clearly in contradiction to (the proof of) Lemma \ref{lemmadistinct} (according to which these two groups have distinct ranks). Hence, $\varphi$ cannot be isomorphic to the identity functor in $\Auteq(\Fuk(W))$. \qed

\begin{prop}\label{propmixed}
	Let $\varphi \in \langle \tau_{K_1}, \tau_{K_2} \rangle \subset \Symp_{ct}(W)$ be a reduced word of projective twists around the Lagragian cores which is a product where the first and last factors are powers of distinct projective twists. Then the functor induced by $\varphi$ is not isomorphic to the identity in $\Auteq\Fuk(W)$, so in particular $\varphi$ is not isotopic to the identity in $\Symp_{ct}(W)$.
\end{prop}

\proof 
By the analogous discussion as in the proof of Lemma \ref{mixedtwist}, it is enough to prove the statement for a word whose reduced form is of the shape 
\begin{align}\label{productprojshape}
\varphi:= \tau_{K_2}^{b_k} \tau_{K_1}^{a_k} \cdots \tau_{K_2}^{b_1}\tau_{K_1}^{a_1} \in \Symp_{ct}(W), \ a_m, b_m \in \Z^*, \ 1\leq m \leq k.
\end{align}

Denote the product of twist functors induced by \eqref{productprojshape} also by $\varphi \in \Auteq(\Fuk(W))$. By iteratively using commutativity of the functors in diagram \eqref{categoricaltwistplumbing}, one can define the corresponding composition of (Dehn) twist functors $\phi  \in \Auteq(\Fuk(Y))$, which by Theorem \ref{keating} cannot be isomorphic to the identity functor.

Moreover, Corollary \ref{rankgroups} shows that not only is $HF(L_1, \phi(L_1))$ non-isomorphic to $HF(L_1, L_1)$, but also that $\lim_{s \to \infty} hf(L_1, \phi^{s}(L_1))\to \infty$. 

The Lagrangian $L_1=\Gamma \circ K_1\subset Y$ is the Lagrangian associated to $K_1$ via the Hopf correspondence, and the symplectic Gysin exact sequence (Corollary \ref{inequalityrank}) applied to the Hopf correspondence gives the inequality \begin{align}\label{inequality}
hf(L_1, \phi(L_1)) \leq 2 hf(K_1, \varphi(K_1)),
\end{align}

which implies that the rank $hf(K_1, \varphi^{s}(K_1))$ also grows at least linearly with $s$.
\qed

\begin{cor}
	Let $\varphi \in \langle \tau_{K_1}, \tau_{K_2} \rangle \subset \Symp_{ct}(W)$ be a word of projective twists of the shape \eqref{productprojshape}. Then there is a Lagrangian $\hat{K} \subset W$ such that $$\lim_{s \to \infty } \rank HF^*(\hat{K}, \varphi^s (\hat{K})) = \infty.$$
\end{cor}
\qed

Finally, we can summarise the proof of Theorem \ref{mainthm}.

\proof[Proof of Theorem \ref{mainthm}]
Let $\varphi \in \langle \tau_{K_1}, \tau_{K_2} \rangle \subset \Symp_{ct}(W)$ a word in the projective twists along the Lagrangian cores of $W$. \begin{enumerate}
	
	\item If the word has the shape $\varphi =\tau_{K_i}^{s} \in \Symp_{ct}(W)$, $i\in \{ 1,2 \}$, $s \in \Z^*$, then its induced functor is not isomorphic to the identity in $\Auteq(\Fuk(W))$ by Proposition \ref{propsingle}.

	\item If the word has the shape $\varphi:= \tau_{K_i}^{b_k} \tau_{K_j}^{a_k} \cdots \tau_{K_i}^{b_1}\tau_{K_j}^{a_1} \in \Symp_{ct}(W)$, $i, j \in \{1,2\}$ $i\neq j$, $a_m, b_m \in \Z^*$, $1\leq m \leq k$, then its induced functor is not isomorphic to the identity in $\Auteq(\Fuk(W))$ by Proposition \ref{propmixed}.
	
	\item If $\varphi$ has any other form, then it must be a conjugate of a word of shape $1.$ or $2.$ and hence the induced functor is not isomorphic to the identity by Corollary \ref{corollaryconjugate}.
\end{enumerate}
\endproof

\subsection{Knotted Lagrangian projective spaces}

The phenomenon that a single (smooth) isotopy class of submanifolds contains infinitely many Lagrangian isotopy classes is called Lagrangian ``knottedness'' (\cite{seidelknotted, evanslags,
	hindknot,liwu, weiwei}). Often, the quest for knottedness is intimately related to the study of isotopy classes of Dehn twists.

In the plumbing of spheres $L_i\cong S^2$, $Y:=T^*L_1\#_{pt}T^*L_2$, we know that for any $r \in \Z$, $\tau_{L_2}^{2r}(L_1)$ is smoothly isotopic to the identity, but not symplectically; as first shown by Seidel (\cite[Theorem 1.1]{seidelknotted}), none of the powers $\tau_{L_2}^{2r}(L_1)$ are Hamiltonian isotopic. 
Our results yield the analogue for plumbing of complex projective spaces (of any dimension).
\begin{cor}\label{corknotted}
	Let $W:=T^*\CP \#_{\C\PP^{\ell}} T^*\CP$ be a clean plumbing along a projective subspace $\C\PP^{\ell}\subset \CP$.
	Each Lagrangian core $K_i \cong \CP$ of $W$ defines a smooth isotopy class which contains infinitely many symplectic isotopy classes of Lagrangian projective spaces.
\end{cor}
\proof Let $K_1, K_2 \cong \CP \subset W$ be the two Lagrangian cores of the plumbing. For $i, j \in \{ 1,2 \}$, $i\neq j$, define the element $\varphi:=\tau_{K_i}\tau_{K_j}$.
Then by Proposition \ref{propmixed} we have

\begin{align}
\lim\limits_{s \to \infty } \rank \HF(K_j, \varphi^s (K_j))= \infty 
\end{align} 
which in particular means that $\varphi^{s_a}(K_j)$ is not Lagrangian isotopic to $\varphi^{s_b}(K_j)$ for any $s_a\neq s_b$, despite being smoothly isotopic (by Theorem \ref{cpntrivial}).
\endproof

\begin{rmk} 
The low dimensional case $n=1$ corresponds to a transverse plumbing of spheres $W:=T^*L_1\#_{pt}T^*L_2$, $L_i\cong S^2$. In that case, the symplectic mapping class group $\pi_0(\Symp_{ct}(W))$ is generated by the Dehn twists $\tau_{L_1}$ and $\tau_{L_2}$ (see Remark \ref{rmklowdimension}). Moreover, Hind (\cite{hindknot}) proved that for any Lagrangian sphere $L \subset W$, there is a word $\tau \in \langle \tau_{L_1}, \tau_{L_2}\rangle$ such that $\tau(L)$ is isotopic to one of the cores $L_1$ or $L_2$. 
		
\end{rmk}

\section{Positive products of twists in Liouville manifolds}\label{dehngeneral}\label{generalsdt}\label{alternativeproof}

The present section covers our results about products of positive powers of Dehn and projective twists. 

In the first part, Section \ref{generalsdt}, we analyse products of (positive powers of) Dehn twists. 
We re--prove a theorem by Barth--Geiges--Zehmisch (Theorem \ref{bgz}) asserting that in a Liouville manifold $(M, \omega)$, no product $\phi\in \Symp_{ct}(M)$ of positive powers of Dehn twists can be symplectically isotopic to the identity.
We provide an alternative proof that was suggested by Paul Seidel. Based on symplectic Picard--Lefschetz theory, the argument for the proof relies on a count of pseudoholomorphic sections of a Lefschetz fibration constructed from the data given by $\phi$ and the Lagrangian spheres associated to the Dehn twists.

Using similar tools, we then prove Theorem \ref{relativeversion1} (Section \ref{relativegeo}), which can be interpreted as a relative version of Theorem \ref{bgz}.
This states that a Liouville manifold $(M, \omega)$ containing Lagrangian spheres and a conical Lagrangian disc $T$ (Definition \ref{cylag}) intersecting one of the spheres transversely at a point cannot admit a positive product of Dehn twists preserving $T$ up to compactly supported symplectic isotopy. 

In Section \ref{generalisationprojective}, we explore the analogous questions for projective twists, by means of the tools developed in Section \ref{hopfcorr}. After setting the necessary conditions to ensure the existence of the Hopf correspondence, we use Theorem \ref{bgz} to prove a comparable result for real projective twists.

\subsection{Alternative proof of Theorem \ref{bgz}}\label{alternativebgz}

In this section, we start studying products of Dehn twists in Liouville manifolds, and we reprove the following theorem.
\thmbgz*

\begin{example}\label{rmkclosed}
	The exactness condition of Theorem \ref{bgz} is necessary, as the following examples show. 
	\begin{enumerate}[label=(\alph*)]
		
		\item Consider the $2$-torus $M:=T^2$, and let $a,b \subset M$ represent the longitude and meridian of $M$. Then the associated Dehn twists satisfy $(\tau_a\tau_b)^6=Id$ in $\pi_0(\Symp_{ct}(M))$.
		This is a classical result, see for example \cite{bmprimer} (see \cite[3.1]{aurouxmon} for the same example in a symplectic setting).
		
		\item Let $(M:=S^2\times S^2, \omega_{S^2}\oplus \omega_{S^2})$, and consider the antidiagonal $\overline{\Delta}:=\{ (x,y) \in S^2 \times S^2 | x+y=0  \}  \subset M$. Then the Dehn twist $\tau_{\overline{\Delta}}$ is symplectically isotopic to an involution $(x,y)\mapsto (y,x)$, which implies $\tau_{\Delta}^2=Id$ in $\pi_0(\Symp_{ct}(M))$ (see \cite[Example 2.9]{seidelectures}).
	\end{enumerate}
\end{example}

\begin{rmk}\label{puncturetorus1}
	
	\begin{enumerate}
		\item 
		The two dimensional case of Theorem \ref{bgz} (for a product of Dehn twists in a Riemann surface) is a consequence of \cite[Theorem 1.3]{smithmonodromy}.
		
		\item The outcome of Theorem \ref{bgz} is strictly geometric, and may not hold for the compact Fukaya category; we are not able to obtain information about the functors associated to the Dehn twists.
		Consider a punctured torus $M:=T^2\setminus \{ *\}$ (the same applies to a punctured genus $g$ surface), and the two (Lagrangian) circles $a$ and $b$, representatives of the homological generators. In the closed case, the composition $(\tau_{a}\tau_b)^6$ is isotopic to the identity by the example above. In the punctured torus, there is an isotopy $(\tau_a\tau_b)^6\simeq \tau_{d}$ , where $\tau_d$ is the Dehn twist along the boundary curve $d$ encircling the puncture (this is a consequence of the \emph{chain relation}, see \cite[Proposition 4.12]{bmprimer}).
		But since the support of $\tau_d$ is disjoint from any exact compact circle in $M$, the product $(\tau_a\tau_b)^6$ still acts as the identity on objects of the \emph{compact} Fukaya category $\Fuk(M)$.

	\end{enumerate}
\end{rmk}

The original proof of \cite{bgzfilling} relies on the theory of open book decompositions, whereas the proof below uses Picard--Lefschetz theory. To simplify notation we prove the version of the theorem where $(M, \omega= d\lambda_M)$ is a Liouville domain.

Let $\phi= \prod_{i=1}^{k} \tau_{L_{j_i}} \in \Symp_{ct}(M)$, $j_i \in \{ 1, \dots , m \}$ be the word in positive powers of Dehn twists in the given collection, and assume by contradiction that this product is compactly supported Hamiltonianly isotopic to the identity (recall $\phi$ is an exact symplectomorphism).

Build an exact Lefschetz fibration $\pi\colon(E, \Omega_E,\lambda_E)\lra (\C, \lambda_{\C})$ with smooth fibre the Liouville domain $(M, \lambda_M)$ identified over base point $z_*\in \C$, vanishing cycles the given Lagrangian spheres $(L_{j_1}, \dots , L_{j_k})$, and monodromy given by the product $\phi \in \Symp_{ct}(M)$ (see Section \ref{dtandlf}).
Let $j_{\C}$ be the standard complex structure on $\C$.

By assumption, the fibration built from the above data has monodromy isotopic to the identity via a compactly supported Hamiltonian isotopy with $\phi_0=\phi$ and $\phi_1=Id$. Then $\pi$ can be extended to a fibration $\hat{\pi}\colon \hat{ E} \lra \CPone$ as follows. 
Let $D_R\subset \C$ be a large circle of radius $R>0$ passing through $z_*$ and containing all the critical values. Define a fibration $E' \to D_{R+1}$ by extending $E|_{D_R}$ to a larger disc $D_{R+1}$ such that for $t\in [0,1]$, the monodromy around $D_{R+t}$ is $\phi_t \in \Symp_{ct}(M)$. Then $\hat{ E}$ is obtained after gluing $E'$ to a trivial fibration with fibre $(M,\omega)$ over a disc neighbourhood of the point at ``infinity'', $\hat{z} \in \C\cup \{ \infty \} \simeq \CPone$.

Moreover, as the symplectic connection around the fibre $\hat{\pi}^{-1}(\hat{z})$ is trivial, $\hat{\pi}\colon \hat{ E}\lra \CPone$ has the following properties.
\begin{enumerate}
	\item There is a closed (possibly degenerate) two form $\hat{\Omega}_{\hat{ E}}$ on $\hat{ E}$ satisfying $\hat{\Omega}_{\hat{ E}}|_{\hat{\pi}^{-1}(z)}=\Omega_E|_{\pi^{-1}(z)}$ for all $z\in \CPone \setminus \hat{z}$, 
	
	\item \label{nhooditem}
	A neighbourhood of the horizontal boundary $V\supset \partial^h \hat{E}$ can be trivialised as $V\cong \CPone \times M^{out}$, where $M^{out}\subset M$ is an open neighbourhood of the boundary of the smooth fibre.
\end{enumerate}

\begin{definition}
The set of almost complex structures compatible with $\hat{\pi}$, denoted by $\J(\hat{E},\hat{\pi},j)$, is defined as follows.
An element $\hat{J}\in \J(\hat{E},\hat{\pi},j)$ satisfies \begin{itemize}
	\item $D\hat{\pi}\circ \hat{J}=j\circ D\hat{\pi}$ where $j$ is the standard complex structure on $\CPone$,
	\item There is an integrable almost complex structure $J_0$ such that $\hat{J}=J_0$ in a neighbourhood of $\Crit(\hat{\pi})$,
	\item For all $z\in \CPone$, the restriction $J^{vv}:=\hat{J}|_{\hat{\pi}^{-1}(z)}$ is an almost complex structure of contact type compatible with the Liouville form $\lambda_M$, and its restriction to $V$ is isomorphic to a product $j\times J^{vv}$,

	\item $\hat{\Omega}_{\hat{ E}}(\cdot, \hat{J} \cdot)$ is symmetric and positive definite. 
\end{itemize}
\end{definition}

The form $\hat{\Omega}_{\hat{E}}$ can be modified to a symplectic form $\hat{\Omega}:=\hat{\Omega}_E+\hat{\pi}^*(\beta)$ that tames $\hat{J}$, for $\beta \in \Omega^2(\CPone)$ (similar to \cite[Lemma 2.1]{seideles}, \cite[Theorem 6.1.4]{mcduffsal}).

From now onwards, we fix a generic element $\hat{J}\in \J(\hat{E},\hat{\pi},j)$, so that, by the same arguments of \cite[Lemma 2.4]{seideles}, all the moduli spaces we encounter satisfy the necessary regularity conditions.

Consider the moduli space of closed $(\hat{J},j)$-holomorphic sections
\begin{align}\label{modulisection}
\mathcal{M}_{\hat{J}}= \left \{ u\colon\C\PP^1  \rightarrow \hat{E}, \ \hat{\pi} \circ u=id_{\C\PP^1}, \ \hat{J} \circ D u= D u \circ j  \right  \}.
\end{align}

The moduli space has a non-empty boundary, but as we explain below this does not cause compactness issues, as the only sections reaching the boundary must be trivial.

\begin{lemma}\label{lemmaconstant}
	The space $\mathcal{M}_{\hat{J}}$ is not empty. Moreover, there is a compact subset $K \subset \hat{E} \setminus \partial^h \hat{E}$ such that for all $u \in \mathcal{M}_{\hat{J}}$, either $\im(u)\subset K$, or $u$ is a trivial (constant) section.
\end{lemma}
\proof Let $q \in V$ a point in a neighbourhood $\partial^h\hat{E}\subset V$ of the horizontal boundary as in $\eqref{nhooditem}$ above. Via the trivialisation of this neighbourhood, one obtains a trivial section $s: \C\PP^1 \rightarrow \hat{E}$ with $s(z)=q$ for all $z \in \C\PP^1$, which is a regular $(\hat{J}, j)$-holomorphic section, so $\mathcal{M}_{\hat{J}}$ is not empty. The rest of the proof follows from a maximum principle as in \cite[Lemma 2.2]{seideles}.
\endproof

We can adapt the argument of \cite[Lemma 2.3]{seideles} to the case of closed curves to show that for our choice of almost complex structure $\hat{J}$, the moduli space $\mathcal{M}_{\hat{J}}$ is a compact smooth manifold with boundary. The only issue that could possibly occur is a loss of compactness for the component containing sections outside the compact part $K$, which, by Lemma \ref{lemmaconstant}, can only be trivial sections. These elements have bounded energy, as they are all in the same homology class. By the Gromov compactness theorem, it follows that the only non-compact phenomenon that can occur in this case is sphere bubbling. The next lemma shows how to discard bubbles.

\begin{lemma}\label{bubblemma}
	Let $u_{\infty}$ be the limit of a (sub)sequence of pseudoholomorphic sections $(u_n)_{n \in \N}$ of the Lefschetz fibration $\hat{\pi}$. A component of $u_{\infty}$ is either an element in the class $\lbrack u_i \rbrack$ (for $i \in \N$) or is contained in a single fibre. In the latter case, the component is a bubble.
\end{lemma}
\proof 
Let $v_1, v_2, \dots v_k$ be the components of $u_{\infty}$. The limiting curve $u_{\infty}$ is assumed to be $(\hat{J},j)$-holomorphic and nonconstant, so it has to have degree one, as $
\sum_{j=1}^k \lbrack \pi \circ v_i \rbrack = \lbrack \pi \circ u_{\infty}\rbrack = \lbrack \C\PP^1 \rbrack.
$ It follows that the degrees of its components sum up to one. All degrees are non-negative, so there is only one component with degree one. If in addition there was a bubble, it would be represented in a degree zero component and therefore would have to be entirely contained in a fibre (note that by positivity of intersections, the bubble cannot intersect other fibres).
\qed

Since the fibres are exact, there can be no bubbling of the type of Lemma \ref{bubblemma}, so the moduli space $\M_{\hat{J}}$ is compact.

\begin{lemma}\label{section}
	Through each point of the smooth fibre $M$ there is at least one holomorphic section $s \in \mathcal{M}_{\hat{J}}$.
\end{lemma}
\proof 
As in the proof of Lemma \ref{lemmaconstant}, we consider a neighbourhood of the horizontal boundary $V\supset\partial^h\hat{E}$, $q\in V$ such that $\hat{\pi}(q)=:z_{gen}\in \CPone \setminus \Crit v(\hat{\pi})$ and the trivial section through $q$, $s\colon \CPone\to \hat{E}$.
Consider

\begin{align*}
\mathcal{M}(\hat{J}, q):= \{ u\in  \mathcal{M}_{\hat{J}}, \  q\in \Im(u)  \} \subset \mathcal{M}_{\hat{J}}.
\end{align*}
It is a smooth compact manifold (by the same arguments as for $\M_{\hat{J}}$).
Moreover, by Lemma \ref{lemmaconstant}, the only element in $\mathcal{M}(\hat{J},q)$ is the trivial $\hat{J}$-holomorphic section $s\colon \CPone\to \hat{E}$ through $q$.

Let $p\in \hat{\pi}^{-1}(z_{gen})$ be any other point in the fibre of $q$, and consider a path $\alpha : \lbrack 0,1 \rbrack \rightarrow M$ with $\alpha(0)=p$, $\alpha(1)=q$.
For every point $\alpha(t)$, $t\in [0,1]$, define $\mathcal{M}(\hat{J},\alpha(t), [s]):=\{  u \in \mathcal{M}_{\hat{J}}, \ \alpha(t)\in \Im(u),  \text{ and } [u]=[s] \}$. Clearly, $\mathcal{M}(\hat{J},\alpha(1), [s])=\mathcal{M}(\hat{J},q)$.

Consider
\begin{align}\label{cobordism}
\mathcal{M}_{cob}:= \bigcup_{t \in \lbrack 0, 1 \rbrack}\mathcal{M}(\hat{J},\alpha(t), [s])  \subset \mathcal{M}_{\hat{J}}.
\end{align}

The boundary components of \eqref{cobordism} are given by $\partial \mathcal{M}_{cob}=\mathcal{M}(\hat{J},p, [s])\sqcup \mathcal{M}(\hat{J},q)$.
We want to show that the space $\mathcal{M}_{cob}$ is compact, so that it defines a one-dimensional cobordism between $\mathcal{M}(\hat{J},p, [s])$ and $\mathcal{M}(\hat{J},q)$.
As before, since $\hat{J} \in \J(\hat{E}, \hat{\pi}, j)$ is chosen to be generic, $\mathcal{M}_{cob}$ is a smooth manifold.

To show that $\mathcal{M}_{cob}$ is compact, the same strategy applies as in the case of $\mathcal{M}_{\hat{J}}$. In particular, consider a sequence $t_i \subset [0,1]$ and for each $i$ a section $u_i \in \M(\hat{J}, \alpha(t_i), [s])\subset \mathcal{M}_{cob}$. 

All the sections of the sequence we are considering belong to the same homology class, by definition. In particular they have the same area, so that Gromov's theorem applies.
Consequently, as $t_i$ tends to a limit value $t_{\infty}$, the sequence $u_i$ converges to a stable map $u_{\infty}$. As before (in the proof of Lemma \ref{bubblemma}), if sphere bubbling occured, bubbles would have to be ``vertical'' (meaning entirely contained in the fibres), which is impossible by the exactness of the fibres.

It follows that $\mathcal{M}_{cob}$ is compact and hence $\mathcal{M}(\hat{J},p, [s])$ and $\mathcal{M}(\hat{J},q)$ are indeed cobordant. 
Since the signed count of the boundary components of a $1$-dimensional compact manifold is zero, it follows that the $0$-dimensional components of the two spaces have the same cardinality. 
In particular, for any $p\in \hat{\pi}^{-1}(z_{gen})$, $\mathcal{M}(\hat{J},p, [s])$ is not empty which means there is at least one element in $\mathcal{M}_{\hat{J}}$ that passes through $p$.

\qed

\begin{cor}\label{evaluationsurj}
	The map induced by the evaluation map \begin{equation}\label{ev1}
	\begin{aligned}
	ev\colon \mathcal{M}_{\hat{J}} \times \CPone & \lra \hat{E} \\
	(u,z)& \longmapsto ev_z(u)=u(z). \end{aligned}
	\end{equation}
	is surjective.

\end{cor}
\proof By Lemma \ref{section}, the image of the map \eqref{ev1} is dense, since each point on a smooth fibre has a preimage. As $\mathcal{M}_{\hat{J}}$ is compact and the mapping is continuous, the result extends to all points of $\hat{E}$ and hence \eqref{ev1} is surjective. \qed

\proof[Proof of Theorem \ref{bgz}]

Assume by contradiction that the product $\phi=\tau_{L_{j_1}} \dots \tau_{L_{j_k}}$ is isotopic to the identity, and build the fibration $\hat{\pi}\colon \hat{ E}\lra \CPone$ and the moduli space $\M_{\hat{J}}$ as above.
By Corollary \ref{evaluationsurj}, the evaluation map $ev\colon \M_{\hat{J}} \times \CPone \lra \hat{E}$ is surjective. Consider the commuting diagram \begin{align*}
\xymatrix{
	\M_{\hat{J}}\times \CPone \ar[r]^{  \  \ \  \ ev} \ar[rd]^{pr_2} &
	\hat{E} \ar[d]^{\hat{\pi}} &\\
	&\CPone}
\end{align*}where $pr_2\colon \M_{\hat{J}}\times \CPone \rightarrow \CPone$ is the projection to the second factor. 

Let $x \in \Crit(\hat{\pi}) \subset \hat{E}$ be any point in the critical set. By the surjectivity of $ev$, there is a pair $(u,w) \in \M_{\hat{J}}\times \CPone$ such that $u(w)=x$, so that $w \in \CPone$ is the critical value associated to $x$. From the diagram we obtain \begin{align*}
D_{(u,w)}(pr_2)=D_{(u,w)}(\hat{\pi} \circ ev) \\
D_{(u,w)}(pr_2)=D_{x} \hat{\pi}  D_{(u,w)}(ev).
\end{align*}
As $x$ is a critical point, $D_x \hat{\pi} =0$, which forces $D_{(u,w)}(pr_2)$ to be the zero map. But this is in contradiction with $D_{(u,w)}(pr_2)$ being surjective. \qed

\begin{cor}
	There is no exact Lefschetz fibration with global monodromy symplectically isotopic to the identity, except for the trivial fibration.
\end{cor}
\endproof

\subsection{Relative version}\label{relativegeo}

Let \begin{align}\label{specialcaseplumbing}
M:=T^*S^m\#_{pt} T^*S^m \#_{pt} T^*S^m \cdots \#_{pt} T^*S^m
\end{align}
be a ``multi-plumbing'' of $m$ spheres (an iterated construction of transverse plumbing of spheres, see Section \ref{plumbing} for the definition of plumbing). By Theorem \ref{bgz}, we know that no product $\phi \in \Symp_{ct}(M)$ of Dehn twist along the core spheres can be compactly supported symplectically isotopic to the identity. 
However, the theorem, a priori, doesn't prevent such a product to act trivially on some Lagrangians submanifolds of $M$. Is it possible to tell whether there are Lagrangians that detect the non-triviality of $\phi$?
Let $T$ be a cotangent fibre of the $j$-th $T^*S^n$-summand, for $j\in \{ 1, \dots , m \}$. The theorem we prove in this section shows that any product of positive Dehn twists along Lagrangian cores and involving the $j$-th sphere does not preserve $T$ up to compactly supported symplectic isotopy.

\thmrelativeversion*

We prove the statement of Theorem \ref{relativeversion1} in the equivalent version where $(M,\omega=d\lambda_M)$ is a Liouville domain and $T \subset M$ is a Lagrangian disc preserved by the Liouville flow near the boundary $\partial M$ (so that $\partial T \subset \partial M$). This is only a choice so that the Lefschetz fibrations involved have exact compact fibres.

As in the statement, write $\phi=\prod_{i=1}^{k} \tau_{L_{j_i}}$, $j_i \in \{ 1, \dots , m \}$. By assumption, there is at least one index $\ell \in \{1, \dots , k\}$ such that $j_{\ell}=j$.
Assuming $\phi(T)\simeq T$ via a compactly supported isotopy, we arrive at the contradictory statement $j \notin \{ i_1, \dots , i_k \}$. 

From the data $(M, (L_{j_1}, \dots, L_{j_{\ell}}, \dots  ,L_{j_k}))$, build an exact Lefschetz fibration $\pi'\colon (E', \Omega_{E'}, \lambda_{E'})\to (\C,\lambda_{\C})$ with smooth fibre the Liouville domain $(M,d\lambda)$, base point $z_*\in \R$, $z_*\gg 0$ such that the $k$ critical values $\Crit v(\pi')=\{ w_{j_1}, \dots, w_{j_{\ell}} , \dots, w_{j_k} \}$ are ordered vertically on the imaginary line, $\Crit v(\pi)\subset i \R$ with a basis of vanishing paths $(\gamma_{j_1}, \dots \gamma_{j_k})$ (\cite[16e]{seidelbook}). 

Let $(\Delta_{\gamma_{j_1}}, \dots , \Delta_{j_{k}})$ be the corresponding basis of Lefschetz thimbles and for every $i=1, \dots , k$, $V_{j_{i}}:=\pi^{-1}(z_*)\cap \Delta_{\gamma_{j_{i}}}$ the associated vanishing cycles, which, under the identification $\pi^{-1}(z_*)=M$ correspond to $L_{j_i}$. Let $\sigma \colon S^1 \to \C$ be a loop encircling all critical values. 

Build a new exact fibration $\pi \colon (E, \Omega_E,\lambda_E)\to (\C,\lambda_{\C})$ associated to the data $(M, (V_{j_1}, \dots ,V_{j_{\ell}}, \dots , V_{j_k}, V_{j_{\ell}}))$, with base point $z_*\in \C$, an extra critical value $w_{j+1}\in \Crit v (\pi) \subset i \R$ and an extra vanishing path $\gamma_{j_{k+1}}$ such that $\Im(\gamma_{j_{k+1}})\cap \Im(\sigma)=\emptyset $
(all the other choices are the same as for $\pi'$).


\begin{figure}[h]
	\centering
\includegraphics[width=13cm]{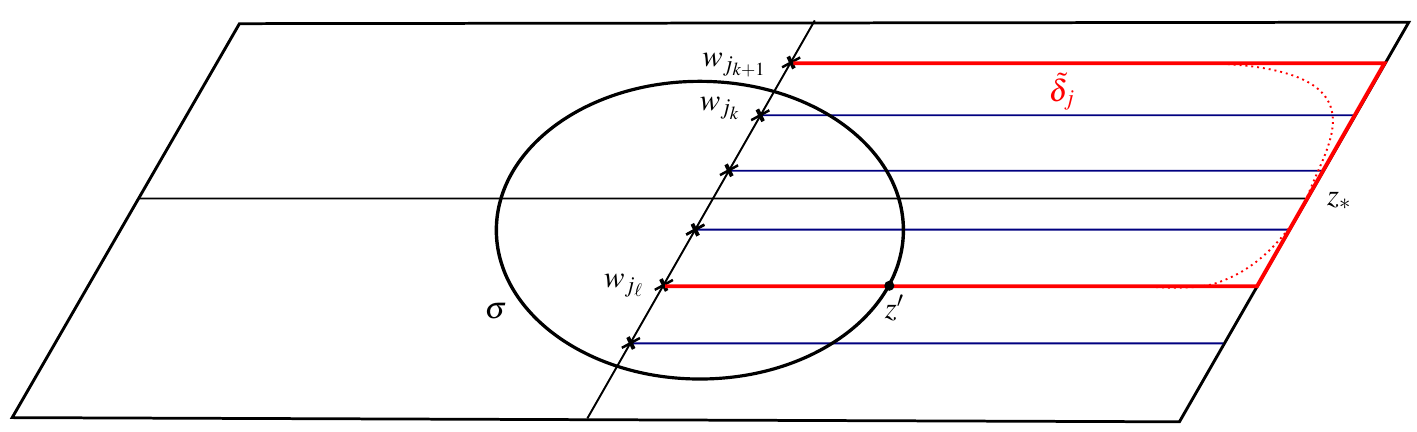}
	\caption{The new fibration $\pi$ has an extra critical value $w_{j_{k+1}}$ and a matching sphere that fibres over the smoothing $\delta_j$ of the red arc $\tilde{\delta}_j$.}\label{newfibration}
\end{figure}

Compared to $\pi'$, there are now two critical points $w_{j_{\ell}}, w_{j_{k+1}}$ associated to the same vanishing cycle $V_{j_{\ell}}$. Therefore, there is a matching path $\delta_j\colon [0,1]\to \C$ with $\delta_j(0)=w_{j_{\ell}}, \ \delta(\frac{1}{2})=z_*, \ \delta_j(1)=w_{j_{k+1}}$, whose parallel transport is a Lagrangian matching sphere $S_j \cong S^{n+1} \subset E$ (Section \ref{dtandlf}, \cite[(16g)]{seidelbook}) fibred by Lagrangians isomorphic to $L_j$ (see Figure \ref{newfibration}). Let $z' \in \Im(\delta_j)\cap \Im(\sigma)$, and via parallel transport identify $T\subset \pi^{-1}(z_*)$ with a copy of the Lagrangian in $\pi^{-1}(z')$.

By construction, the monodromy around $\sigma$ is given by the product $\phi$. By assumption there is an isotopy $\phi(T)\simeq T$, so parallel transport of $T$ along $\sigma$ yields a well-defined Lagrangian $P_{\sigma} \subset E$. For $z\in \Im(\sigma)$, let $T_z\subset \pi^{-1}(z)$ be the exact fibres of $P_{\sigma}$. Then $\Omega_E|_{P_{\sigma}}=df_{\sigma}+\pi^*(\kappa_{\sigma})$ for a function $f_{\sigma}\in C^{\infty}(P_{\sigma}, \R)$ such that for every $z\in \Im(\sigma)$, $f_{\sigma}|_{\pi^{-1}(z)}$ makes $T_z$ exact and $\kappa_{\sigma} \in \Omega^1(\Im(\sigma))$ (\cite[Lemma 1.3]{seideles}).

\begin{lemma}\label{lemmanontrivial}
	The Lagrangian $P_{\sigma}$ defines a nontrivial class in $H_{n+1}(E, \partial E;\Z)$.
\end{lemma}
\proof The matching sphere $S_j$ and the disc $P_{\sigma}$ are properly embedded Lagrangian submanifolds meeting transversely at the point $y \in L_j$ lying over the intersection between $\sigma$ and the matching path associated to $S_j$. Their homological intersection, which is the image of a non-degenerate pairing $$H_{n+1}(E;\Z)\times H_{n+1}(E, \partial E;\Z) \rightarrow \Z$$ is one, so in particular $P_{\sigma}$ represents a non trivial homology class in $H_{n+1}(E, \partial E;\Z)$. 
\endproof

\subsubsection{Proof of Theorem \ref{relativeversion1}} 
Let $D \subseteq \C$ be the disc bounded by the loop $\sigma$ in the base of $\pi$.
The idea for the proof of Theorem \ref{relativeversion1} is based on a section-count which follows the same principles of Section \ref{alternativeproof}. In this context however, we consider pseudoholomorphic sections definining boundary conditions for $E|_{D}$ on $P_{\sigma}$.

Let $\J(\pi, E, j_{\C})$ be the set of almost complex structures compatible with $\pi$ (see Definition \ref{compatiblej}), where $j_{\C}$ is the standard complex structure on $\C$. For a generic element $J \in \J(\pi,E, j_{\C})$, \begin{align}\label{modulij}
\M(J, P_{\sigma}):=\{ u\colon (D, \partial D)\lra (E, P_{\sigma}), \ \pi\circ u=id_D, \  J\circ Du=Du \circ j_{\C}|_D \}
\end{align} be the moduli space of pseudoholomorphic sections with boundary condition on $P_{\sigma}$.

The Lagrangian $P_{\sigma}$ is fibred by copies of the exact Lagrangian $T\subset M$, and therefore $P_{\sigma} \cap \partial^hE\neq \emptyset$, it is not disjoint from the horizontal boundary. As a result, the moduli space $\M(J, P_{\sigma})$ is not compact, but fortunately its non-compact ends are very well-behaved.

Below, we show that for a generic almost complex structure in $\J(\pi,E, j_{\C})$, the ``non-compact'' elements (those sections reaching the horizontal boundary) of the moduli space \eqref{modulij} are regular. We do this by showing that such sections must be trivial---and the trivial section can be made regular, as the almost complex structure is product-like near $\partial^hE$.
For all the other holomorphic sections, which are entirely contained in the compact region, the same regularity arguments as in \cite{seideles} apply.

\begin{lemma}\label{regularity1}
	There is $\hat{J} \in \J(\pi,E, j_{\C})$ with the following property. 
	If $v\colon (D, \partial D) \to (E, P_{\sigma})$ is a $\hat{J}$-holomorphic section with boundary condition on $P_{\sigma}$, such that $u(\sigma)\cap \del^hE\neq \emptyset$ but not fully contained in $P_{\sigma} \cap \partial ^hE$, then $v$ is a constant section.
\end{lemma}

\proof
We show that any generic element $J\in \J(\pi,E, j_{\C})$ can be deformed to an almost complex structure $\hat{J} \in \J(\pi,E, j_{\C})$ as in the statement. To do that we use a \emph{reverse isoperimetric inequality} from \cite{ysreverse} that applies to the Liouville completion of $E$.

Identify a collar neighbourhood of $\partial^hE$ with $C(\partial^hE):=\C \times ((-\eps, 0] \times \partial M)$), and consider the Liouville completion of $E$, $(\overline{E}, \omega_{\overline{E}})$, obtained by gluing a cylindrical end $U^h:=\C \times ([0, \infty) \times \partial M )$ along a collar neighbourhood of the horizontal boundary $\partial^hE$, such that $\omega_{\overline{E}}|_{U^h}=d(\lambda_{\C}+e^t\lambda_M|_{\partial M})$, for the coordinate $t$ on $[0, \infty)$.

Let $(\overline{M}, \overline{\omega})$ be the generic smooth fibre of $\overline{E}$, and $\overline{T} \subset \overline{M}$ the Lagrangian obtained from $T$ by gluing a conical end at the boundary. Accordingly, let $\overline{P}_{\sigma}\subset \overline{E}$ be the ``completion'' of $P_{\sigma}\subset E$ in $\overline{E}$. This Lagrangian can be trivialised outside of a compact set as $\partial D \times U^{\infty} \subset \C \times U^{\infty} \subset \overline{E}$, where $U^{\infty}\subset \overline{M}$ is a neighbourhood of the cylindrical end of $\overline{T}$. Extend $J$ to a cylindrical almost complex structure $\overline{J}$ on $\overline{E}$ (see Definition \ref{defacs}).

By \cite[Lemma 2.43]{gps}, $(\overline{E}, \omega_{\overline{E}})$ has \emph{bounded geometry} in the sense of \cite[Definition 2.42]{gps}, which is equivalent to the notion of bounded geometry of \cite[1.4]{ysreverse}, see \cite[p.104]{gps}. The same holds for the Lagrangian $\overline{P}_{\sigma} \subset \overline{E}$, as it is compact in the base direction, and conical in the fibre direction (see the proof of \cite[Lemma 2.43]{gps}). Bounded geometry implies that for any $\overline{J}$--holomorphic section $u\colon (D, \partial D)\to (\overline{E}, \overline{P}_{\sigma})$ there is a reverse isoperimetric inequality (\cite[Theorem 1.4]{ysreverse}) \begin{align}\label{reverseisop}
\ell(u|_{\partial D}) \leq a(u)\cdot C,
\end{align} where $\ell$ is the length function associated to a $\overline{J}$-compatible metric $g_{\overline{J}}$, $C>0$ is a constant depending on $\overline{E}$, and $a(u)$ is the area of the curve. 

Let $A:=\int_{\partial D}\kappa_{\sigma}$, and set $R:=A \cdot C$. 
For $R>0$, consider a piece of symplectisation $(E_{R+1}:=E\cup \C \times ([0, R+1] \times \partial M), \omega_{R+1})$ with $\omega_{R+1}|_E=\omega_E$ and $\omega_{R+1}|_{\C \times ([0, R+1] \times \partial M )}=d(\lambda_{\C}+e^t\lambda_M)$, and a compatible almost complex structure $J_{R+1}=\overline{J}|_{E_{R+1}}$ of contact type. Clearly $E\subset E_{R+1}\subset \overline{E}$, and there is a diffeomorphism $\psi \colon E_{R+1} \to E$, that is the identity on $E\setminus C(\partial^hE)$ and compresses $\C \times ((-\eps, R+1] \times \partial M)$ to $\C \times ((-\eps, 0] \times \partial M)$ via the negative Liouville flow.

Every $\overline{J}$-holomorphic curve $u\colon (D,\partial D)\to (\overline{E}, \overline{P}_{\sigma})$ such that there are $z_1, z_2 \in \partial D$ with $u(z_1)\in Int(E)$ and $u(z_2)\in \overline{E}\setminus E_{R+1}$ satisfies $d(u(z_1), u(z_2))>A\cdot C$ and the inequality \eqref{reverseisop}

Now set $\hat{J}:=\psi_*(J_{R+1})$. This satisfies the requirements of the Lemma. Namely, let $v\colon (D, \partial D)\to (E,P_{\sigma})$ be a $\hat{J}$-holomorphic section as in the statement, i.e such that there are $z_1, z_2 \in \partial D$ with $v(z_1) \in P_{\sigma} \setminus (P_{\sigma} \cap \partial ^hE)$ and $v(z_2) \in P_{\sigma} \cap \partial ^hE$. 

Then we certainly have $d(v(z_1), v(z_2))< \ell(v|_{\partial D})$ for the distance function $d$ and the length $\ell$ associated to a compatible metric $g_{\hat{J}}$. On the other hand, the area of $v$ is bounded by a fixed upper bound since $a(v)=\int_Dv^*\Omega_E=\int_D d(v^*\lambda_E)=\int_{\partial D}v^*(\lambda_E)=\int_{\partial D}\kappa_{\sigma}=A$ by exactness of $\Omega_E$ and fibrewise exactness of $P_{\sigma}$. 

By stretching the neck in a neighbourhood of the boundary of $E$ to $E_{R+1}$, the pullback $\psi^*(v)$ produces a contradiction, since $d(\psi^*(v(z_1)), \psi^*(v(z_2))<\ell(\psi^*(v|_{\partial D}))< a(\psi^*(v))\cdot C=A\cdot C$, but also $d(\psi^*(v(z_1)), \psi^*(v(z_2))>A\cdot C$ by construction of $E_{R+1}$. Therefore, $v$ has to be a constant section.

\qed

From now onwards, fix an almost complex structure $\hat{J}\in \J(\pi,E, j_{\C})$ as in Lemma \ref{regularity1}. The above results imply that the only possible scenario left to consider in the case of a non-constant section with boundary condition on $P_{\sigma}$ intersecting $\partial^hE$, is to be entirely contained in the horizontal boundary of the fibration. 
\begin{lemma}\label{regularity2}
	Let $u\colon D \to E$ be a $\hat{J}$-holomorphic section such that $\im(u) \subset \partial^h E$. Then $u$ is a constant section.
\end{lemma}
\proof Assume there is a non-constant section $u\colon D \to E$ such that $\im(u)\subset \partial^h E$. Identify (via a trivialisation as in \eqref{productlike}) a neighbourhood of $\partial^hE$ as $U^{\partial}\cong \C \times M^{out}\subset \C \times M$ for an open neighbourhood $M^{out}\subset M$ of $\partial M$. Then the projection of $\Im(u)$ to $M$ defines a non--constant $\hat{J}|_{M}$--holomorphic disc $u\colon (D, \partial D) \to (M,T)$, which, by the exactness assumptions on $M$, cannot exist. Therefore, $u$ must be a constant section.
\qed

We now prove that there are no compactness issues. The moduli space $\M(\hat{J}, P_{\sigma})$ has non-compact end, but by the regularity discussion above, the only sections reaching it are the constant ones, and all elements of $\M(\hat{J}, P_{\sigma})$ have bounded energy so that the Gromov compactness theorem applies. The bubbles components in the Gromov limit of a sequence of $(\hat{J},j_{\C})$-holomorphic sections in $\M(\hat{J}, P_{\sigma})$ are either spheres in the fibres over $D$, or discs in the fibres $\pi^{-1}(z)$, for $z\in \Im(\sigma)$, with boundary condition on $T_z$.
Both options can be discarded by exactness of $E$ and fibrewise exactness of $P_{\sigma}$.

\begin{lemma}\label{lemmasurj}
	The evaluation map \begin{align}\label{evaluation}
	ev: \M(\hat{J}, P_{\sigma}) \times D& \lra E \\
	(u,z) &\longmapsto u(z) \nonumber
	\end{align}
	\begin{enumerate}[label=(\roman*)]
		\item is proper;
		\item restricts to a surjective map $\M(\hat{J},P_{\sigma}) \times \partial D \lra P_{\sigma}$ of degree one.
	\end{enumerate}
\end{lemma} 
\proof 
(i) To prove this property is enough to show that every sequence of sections $\{ u_k \}_{k \in \N}$ in $\M(\hat{J},P_{\sigma})$ whose image under $ev$ lies in a relatively compact set of $E$ has a convergent subsequence. Consider such a sequence. If its image under \eqref{evaluation} lie in a compact set then by exactness there is an upper bound to the energy of all elements in the sequence (which is bounded by a finite value determined by the maximum among all areas of the curves). Then, by the Gromov compactness theorem, $\{ u_k \}_{k}$ admits a subsequence converging to a stable map, which, in the absence of bubbles, can only be another section. 

(ii) To prove the second point, we show that the algebraic count of sections through every point of $P_{\sigma}$ is one. Let $U^{\partial}\supset \partial^hE$ be a neighbourhood of the horizontal boundary as in the proof of the previous lemma and $q \in U^{\partial}\cap\pi^{-1}(\sigma)$. Since $\phi$ is compactly supported in a neighbourhood of the vanishing cycles, the monodromy around $\sigma$ preserves $q$. 
By lemmata \ref{regularity1}, \ref{regularity2}, the moduli space $\M(\hat{J}, q):=\{ u\in \M(\hat{J},P_{\sigma}), q\in \Im(u) \} \subset \M(\hat{J}, P_{\sigma})$ is compact and only contains the constant section $s\colon D \rightarrow E$ through $q$.

Given another point $p \in P_{\sigma}$, consider the path $\alpha \colon [0,1] \to M$ with $\alpha(0)=p, \alpha(1)=q$ and define
\begin{align*}
\M(\hat{J},P_{\sigma},\alpha(t), [s]):= \{ u\in \mathcal{M}(\hat{J},P_{\sigma}),  \ \alpha(t) \in \Im(u) \text{ and } [u]=[s] \}.
\end{align*}
Clearly, $\M(\hat{J},P_{\sigma}, \alpha(1), [s])=\M(\hat{J}, q)$.

Consider
\begin{align}\label{cobordism2}
\mathcal{M}_{cob}:= \bigcup_{t \in \lbrack 0, 1 \rbrack}\mathcal{M}(\hat{J},P_{\sigma}, \alpha(t), [s]) \subset \M(\hat{J}, P_{\sigma}).
\end{align}

All elements in $\M_{cob}$ are in the same homology class so that the same compactness arguments apply as above. Compactness implies that for every $p \in P_{\sigma}$, the moduli space $\M(\hat{J}, P_{\sigma}, p, [s])$ is cobordant to the moduli space $\M(\hat{J},q)$. Therefore, by the same reasoning as in the proof of Lemma \ref{section}, through each point of $P_{\sigma}$ there is algebraically a unique section in $\M(\hat{J},P_{\sigma})$, so that the restriction $\M(\hat{J},P_{\sigma}) \times \partial D \lra P_{\sigma}$ is surjective and of degree one.
\endproof 

\proof[Proof of Theorem \ref{relativeversion1}]
Under the assumption that $\phi(T)\simeq T$, we have proved that $P_{\sigma}$ represents a non-trivial class in $H_{n+1}(E, \partial E)$ (Lemma \ref{lemmanontrivial}). The same assumption however also yields Lemma \ref{lemmasurj}, which in particular implies that $ev_*(\mathcal{M}_{\hat{J}}(E,P_{\sigma})\times \partial D)=\lbrack P_{\sigma} \rbrack \in H_{n+1}(E, \partial E)$ is realised as the boundary of the chain $ev_*(\mathcal{M}_{\hat{J}}(E,P_{\sigma}) \times D) \in C_{n+2}(E, \partial E)$. This is a contradiction, which concludes the proof of the theorem.\endproof

\subsection{Product of projective twists}\label{generalisationprojective}\label{projectiveproduct}

We continue the investigation on positive products of twists in Liouville manifolds, this time focussing on projective twists. 
Ideally, one would try to generalise as many results from the previous sections to this situation. 

The previous section heavily relied on the link between Dehn twists and Lefschetz fibrations, and many constructions we used depended on section-count invariants of Lefschetz fibrations.

Perutz showed in \cite{perutzmbl} that any \emph{fibred twist} admits a representation as the local monodromy of a Morse--Bott--Lefschetz (MBL) fibration. Projective twists can be thought of as an example of $S^1$-fibred twists, so we could envisage extending the mechanisms behind the proof for the spherical case to the setting of MBL fibrations (following \cite{perutzmbl} and \cite{wwfibred}) to show the analogous statement for projective twists. 

\begin{quest}\label{conjecture}
	Let $\varphi \in \Symp_{ct}(W)$ be a non-empty composition of positive powers of projective twists on a Liouville manifold $(W, \omega)$ of dimension at least four. Can $\varphi$ be isotopic to the identity in $\Symp_{ct}(W)$?
\end{quest}

Unfortunately, the section-count strategy presents a route filled with obstacles; the central problem being the lack of compactness of moduli spaces of sections of MBL fibrations. The critical locus $\Crit(\pi)$ of such a fibration is a compact symplectic submanifold of the total space, and in general contains rational curves. The total space of a MBL fibration $\pi\colon E\lra \C$ associated to a projective twist cannot be made into an exact symplectic manifold, so bubbling phenomena can become an issue when considering moduli spaces of pseudoholomorphic sections. 

Instead, the idea remains, as in Section \ref{freegenplumb}, to use the Hopf correspondence to translate a situation involving projective twists into one involving Dehn twists.

\thmrealgeneral*

\proof As in Section \ref{hopfrpn}, let $q\colon (\tilde{W}, \tilde{\omega}) \rightarrow (W, \omega)$ be the symplectic double cover given by the class $\alpha$ and $L_1, \dots , L_m \subset \tilde{W}$ Lagrangian spheres obtained as double cover of $K_1, \dots ,K_m \subset W$. The composition of projective twists $\varphi \in \Symp_{ct}(W)$ lifts to a composition of spherical Dehn twists $\phi \in \Symp_{ct}(\tilde{W})$.
Assume there is an isotopy $(\varphi_t)_{0 \leq t \leq 1}$ connecting the composition of projective twists $\varphi_0=\varphi$ to the identity $\varphi_1=Id$. The isotopy lifts to a family of compactly supported maps $(\phi_t)_{0 \leq t \leq 1}$ in the double cover $\tilde{W}$, where $\phi_0=\phi$ is the lift of $\varphi$. Then, $\phi_1$ covers the identity and can therefore only be either the identity or a deck transformation. The latter type would define a non-compactly supported symplectomorphism, hence $\phi_1$ must be the identity. It follows that $\phi \in \Symp_{ct}(\tilde{W})$ is a composition of Dehn twists in a Liouville domain which is isotopic to the identity, contradicting Theorem \ref{bgz}.
\endproof

\begin{rmk}A similar argument fails when applied to complex projective twists.
	Let $(W^{4n},\omega)$ be a symplectic manifold with complex projective Lagrangians $K_1, \dots, K_m$ satisfying Assumption \eqref{assptcx}. The fibration $(Y,\Omega)\to (W, \omega)$ constructed from the cohomological condition is not proper, so an isotopy in $\Symp_{ct}(W)$ cannot be lifted to an isotopy in $\Symp_{ct}(Y)$.
\end{rmk}

\section{Epilogue: framings of projectie twists, homotopy projective Lagrangians}\label{epilogue}

As a last application of the Hopf correspondence, we examine homotopy projective Lagrangians. We prove two non-embedding results for Lagrangian projective spaces in non-standard homeomorphism/diffeomorphism classes (Theorems \ref{cor1} and \ref{cor2}), and for $n\geq 19$, the existence of projective twists obtained from a non-standard choice of framing, that are not Hamiltonian isotopic to the standard $\tau_{\CP} \in \Symp_{ct}(T^*\CP)$:

\thmframing*
Embedding theorems are obtained in Section \ref{nonemb} using homotopy theory results combined with the existing state of the art of the nearby Lagrangian conjecture, and the use of the Hopf correspondence.

We subsequently investigate the question of framings for projective twists in Section \ref{aptwist}. For that purpose, we utilise the current literature on framing of Dehn twists, a pairing constructed by Bredon, and the Hopf correspondence.
This enables us to obtain instances in which the (Hamiltonian isotopy class of the) local projective twist does depend on a choice of framing of the associated Lagrangian projective space. With the additional use of topological modular forms, we explain why there should be infinitely many such examples.

\subsection{Lagrangian non-embeddings of projective spaces}\label{nonemb}\label{stateart}

The \emph{nearby Lagrangian conjecture} states that given a closed smooth manifold $Q$, any closed exact Lagrangian submanifold of $(T^*Q,d\lambda_{Q})$ is Hamiltonian isotopic to the zero section. If this conjecture was true, the existence of another closed exact Lagrangian embedding $L \hookrightarrow T^*Q$ would yield a diffeomorphism, $L \cong Q$. By Weinstein's neighbourhood theorem, the latter version of the statement can also be read as: if $(T^*L,d\lambda_{T^*L})$ is symplectomorphic to $(T^*Q,d\lambda_{T^*Q})$, then $L$ is diffeomorphic to $Q$.

The conjecture has been verified for some specific examples ($T^*S^2, T^*\R\PP^2$ by Hind \cite{hindknot} and Li--Wu \cite{liwu}, $T^*T^2$ by Dimitroglou Rizell--Goodman--Ivrii \cite{rgitori}), and weaker versions of it have been proved. Currently, the most general feature one can deduce from an exact Lagrangian embedding in $(T^*Q, d\lambda_{T^*Q})$ is (simple) homotopy equivalence.

\begin{thm}[\cite{abouzaidnearby}, \cite{kraghnearby}, \cite{aknearby}]\label{heq}
	If $L \subset T^*Q$ is a closed exact Lagrangian embedding, then the projection $L \subset T^*Q \xrightarrow{p} Q$ is a (simple) homotopy equivalence. 
\end{thm}

\begin{rmk}
	Note that if $L \subset T^*Q \xrightarrow{p} Q$ is a Lagrangian as in the above statement, then $TL \otimes \C \cong p^*(TQ \otimes \C)$. The Pontryagin classes $p_i\in H^{4k}(\cdot)$ satisfy $2p_i(L)=2p_i(Q)$. Moreover, the (rational) Pontryagin classes $p_i$ are homeomorphism invariants (\cite{novikov}).
\end{rmk}

Equipped with the connected sum operation, the set of h-cobordism classes of homotopy $m$-spheres $\Theta_m$ has an abelian group structure (where the standard sphere plays the role of neutral element). We will always assume $m>5$, in which case the elements of $\Theta_m$ correspond to diffeomorphism classes of $m$-spheres.

The group $\Theta_m$ fits in an exact sequence (\cite{kmspheres})
\begin{align}\label{kmexact}
0\longrightarrow	bP_{m+1}	\xlongrightarrow{ } \Theta_m \xlongrightarrow{\psi} \coker(J_m)\longrightarrow bP_{m}.
\end{align}
Here $bP_{m+1}=ker(\psi) \subset \Theta_m$ denotes the subgroup of homotopy $m$-spheres bounding an $(m+1)$-dimensional parallelisable manifold, and $J_m \colon \pi_m(O)\rightarrow \pi_m(S)$ is a map from the $m$-th stable homotopy group $\pi_m(O)=\lim\limits_{\ell \to \infty}\pi_m(SO(\ell))$ to the $m$-th stable homotopy group of spheres $\pi_m(S):=\lim\limits_{\ell \to \infty}\pi_{m+\ell}(S^{\ell})$ (see e.g \cite[Section 3]{levine}). This group is also called the \emph{$m$-th stable stem}.

Throughout the section, we will repeatedly use the following fact about the sequence \eqref{kmexact}.

\begin{thm}[{{\cite[Theorem 5.1]{kmspheres}}}]
	If $m$ is an odd integer, $bP_m=0$. Consequently, for any odd $m$, $\psi\colon \Theta_m \to \coker(J_m)$ is surjective.
\end{thm}

In the symplectic setting, homotopy spheres are good candidates to test the nearby Lagrangian conjecture.

\begin{thm}[\cite{abouzaidframed}, extended by \cite{ekspheres}]\label{nearby}
	Let $m>4$ odd. If $\Sigma, \Sigma' \in \Theta_m$ and $T^*\Sigma$ is symplectomorphic to $T^*\Sigma'$, then $\lbrack \Sigma \rbrack = \pm \lbrack \Sigma'\rbrack \in \Theta_m / bP_{m+1}$.
\end{thm}

It will be practical to paraphrase the above theorem as follows.

\begin{cor}\label{lagrembedding}
	If $m>4$ is odd and $\Sigma \in \Theta_m \setminus bP_{m+1}$, then $\Sigma$ does not admit a Lagrangian embedding into $T^*S^m$.
\end{cor}

\begin{definition}\label{abolishexotic}
	By personal choice of the author, we depart from the classic terminology of \emph{exotic} manifolds. Instead, we will call a smooth manifold that is homeomorphic, but not diffeomorphic, to the standard sphere an \emph{AD} sphere (AD stands for Alternative Differentiable structure).
	Correspondingly, a smooth manifold that is homeomorphic, but not diffeomorphic, to the standard $\CP$ will be called an \emph{AD} projective space. Finally, a smooth manifold that is homotopy equivalent, but not homeomorphic, to the standard projective space will be called an \emph{AT} projective space (where AT stands for Alternative Topological structure).
\end{definition}

\subsubsection{Results}
The results of this section hinge on the existence of homotopy projective spaces that are obtained as the reduced space of a circle action on an AD sphere. It is not always possible to relate
an $n$-dimensional AD/AT projective space to a $(2n+1)$-dimensional AD sphere in this way. Below, we start by exploring a few facts about AD/AT projective spaces, after which we can discuss three interesting examples where the desired phenomenon is observed (the spaces of Theorems \ref{cor1}, \ref{cor2}).

\begin{definition}[{{\cite{kawakuboinertia}}}]
	The inertia group $I(M)$ of an oriented closed smooth manifold $M$ is the subgroup of $\Theta_m$ consisting of homotopy spheres $S \in \Theta_m$ such that the connected sum $M\# S$ is in the same diffeomorphism class as $M$.
\end{definition}

If $I(\CP)=0$ and $\Theta_{2n}$ is non-trivial, one can build an AD projective space as follows. Given an AD sphere $\Sigma \in \Theta_{2n}$ the connected sum $\CP \# \Sigma$ (a $0$-dimensional surgery) is another manifold homeomorphic to $\CP$ but not diffeomorphic to it. For $n \geq 8$, there are examples for which the inertia group $I(\CP)$ is non-trivial (see \cite{kawakuboinertia}); in those cases the smooth structure of the resulting manifold is not automatically distinct from the standard smooth structure on $\CP$. In dimension four, we know:
\begin{thm}[{{\cite{kasilingamprojective}}}]\label{cp4thm}
	There are two possible distinct smooth structures on a manifold homeomorphic to $\C\PP^4$: the standard $\C\PP^4$-structure, and the one on $\C\PP^4\#\Sigma^8$, where $\Sigma^8 \in \Theta_8$ is the unique AD $8$-sphere.
\end{thm}

In contrast, it is known that there is an abundance of AT projective spaces; for even integers $n\geq 4$, there are infinitely many AT projective spaces, distinguished by the first Pontryagin class $p_1 \in H^4(\CP; \Z)$ (\cite{hsiang}).

Is there a way to associate an AD sphere to an AD/AT projective space? Given an AD/AT projective space $K$, the unit bundle of the line bundle $\mathcal{L} \to K$ satisfying $c_1(\mathcal{L})=\alpha_K$ (where $\alpha_K \in H^2(K;\Z)$ is the cohomology generator) could still be diffeomorphic to a standard sphere. Note that in the special case where the projective space is a surgery of the form $K=\CP\#\Sigma$, for an AD sphere $\Sigma \in \Theta_{2n}$, the $(2n+1)$-sphere obtained as the unit bundle of $\mathcal{L}\to K$ is given by $stab(\Sigma) \in \Theta_{2n+1}$ (where $stab$ is the map constructed in Section \ref{aptwist}, see Remark \ref{rmkstab}).

On the other hand, one could examine $S^1$-quotients of AD spheres $\widetilde{S} \in \Theta_{2n+1}$. A priori this is not always a successful strategy, as not all homotopy spheres admit a smooth free circle action. But if such an action exists, then the quotient $P:=\widetilde{S}/S^1$ resulting from it is an AD or AT projective space. Namely, this reduced space is necessarily homotopy equivalent to a projective space (\cite{hsiang}), but it is at least not diffeomorphic to the standard $\CP$ (since circle bundles over $P$ are classified by elements of $H^2(P;\Z)$, and if $P$ was the standard projective space, then the total space of the line bundle would have to be a standard sphere).

\begin{thm}[{{\cite[Sections 2-3]{jamesfree}}}]\label{9sphere}
	There is a homotopy $9$-sphere $\widetilde{S}$ such that \begin{enumerate}[label=(\roman*)]
		\item$\widetilde{S} \notin bP_{10} \cong \Zmod$.
		\item $\widetilde{S}$ admits a free action of $S^1$.
		\item The quotient $P:=\widetilde{S} /S^1$ is not homeomorphic to $\C\PP^4$.
		\item $P$ and the standard $\C\PP^4$ have the same tangent bundles.
	\end{enumerate}
\end{thm}
\begin{rmk}\label{rmkaction}
	In \cite[Section 3]{jamesfree}, James notes that there is another $S^1$-action on $\widetilde{S}$ with quotient space $P\#\Sigma^8$. The latter is an AT projective space that is not diffeomorphic to $P$. 
\end{rmk}

We now have enough material to state and prove the results of this section.

\begin{thm}\label{cor1}
	There is a manifold $P$ homotopy equivalent to $\C \PP^4$ and with the same first Pontryagin class such that neither $P$ nor $P \# \Sigma^8$ admit an exact Lagrangian embedding into $T^*\C\PP^4$.
\end{thm}

\proof
	Consider the homotopy $9$-sphere $\widetilde{S}$ admitting a free $S^1$-action of Theorem \ref{9sphere}. The quotient $P=\widetilde{S} /S^1$ is homotopy equivalent to $\C\PP^4$, but by Theorem \ref{9sphere} (iii), it is not homeomorphic to it. The first Pontryagin classes of $P$ and $\C\PP^4$ coincide by Theorem \ref{9sphere} (iv).
	Assume there is a Lagrangian embedding $P \hookrightarrow T^*\C\PP^4$. The Hopf correspondence (see Section \ref{localmodelcpn}) lifts $P$ to $\widetilde{S}$, giving an exact Lagrangian embedding $\widetilde{S} \hookrightarrow T^*S^9$. However, by Theorem \ref{9sphere}, $\widetilde{S} \in \Theta_9\setminus bP_{10}$, so the existence of the Lagrangian embedding contradicts Corollary \ref{lagrembedding}. \\
	The same argument applies to prove that $P \#\Sigma^8$ does not embed as Lagrangian into $T^*\C\PP^4$. Namely, the Hopf correspondence would, in that case too, lift (via the $S^1$-action of Remark \ref{rmkaction}) $P\#\Sigma^8$ to $\widetilde{S}$ (\cite[Section 3]{jamesfree}).
\qed

\begin{rmk}
	Our techniques do not allow to prove whether the AD projective space $\C\PP^4 \#\Sigma^8$ of Theorem \ref{cp4thm} does admit a Lagrangian embedding into $T^*\C\PP^4$ or not.
\end{rmk}

\begin{thm}\label{cor2}
	Let $\Sigma^{14}\in \Theta_{14}$ be the unique AD $14$-sphere. Then $\C\PP^7\#\Sigma^{14}$ does not admit an exact Lagrangian embedding into $T^*\C\PP^{7}$.
\end{thm}
\proof
	First note that $\Theta_{14} \cong \Zmod$ and $bP_{15}=0$ (\cite{kmspheres}), so there is a unique AD $14$-sphere, denoted by $\Sigma^{14}$ as in the statement.
	By \cite[Theorem 4.6]{bredonpairing} there is an AD sphere $\Sigma^{15} \in \Theta_{15} \setminus bP_{16}$ admitting a free $S^1$-action, with quotient $P:=\C\PP^7 \# \Sigma^{14}$.
	If $P$ admitted a Lagrangian embedding in $T^* \C \PP^7$, the Hopf correspondence would yield a Lagrangian embedding $\Sigma^{15}\hookrightarrow T^*S^{15}$. But $\Sigma^{15} \notin bP_{16}$, which would contradict Corollary \ref{lagrembedding} (for the same reasons as in the proof of Theorem \ref{cor1}).
\qed

\subsection{Framing of projective twists}\label{aptwist}
The background material that we use to examine the question of framing of projective twists is based on \cite{dretwist}, in which Dimitroglou Rizell and Evans proved that the Hamiltonian isotopy class of a Dehn twist does in general depend on a choice of framing. \\
Let $(M, \omega)$ be a symplectic manifold. Given a framing of a Lagrangian sphere $L\subset M$, i.e a diffeomorphism $S^n \lra L$ (see Section \ref{deftwist}), the precomposition with an element $F\in \Diff(S^m)$ yields another framing.

Consider the symplectomorphism $F^*\colon T^*S^m \to T^*S^m$ induced by the lift of $F$ to the cotangent bundle $T^*S^m$. The standard model twist $\tau_{S^m}^{loc} \in \Symp_{ct}(T^*S^m)$ can be replaced by $F^* \circ  \tau_{S^m}^{loc} \circ (F^{-1})^*$, and the latter can be implanted in a Weinstein neighbourhood as in Definition \ref{weinsteinextend} to produce a new element in $\Symp_{ct}(M)$. To study framings of twists, we can then restrict to these parametrisations of the standard model twist $\tau_{S^m}:=\tau_{S^m}^{loc} \in \Symp_{ct}(T^*S^m)$.

A core fact for the study of parametrisations of twists is the isomorphism $\pi_0(\Diff^+(S^m)) \cong \Theta_{m+1}$ (\cite{kmspheres}, \cite{cerf}). In particular, given a non-trivial diffeomorphism $F\in \Diff^+(S^m)$, there is a $(m+1)$-dimensional AD sphere constructed as follows.

\begin{definition}
	Let $F \in \Diff(S^m)$ not isotopic to the identity. Then $\Sigma_F:=D^{m+1} \cup_{F}D^{m+1} \in \Sigma_{m+1}$ is an $(m+1)$-dimensional homotopy sphere obtained by gluing two $(m+1)$-discs along their boundary $S^{m}$ twisted by $F$. In the notation of \cite[Definition 1.4]{dretwist} (which is more apt to visualise the Lagrangian suspension we utilise in Section \ref{resultsframing}), this is equivalent to
	$$\Sigma_F:=(D^{m+1}\times S^0) \cup_{\Phi} S^m\times [0,1]$$ glued along $S^m\times S^0$ via $\Phi\colon S^m \times S^0 \to S^m \times S^0$, $\Phi(x,y)=(F(x),y)$.
\end{definition}

Also recall that there is an isomorphism $\pi_0(\Diff^+(S^{m})) \cong \pi_0(\Diff^+_{ct}(D^{m}))$ induced by a map $\Diff_{ct}^+(D^{m})\to \Diff^+(S^{m})$ which extends all elements of $\Diff^+_{ct}(D^n)$ over a capping disc.

Dimitroglou Rizell and Evans proved the existence of Dehn twists, whose Hamiltonian isotopy class depends on the choice of framing.

\begin{definition}[{{\cite[Definition 1.1]{dretwist}}}]
	Fix a cotangent fibre $\Lambda \subset T^*S^{m}$ and let $\L_{m} \subset \Theta_{m}$ be the subset of homotopy spheres which admit a Lagrangian embedding into $T^*S^{m}$ with the additional requirement that the embedding intersects $\Lambda$ transversely in exactly one point.
\end{definition}

\begin{thm}[{{\cite[Theorem A]{dretwist}}}]\label{dre}
	Let $F \in \Diff^+(S^{m})$ be such that $\Sigma_F\notin \L_{m+1}$. Then $\tau_{S^{m}}^{-1}\circ (F^* \circ \tau_{S^{m}}\circ (F^{-1})^*)$
	is not trivial in $\pi_{0}(\Symp_{ct}(T^*S^{m}))$.
\end{thm}

In the rest of the section, we analyse the analogous problem for reparametrisations $f\in \Diff(\CP)$ of projective twists. We prove that there exist $n \in \N$ such that the twist $\tau_f:=f^* \circ \tau_{\CP}\circ (f^{-1})^*$
is not isotopic to the standard projective twist in $\pi_{0}(\Symp_{ct}(T^*\CP))$, where $f^*\colon T^*\CP \to T^*\CP$ is the symplectomorphism induced by the lift of $f$ to the cotangent bundle. We will not directly use Theorem \ref{dre} but an intermediary result (Proposition \ref{propinclusion} below) that Dimitroglou Rizell and Evans proved (using \cite{abouzaidframed, aknearby, ekholmsmith}) to support their arguments.

\begin{prop}[{{\cite[Proposition 1.2]{dretwist}}}] \label{propinclusion}
	There is an inclusion $\L_{m} \subset bP_{m+1}$.
\end{prop}

\begin{rmk}
	There is a slight abuse of terminology in the entirety of the section. A \emph{framing} will be employed (as in the rest of the paper) in the non-standard sense à la Seidel to signify a smooth parametrisation of a sphere. The classical topological notion of framing (as a trivialisation of the normal bundle) is also needed in this section, and in order to avoid a conflict of nomenclature, we call the latter a \emph{normal framing}.
\end{rmk}

\subsubsection{Bredon's pairing}\label{bredonpairing}

We begin by introducing an essential component of the arguments of this section; a map
\begin{align}\label{stabmap}
stab\colon \Theta_{m}\longrightarrow \Theta_{m+1}
\end{align} obtained as a special case of a homomorphism $\Theta_m \otimes \pi_{\ell}(S) \lra \Theta_{m+\ell}$ studied in \cite{bredonpairing}.

Consider the linear action of $\Soo\simeq S^1$ on $S^{m+1}\subset \R^{m+2}$ via the representation $$\Soo \lra \mathrm{SO}(m+2), \ A\longmapsto \varphi(A)=
\begin{bmatrix}
A & & \\
&A & \\
&  & \ddots \\
\end{bmatrix}$$
with $1$ in the right hand bottom corner if $m$ is odd. This is the linear $S^1$-action on $S^{m+1}$ which is free if $m$ even (in which case it is the standard Hopf action), and whose fixed point set is $S^0$ if $m$ odd.

For $\Sigma \in \Theta_m$, Bredon's construction (\cite[Sections 1, 4]{bredonpairing}) yields a homotopy $(m+1)$-sphere as follows. Let $V\subset \Sigma$ be an open neighbourhood of a point $p\in \Sigma$, and $g\colon (V,p)\to (\R^m, 0)$ an orientation reversing diffeomorphism. Let $B:=g^{-1}(D^m)\subset V\subset \Sigma$, where $D^m\subset \R^m$ is the unit disc. 

Let $\CC\cong S^1\subset S^{m+1}$ be a principal orbit of the $\Soo$-action on $S^{m+1}$, equipped with a normal framing $\F\colon \CC \times \R^m \to S^{m+1}$. 

Define
\begin{align}\label{defstab}
stab(\Sigma):=S^{m+1}\setminus (\F(\CC\times D^m)) \ \cup \ \CC \times (\Sigma \setminus B)
\end{align}where the two pieces are glued along their boundaries, which can be identified via a diffeomorphism $\F(\CC \times (\R^m \setminus \{0 \})) \cong \CC \times (V \setminus \{ p\})$ as in \cite[p. 435]{bredonpairing}.

The normally framed orbit $(\CC, \F)$ represents an element $\gamma \in \pi_{m+1}(S^{m})\cong \pi_4(S^3) \cong \Zmod$ via the Thom--Pontryagin construction (see \cite[87]{milnortop}). With this identification in mind, the map $stab$ is derived from a pairing $\Theta_m \times \pi_{m+1}(S^m)\to \Theta_{m+1}, \ (\Sigma, \gamma) \mapsto stab(\Sigma)=\langle \Sigma, \gamma \rangle$ (see \cite[(1)]{bredonpairing}). The latter induces a homomorphism (\cite[(2)]{bredonpairing}) 
\begin{align}\label{homostab}
\Theta_m \otimes \pi_1(S) \lra \Theta_{m+1}.
\end{align}

To determine the class $\gamma$, we follow \cite[(4.1)]{bredonpairing} and find that $\gamma=\eta^{j}$, where $\eta \in \pi_1(S):=\pi_4(S^3)\cong \Zmod$ is the non-trivial element in the stable stem $\pi_1(S)$, and\begin{align*}
j =\left\{
\begin{array}{ll}
\frac{m}{2}	&   \text{ If } m \textbf{ even } \text{(i.e the action on } S^{m+1} \text{ is free)} \\
\frac{m-1}{2} & \text{ If } m \textbf{ odd } \text{ (i.e the action } S^{m+1} \text{ is not free)}.
\end{array}
\right.
\end{align*}	

Intuitively, if a normally framed Hopf circle in $S^3$ represents the class $\eta\in \pi_4(S^3)$, then $\gamma$ is determined by the number of times (mod 2) that this normal framing fits in the normal bundle to $\CC \subset S^{m+1}$.

For $m+1=2n+1$, and $m+1=2n+2$, we have $j=n$ and
\begin{align}\label{powereta}
\gamma = \eta^{n}=\left\{
\begin{array}{ll}
\eta&   \text{ If } n \textbf{ odd} \\
0 & \text{ If } n \textbf{ even}.
\end{array}
\right.
\end{align}

\begin{rmk}\label{rmkstab}
	For an even dimensional homotopy sphere $\Sigma \in \Theta_{2n}$, the image $stab(\Sigma)$ can also be described as follows (this remark is relevant for Section \ref{nonemb}). Consider the surgery $\CP \# \Sigma$ and the complex line bundle $\mathcal{L}\to \CP \# \Sigma$ associated to the generator of $H^2(\CP\#\Sigma; \Z)$. Then, $stab(\Sigma)$ is the homotopy sphere obtained as the unit circle bundle of $\mathcal{L}$. 
\end{rmk}

We now focus on the case $m+1=2n+2$.
\begin{lemma}\label{nontrivialcriterion}
	The map $\Theta_{2n+1} \lra \Theta_{2n+2}/bP_{2n+3}$ is non-trivial for $n=19, 23, 25, 29$.    
\end{lemma}
\begin{proof}
	There is a commuting diagram (see \cite[Corollary 2.2]{bredonpairing}) obtained from the exact sequence \eqref{kmexact}, \begin{equation}\label{thetacoversj}
	\begin{split}
	\xymatrixcolsep{5pc} \xymatrix{
		\Theta_{2n+1}   \ar[r]^{stab} \ar[d]^{\psi} &
		\Theta_{2n+2} \ar[d]^{\psi} &\\
		\coker(J_{2n+1}) \ar[r]^{(-)\cdot \eta^{n}}	& \coker(J_{2n+2})}.	
	\end{split}
	\end{equation}
	where $(-)\cdot \eta^{n} \colon \coker(J_{2n+1}) \to \coker(J_{2n+2})$, is a map descending from the multiplication $ \pi_{2n+1}(S) \times \pi_1(S)\rightarrow \pi_{2n+2}(S)$ with the class $\eta \in \pi_1(S) \cong \Zmod$, which is well defined since for $\ell+1 <m$, $\im(J_m) \cdot \im(J_{\ell}) \subseteq \im(J_{m+\ell})$, the image of the $J$-homomorphism is preserved under multiplication with elements of the stable stems.

	By \eqref{powereta}, we know that a necessary requirement for the map $stab$ to be non-trivial is to have $n=2k+1$ for some $k \in \N$ so that $\eta^n=\eta$ is non-trivial. In that case, we get
	
	\begin{equation}
	\begin{split}
	\xymatrixcolsep{5pc} \xymatrix{
		\Theta_{4k+3}   \ar[r]^{stab} \ar[d]^{\psi} &
		\Theta_{4(k+1)} \ar[d]^{\psi} &\\
		\coker(J_{4k+3}) \ar[r]^{(-)\cdot \eta}	& \coker(J_{4(k+1)})}.	
	\end{split}
	\end{equation}
	
	The vertical maps are both surjective since $\psi$ is always surjective in odd dimensions, and when $m\equiv 0\mod 4$ (see \cite[Theorem 5.4]{levine}).
	
	The exact sequence \eqref{kmexact} implies that $\coker(J_{4k+4})\cong \Theta_{4k+4}/\ker(\psi)\cong \Theta_{4k+4}$, and the non-triviality of the composition $\psi \circ stab \colon \Theta_{4k+3} \to  \Theta_{4k+4}$ is equivalent to the non-triviality of the multiplication $ (-) \cdot \eta \colon \coker(J_{4k+3}) \to \coker(J_{4k+4})$.
	This amounts to looking for elements in the stable stems whose $\eta$-multiples are not in the image of $J$. As $\eta$ is of order two, this information can be found in the ``two-primary part'' of the stable stems, the subgroups obtained after quotienting all elements of odd order. These are tabulated in a diagram in \cite[p.385]{hatcherat}, where the elements of interest appear to be in degrees $2n+1\in \{ 39, 47, 51, 59\}$, which means that $n\in \{19, 23,25, 29 \}$.
\end{proof}

The rest of the section is dedicated to explain how to relate a parametrisation $f\in \Diff^+(\CP)$ of the standard projective twist to a parametrisation $F\in \Diff^+(S^{2n+1})$ of the standard Dehn twist.

\begin{lemma}\label{liftlemma0}
	Let $n$ be an odd integer and $f \in \Diff^+(\CP)$ an orientation preserving diffeomorphism. There exists a diffeomorphism $F \in \Diff^+(S^{2n+1})$ satisfying $h \circ F=f \circ h$, i.e $F$ is the lift of $f$ by the Hopf bundle map.
\end{lemma}
\proof Let $h\colon S^{2n+1}\to \CP$ the Hopf bundle map. A diffeomorphism $f\colon \CP \to \CP$ always admits a lift to a continuous function on the total space $F\colon S^{2n+1}\to S^{2n+1}$. Namely, the only possible obstructions would have to lie in the groups $H^{i+1}(S^{2n+1},\pi_i(S^1))$, which are trivial for any choice of $i$.

We now show $f$ lifts to a diffeomorphism. To do that, first note that the map induced on the second cohomology $\overline{f} \colon H^2(\CP)\rightarrow H^2(\CP)$ is the identity, since $f$ is orientation-preserving. 
The fact that $\overline{f} \colon H^2(\CP)\rightarrow H^2(\CP)$ is the identity implies that $f$ lifts to a fibrewise linear map $F: \mathcal{L}\rightarrow \mathcal{L}$ to the total space of the line bundle $\mathcal{L}$ satisfying $c_1(\mathcal{L})=\alpha \in H^2(\CP;\Z)$, that fits in a commuting diagram \begin{align}\label{commuting}
h \circ F=f \circ h.
\end{align}
If $F$ had non-trivial kernel, the commutativity \eqref{commuting} would require it to be contained in the fibres of $h$, but every such map would be forced to be the zero map, which is impossible. So $F$ is a diffeomorphism. \endproof

\begin{lemma}\label{stablift}
	Let $f\in \Diff^+(\CP)$ be a diffeomorphism supported in an open chart, i.e $f$ is induced by an element of $\Diff_{ct}^+(D^{2n})$, and let $\Sigma_f \in \Theta_{2n+1}$ be the homotopy $(2n+1)$-sphere associated to $f$.
	
	Let $F \in \Diff^+(S^{2n+1})$ be the $S^1$-equivariant lift of $f$ of Lemma \ref{liftlemma0}, and $\Sigma_F \in \Theta_{2n+2}$ the corresponding homotopy $(2n+2)$-sphere. Then $stab(\Sigma_f)=\Sigma_F$. 
\end{lemma}
\proof The lift $F\in \Diff^+(S^{2n+1})$ is supported in a tubular neighbourhood of a Hopf circle in $S^{2n+1}$. To build $\Sigma_F$, identify $S^{2n+1}$ with an equator in $S^{2n+2}$, and consider the above circle as a normally framed circle $\CC \subset S^{2n+2}$. That requires a choice of trivialisation of the normal bundle to a Hopf circle, a normal framing $\F\colon \CC \times \R^{2n+1} \to S^{2n+2}$ that defines the support of the gluing map for the construction of $\Sigma_F$: on $\F(\CC \times D^{2n+1})$, $F$ acts as $id \times f$.
By the same arguments as in the beginning of this section, the normal framing of this Hopf circle corresponds to the class $\eta^n \in \pi_1(S)$, so that $\Sigma_F=stab(\Sigma_f)$ by the construction \eqref{defstab}.
\qed

\subsubsection{Results}\label{resultsframing}
We first mention an auxiliary result from \cite{dretwist} we will need in the proof of Theorem \ref{framingthm}.

\begin{lemma}[{{\cite[Proposition 2.5]{dretwist}}}]\label{lemmaembedding}
	Consider $(T^*S^{2n+1}, d\lambda_{T^*S^{2n+1}})$ with the well-known structure of Lefschetz fibration $T^*S^{2n+1}\to \C$ with smooth fibre $(T^*S^{2n}, d\lambda_{T^*S^{2n}})$ and two singular fibres.
	Let $L\subset T^*S^{2n+1}$ be the standard Lagrangian embedding of the zero section.
	There is an open symplectic embedding \begin{align}
	e\colon T^*S^{2n+1}\times T^*[0,1] \lra T^*S^{2n+2}
	\end{align} such that \begin{itemize}
		\item $L\times [0,1]$ is sent to a subset of the zero section $S^{2n+2}\subset T^*S^{2n+2}$ (the matching sphere);
		\item The image of the embedding is disjoint from a particular cotangent fibre $\Lambda \subset T^*S^{2n+2}$ (a Lefschetz thimble).
	\end{itemize}
\end{lemma}

\begin{prop}\label{criterion}
	If the map $stab\colon \Theta_{2n+1} \rightarrow \Theta_{2n+2}$ is non-trivial and $n$ is odd, then the $\CP$-twist depends on a choice of framing.
\end{prop}

\proof 
Choose a framing $f \in \Diff^+(\CP)$, coming from an element of $\Diff^+_{ct}(D^{2n})$ extended by the identity on the projective space.
Let $F\in \Diff^+(S^{2n+1})$ be the $S^1$-equivariant lift of $f$ as in Lemma \ref{liftlemma0}, supported in a tubular neighbourhood of a Hopf circle $\F \colon \cong S^1 \times D^{2n} \subset S^{2n+1}$. 	Let $\Sigma_f \in\Theta_{2n+1}$ be the sphere associated to $f$, and $\Sigma_F\in \Theta_{2n+2}$ that associated to $F$.
By Lemma \ref{stablift}, $\Sigma_F=stab(\Sigma_f)=\langle  \Sigma_f, \eta^n \rangle \in \Theta_{2n+2}$. Since $n$ is odd, $\eta^n=\eta$ and $\Sigma_F=stab(\Sigma_f)\in \Theta_{2n+2}$ is non-trivial.

The map $f^*$ induced by $f$ on the cotangent bundle is not compactly supported, but can be used to define the compactly supported conjugation
\begin{equation}
\begin{split}
\tau_f:=f^* \circ \tau_{\CP}\circ (f^{-1})^* \colon T^*\CP \lra T^*\CP
\end{split}
\end{equation}
of the projective twist $\tau_{\CP}\in \Symp_{ct}(T^*\CP)$.

We show in Lemma \ref{twistnonstd} below that $\tau_f$ belongs to a Hamiltonian class distinct from that of the standard projective twist.
\qed

\begin{lemma}\label{twistnonstd}
	The twist $\tau_{f}$ is not isotopic to the standard twist $\tau_{\CP} \in \Symp_{ct}(T^*\CP)$.
\end{lemma}
\proof

Assume by contradiction that $\tau_{\CP}^{-1}\circ \tau_f$ is (Hamiltonian) isotopic to the identity in $\Symp_{ct}(T^*\CP)$. Let $(\phi_t)_{t\in [0,1]}$ be an isotopy connecting the two symplectomorphisms in $\Symp_{ct}(T^*\CP)$ such that 
there are $s'>s \in (0,1)$ with \begin{align}
\phi_t = \left\{
\begin{array}{ll}
\tau_{\CP}^{-1}\circ \tau_f	&   \text{ If } t\leq s \\
Id & \text{ If } t\geq s'.
\end{array}
\right.
\end{align}
Let $H\colon T^*\CP \times [0,1] \to \R$ be the generating Hamiltonian function. Define the Lagrangian embedding \begin{equation}\label{suspensioncpn}
\begin{split}
\psi'\colon K \times &[0,1] \lra T^*\CP \times T^*[0,1] \\
&(x,t) \longmapsto(\phi_t(x), t, -H(\phi_t(x),t)),
\end{split}
\end{equation}

where $K\subset T^*\CP$ is the standard Lagrangian embedding of the zero section. By construction, near the ends of the interval, $K\times [0,1]$ is preserved by \eqref{suspensioncpn} in the sense that $\psi' (K \times [0,s])\subset K \times [0,s]$ and $\psi'(K \times [s',1])\subset K\times [s',1]$.

On each fibre $T^*\CP$ of $T^*\CP \times T^*[0,1]$, apply the Hopf correspondence to lift the image of \eqref{suspensioncpn} to a Lagrangian embedding \begin{align}\label{embeddingsphere}
\Psi \colon L \times [0,1] \to T^*S^{2n+1}\times T^*[0,1]
\end{align}
(where $L\subset T^*S^{2n+1}$ is the standard Lagrangian embedding of the zero section) such that $L\times I$ is preserved by $\Psi$ near the ends of the interval.

By Lemma \ref{lemmaembedding} we can replace $e(L\times [0,1]) \subset T^*S^{2n+2}$ by the Lagrangian suspension $\Psi(L\times [0,1])$, so that the ends of $\Psi(L\times [0,1])$ are ``capped'' into a $(2n+2)$-dimensional sphere diffeomorphic to $\Sigma_F\in \Theta_{2n+2}$ (see \cite[3.3]{dretwist}) which intersects a cotangent fibre once transversely and is therefore contained in $\L_{2n+2}$. By Lemma \ref{propinclusion}, $\L_{2n+2} \subset bP_{2n+3}$ and since $bP_{2n+3}=0$ (this holds for all odd integers, see \cite{kmspheres}), $\Sigma_F$ has to be the standard sphere. However, as we have seen above, $\Sigma_F=\langle \eta^n, \Sigma_f \rangle \in \Theta_{2n+2}$ is non-trivial as $n$ is odd. This is a contradiction, which proves Proposition \ref{twistnonstd}; $\tau_{f}$ cannot be isotopic to the standard projective twist in $\Symp_{ct}(T^*\CP)$.
\qed

Combining Lemma \ref{nontrivialcriterion} with Proposition \ref{criterion} yields the result we want.

\begin{cor}\label{embcor1}
	The $\CP$-twist depends on the framing when $n=19,23,25,29$.
\end{cor} \qed

\begin{rmk}\label{infinitemany}
	The $\CP$-twist depends on the choice of framing for infinitely many dimensions $n$.
\end{rmk}
\proof 
	One way to obtain infinite families of nontrivial multiples of $\eta$ which are not contained in the image of $J$ is by detecting them in topological modular forms, denoted $tmf$ (we refer to \cite{tmf} for a survey on the subject). 
	There is a ``Hurewicz homomorphism'' $\pi_*(\mathbb{S})\to \pi_*(tmf)$ between the ring of stable homotopy groups of spheres and the homotopy ring of $tmf$, 
	and the two primary components of the ring of homotopy groups have a certain kind of periodicity of degree 192.
	Therefore, if we can identify an element in one of the homotopy groups $\pi_{4k+3}(tmf)$ that is also in the image of the Hurewicz homomorphism and arises as a product of $\eta$, we obtain a periodic family of elements to which the argument of Lemma \ref{nontrivialcriterion} applies. 
	
	A (partially conjectural) diagram depicting the two-primary components can be found in \cite{tmf} and it is helpful to first identify a potential candidate. Degree $39=4\cdot 9+3$ presents an element which has been confirmed to be the image of a non-trivial multiple of $\eta$ (see \cite[Corollary 11.2 ]{tmf2}, there the element in question is called $u$ and arises as image of a product of $\overline{\kappa}, \nu, \eta $ and $\kappa$; all of these are standard names of generators of stable homotopy groups stems).
	It follows that in every dimension $m\equiv 39\; (\bmod\; 192)$ there is an element for which the map $(-)\cdot \eta \colon \coker (J_m) \to \coker(J_{m+1})$ and hence $stab\colon \Theta_m \to \Theta_{m+1}$ are not trivial. Recall that $m=4k+3=2n+1$ so that by Proposition \ref{criterion}, the projective twist depends on the framing for $n\equiv 19\; (\bmod\; 96)$.	Further scrutiny of the literature would provide other such elements---e.g for $m=59$ $(n=29)$.
\qed

\begin{rmk}
	It is very likely that a version of Corollary \ref{infinitemany} holds for $\HP$-twists as well. Bredon computes (see \cite[p.446]{bredonpairing}) the class that would be associated to a framing of $S^3 \subset S^{4n+3}$, which is a power of $\nu \in \pi_3(S) = \lim\limits_{m}\pi_{m+3}(S^m)\cong \pi_8(S^5)\cong \Z_{24}$. Non-triviality results for the map $stab$ in this case would not only be depending on the parity of $n$, so a non-vanishing criterion would be harder to obtain. But such a criterion could then be combined with existence of smooth semi free actions of $S^3$ on homotopy $(4k+3)$-spheres explicitly computed in \cite[Theorem 4.4, 4.7]{bredonpairing} (also note that there are infinitely many inequivalent free $S^3$-actions on homotopy $S^{4k+3}$-spheres, by \cite[Theorem 3]{hsiang}). Then, the above strategy could be applied to obtain infinitely many dimensions in which the $\HP$-twist would depend on the framing.
	
\end{rmk}

\begin{cor}\label{notgenerated}
	In the above dimensions, $\Symp_{ct}(T^*\CP) \not \simeq\Z$.
\end{cor}
\proof If $\tau_{\CP} \in \pi_0(\Symp_{ct}(T^*\CP))$ is the standardly framed twist along the zero section, then we claim that $\Z\langle \tau_{\CP} \rangle \subsetneq \pi_0(\Symp_{ct}(T^*\CP))$. Let $f \in \Diff_{ct}^+(\CP)$ be a framing such that the projective twist $\tau_f \in \Symp_{ct}(T^*\CP)$ defined using $f$ is not isotopic to $\tau_{\CP}$, as in Corollary \ref{embcor1}. Then, $\tau_f^{-1}\circ\tau_{\CP}$ cannot be isotopic to any power $\tau^k_{\CP}$, for any $k \in \Z$. This is because $\tau_{\CP}$, viewed as a graded symplectomorphism, acts non-trivially on the grading of the zero section, viewed as a graded Lagrangian (see \cite[Lemma 5.7]{seidelgraded}), whereas $\tau_f^{-1}\circ\tau_{\CP}$ acts trivially on the grading (see also \cite[Remark 1.5]{dretwist}). \endproof

\end{document}